\documentclass[10pt,reqno]{amsart}
\usepackage{amssymb,mathrsfs,graphicx,extpfeil,amsmath}
\usepackage{ifthen,latexsym,float,colortbl}
\usepackage[margin=1in]{geometry}
\usepackage{caption}
\usepackage{sidecap}
\usepackage{hyperref}
\usepackage{rotating,dsfont,stackengine}
\usepackage{colortbl}
\definecolor{black}{rgb}{0.0, 0.0, 0.0}
\definecolor{red}{rgb}{1.0, 0.5, 0.5}
\provideboolean{shownotes} % define a boolean variable
\setboolean{shownotes}{true} % defined as ``true'' or ``false''
\newcommand{\margnote}[1]{
\ifthenelse{\boolean{shownotes}}%
{\marginpar{\raggedright\tiny\texttt{#1}}}%
{}%
}
\newcommand{\hole}[1]{
\ifthenelse{\boolean{shownotes}}%
{\begin{center} \fbox{ \rule {.25cm}{0cm} \rule[-.1cm]{0cm}{.4cm}
\parbox{.85\textwidth}{\begin{center} \texttt{#1}\end{center}} \rule
{.25cm}{0cm}}\end{center}} {} }

\graphicspath{{pics/}}
\title[Euler-Riesz equations]{The global Cauchy problem for the Euler--Riesz equations}

\author[Choi]{Young-Pil Choi}
\address[Young-Pil Choi]{\newline Department of Mathematics \newline
Yonsei University, 50 Yonsei-Ro, Seodaemun-Gu, Seoul 03722, Republic of Korea}
\email{ypchoi@yonsei.ac.kr}

\author[Jung]{Jinwook Jung}
\address[Jinwook Jung]{\newline Department of Mathematics and Institute of Pure and Applied Mathematics \newline Jeonbuk National University, 567 Baekje-daero, Deokjin-gu, Jeonju-si, Jeollabuk-do 54896,  Republic of Korea}
\email{2jwook12@gmail.com}

\author[Lee]{Yoonjung Lee}
\address[Yoonjung Lee]{\newline Department of Mathematics \newline
Yonsei University, 50 Yonsei-Ro, Seodaemun-Gu, Seoul 03722, Republic of Korea}
\email{yjglee@yonsei.ac.kr}

\numberwithin{equation}{section}

\newtheorem{theorem}{Theorem}[section]
\newtheorem{lemma}{Lemma}[section]

\newtheorem{proposition}{Proposition}[section]
\newtheorem{remark}{Remark}[section]

\newcommand{\R}{\mathbb R}

\newcommand{\N}{\mathbb N}

\newcommand{\ls}{\lesssim}

\newcommand{\mc}{\mathcal C}

\newcommand{\bq}{\begin{equation}}
\newcommand{\eq}{\end{equation}}
\newcommand{\e}{\varepsilon}
\newcommand{\lt}{\left}
\newcommand{\rt}{\right}

\newcommand{\pa}{\partial}

\newcommand{\w}{\widetilde}
\newcommand{\sfI}{\mathsf{I}}
\newcommand{\sfJ}{\mathsf{J}}
\newcommand{\sfK}{\mathsf{K}}

\makeatletter
\def\moverlay{\mathpalette\mov@rlay}
\def\mov@rlay#1#2{\leavevmode\vtop{%
   \baselineskip\z@skip \lineskiplimit-\maxdimen
   \ialign{\hfil$\m@th#1##$\hfil\cr#2\crcr}}}
\newcommand{\charfusion}[3][\mathord]{
    #1{\ifx#1\mathop\vphantom{#2}\fi
        \mathpalette\mov@rlay{#2\cr#3}
      }
    \ifx#1\mathop\expandafter\displaylimits\fi}
\makeatother

\newcommand{\intr}{\int_{\R^d}}

\begin{document}
%%%%%%%%%%%%%%%%
\allowdisplaybreaks

%\date{\today}


\begin{abstract} We completely resolve the global Cauchy problem for the multi-dimensional Euler--Riesz equations, where the interaction forcing is given by $\nabla (-\Delta)^{-\sigma/2}\rho$ for some $\sigma \in (0,2)$. We construct the global-in-time unique solution to the Euler--Riesz system in a $H^s$ Sobolev space under a smallness assumption on the initial density and a {\it dispersive} spectral condition on the initial velocity. Moreover, we investigate the algebraic time decay of convergences for the constructed solutions. Our results cover the both attractive and repulsive cases as well as the whole regime $\sigma \in (0,2)$.

\end{abstract}

\maketitle \centerline{\date}

\tableofcontents

%%%%%%%%%%%%%%%%%%%%%%%%%%%%%%%%%%%%%%%%%%%%%%%%%%%%%%%%%%%%%%%%%%%%%%%%%%%%%%%%%5
%
%
%                        Section: Introduction 
%
%
%%%%%%%%%%%%%%%%%%%%%%%%%%%%%%%%%%%%%%%%%%%%%%%%%%%%%%%%%%%%%%%%%%%%%%%%%%%%%%%%%
\section{Introduction}
In the present work, we are interested in the global-in-time Cauchy problem for the Euler--Riesz equations:
\begin{align}\label{ER_main}
\begin{aligned}
&\pa_t \rho + \nabla \cdot (\rho u) = 0, \quad (x,t) \in \R^d\times \R_+, \cr
&\pa_t (\rho u) + \nabla \cdot (\rho u \otimes u) +c_P\nabla(\rho^{\gamma})=  \lambda \rho \nabla \Lambda^{-\sigma} \rho,
\end{aligned}
\end{align}
subject to initial data:
\[
(\rho(x,0), u(x,0)) = (\rho_0(x), u_0(x)), \quad x \in \R^d,
\]
where $\rho=\rho(x,t)$ and $u= u(x,t)$ denote the density and velocity  of the fluid at time $t$ and position $x$, respectively. 
Here, the Riesz operator $\Lambda^{s}$ is defined by $(-\Delta)^{s/2}$, $s \in \R$, and in the current work we concentrate on the case 
\[
\sigma \in (0,\min\{d,2\}).
\] 
Note that the case $\sigma = 2$ corresponds to the classical Euler--Poisson system. Here, $\gamma > 1$, $c_P$ is a nonnegative constant, and $\lambda$ is a constant. Depending on the sign of $\lambda$, the forcing term describes the attractive ($\lambda > 0$) and repulsive ($\lambda < 0$) interactions. 

The pressureless case $(c_P = 0)$ is introduced and rigorously derived in \cite{Ser20} as a mean-field limit of the second-order system corresponding to Newton's law with Riesz interactional potential. The rigorous derivation of the system \eqref{ER_main} with the isothermal pressure $(\gamma=1)$ from the Vlasov--Riesz--Fokker--Planck-type equation is studied in \cite{CJapp}. The relation between \eqref{ER_main} and the fractional porous medium equation is also investigated in \cite{CJ21} through a relaxation limit.

One of the main difficulties in analyzing the Euler system arises from the finite-time singularity formation. More precisely, no matter how smooth and small the initial data are, the equations may develop a singularity in finite time. In the repulsive case $(\lambda <0)$, the system \eqref{ER_main} has a dissipative effect from the interactions. However, even in that case, it is known that the system \eqref{ER_main} has a singularity formation. We refer to \cite{Che90, CJ22, Mak92, MP90, Per90} for the study of the finite-time breakdown of smoothness of solutions to the Euler, Euler--Poisson, or Euler--Riesz equations. Thus, for the global-in-time existence of solutions, it is natural to take into account solutions with lower regularities, for instance, measure-valued solutions, entropy weak solutions \cite{BG98, CDGS24, DiP85, FGSW16, LPS96, WRS96}. We refer to \cite{Che05} and references therein for a general survey of the Euler equations.

Some of the previous works on the global existence of regular solutions for the Euler--Poisson system, \eqref{ER_main} with $\sigma = 2$, can be summarized as follows. In the case of zero background state, the critical thresholds phenomena, that lead to global regularity or finite-time blow-up of regular solutions, in the one-dimensional Euler--Poisson system with and without pressure are observed in \cite{TW08} and \cite{ELT01}, respectively. In the presence of the constant background state, the global-in-time existence of regular solutions to the one-dimensional Euler-Poisson is studied in \cite{GHZ16}, and two and three-dimensional cases are dealt with in \cite{Guo98, IP13, LW14}.

For the Euler--Riesz system \eqref{ER_main}, the local-in-time well-posedness is established in a recent work \cite{CJ22}. Roughly speaking, the interaction force in \eqref{ER_main} can be regarded as $\nabla \Lambda^{-\sigma} \rho \sim |\nabla|^{1-\sigma} \rho$ in terms of regularity. This observation shows that establishing the well-posedness theory of \eqref{ER_main} is more challenging in the case $\sigma \in (0,1)$. In fact, this regularity issue is already addressed in \cite{Cha23, CJeapp, CJKpre} for the Vlasov--Riesz system. Since the case $\sigma= 1$ plays a critical value and has physical meaning as well, motivated from \cite{BDIV97, CIP94, Ill00}, we call that particular case the Manev potential and divide the regime $0 < \sigma < 2$ into two parts:
\begin{itemize}
\item[(i)] $1\leq \sigma < 2$ (sub-Manev potential case) and
\item[(ii)] $0 < \sigma <1$ (super-Manev potential case).
\end{itemize}

In the current work, we study the global-in-time existence and uniqueness of regular solutions to the system \eqref{ER_main}. In the presence of linear damping in velocity and constant background state, the global-in-time existence and uniqueness of regular solutions are obtained in \cite{CJ23} and \cite{CJLpre} in the repulsive and attractive interaction cases, respectively. However, in the undamped case, as mentioned above it is not expected to have the global existence of regular solutions for general initial data due to the lack of dissipations. Thus, inspired by \cite{BDDN21, DD22, GS97, G98}, we analyze sufficient conditions on initial data that ensure the global-in-time existence of regular solutions. More precisely, we consider a sufficiently small, regular initial density and a {\it dispersive} spectral condition on the initial velocity that causes fluid particles to spread. In this framework, the global-in-time Cauchy problem for the Euler system with sub-Manev interactions, i.e. \eqref{ER_main} with $c_P > 0$, $\lambda = \pm1$, and $\sigma \in [1,2)$ is discussed in a recent work \cite{DD21}. 

%%%%%%%%%%%%%%%%%%%%%%%%%%%%%%%%%%%%%%%%%%%%%%%%%%%%%%%%%%%%%%%%%%%%%%%%%%%%%%%%%5
%
%
%                        Section: Introduction 
%
%
%%%%%%%%%%%%%%%%%%%%%%%%%%%%%%%%%%%%%%%%%%%%%%%%%%%%%%%%%%%%%%%%%%%%%%%%%%%%%%%%%
\subsection{Main results}
Our main results are two-fold. First, we investigate the global well-posedness of the pressureless, repulsive Euler-Riesz equations, i.e. the system \eqref{ER_main} with $c_P = 0$ and $\lambda = -1$. Secondly, we investigate the global well-posedness of Euler--Riesz equations with pressure, i.e. the system \eqref{ER_main} with $c_P>0$.

To state our results, we consider the Burgers' equation:
\bq\label{burgers}
\begin{aligned}
&\pa_t v + v \cdot \nabla v = 0, \quad (x,t) \in \R^d\times \R_+, \cr
 &v(x,0) = v_0(x), \quad x \in \R^d.
\end{aligned}
\eq
For $s>\frac d2+1$, we define the space $E^s$ as
\[
E^s := \{ z \in \mc(\R^d; \R^d)  :  \nabla z \in L^\infty \mbox{ and } \nabla^2 z \in H^{s-2}\}.
\]
Here, $\nabla^k$ denotes the $k$-th order gradient.

We then recall from \cite{BDDN21} the following result for the Burgers' equation.
\begin{proposition}\label{prop_bur}
	Suppose that $v_0 \in E^s$ satisfies
	\bq\label{sp_cond}
	\mbox{there exists} \quad \e>0 \quad \mbox{such that} \quad \mbox{dist(\emph{Sp}(}Dv_0(x)), \R_-) \ge \e \quad \forall x \in \R^d,
	\eq
	where $\emph{Sp} A$ denotes the spectrum of a matrix $A$. 
	Then the following holds:
	\begin{itemize}
		\item[(i)] The system \eqref{burgers} has a global-in-time classical solution corresponding to the initial datum $v_0(x)$ such that
		\[
		\nabla^2 v \in \mc^j(\R_+; H^{s-2-j}) \quad \mbox{for } \ j=0,1.
		\]
		\item[(ii)] Moreover, there exists a function $K \in \mc_b(\R^d\times \R_+ ; \R^d \times \R^d)$ such that the relation
		\bq\label{dv}
		\nabla v(x,t) = \frac{\mathbb{I}_d}{1+t} + \frac{K(x,t)}{(1+t)^2}
		\eq
		holds for any $(x, t)\in \R^d\times\R_+$. Here, $\mathbb{I}_d$ stands for the $d \times d$ identity matrix and $K$ satisfies
		\bq\label{K_est}
		\|K(t)\|_{\dot{H}^{\ell}} \le k_\ell(1+t)^{\frac d2 -\ell} \quad \mbox{for any } \  0 < \ell \le s-1
		\eq
		for some constant $k_\ell>0$. In particular, we get 
\bq\label{divv}
\nabla \cdot v(x,t) = \frac{d}{1+t} + \frac{ \mbox{\rm Tr} K(x,t)}{(1+t)^2}.
\eq
Here, the symbol $\mbox{\rm Tr}$ denotes the trace of a matrix. 
		\item[(iii)] If $\nabla^2 v_0$ is bounded, then we have
		\[
		\|\nabla^2 v(t)\|_{L^\infty} \le C(1+t)^{-3}.
		\]
	\end{itemize} 
\end{proposition}
%%%%%%%%%%%%%%%%%%%%%%%%%%%%%%%%%%%%%%%%%%%%%%%%%%%%%%%%%%%%%%%%%%%%%%%%%%%%%%%%%5
%
%
%                        Section: Introduction 
%
%
%%%%%%%%%%%%%%%%%%%%%%%%%%%%%%%%%%%%%%%%%%%%%%%%%%%%%%%%%%%%%%%%%%%%%%%%%%%%%%%%%
\subsubsection{Pressureless case}\label{ssec:px}
In the absence of pressure, i.e. $c_P=0$, by introducing $n:= \sqrt{\rho}$, we reformulate the system \eqref{ER_main} as
\bq\label{ER_main2}
\begin{aligned}
&\pa_t n + u \cdot \nabla n + \frac12 n\nabla\cdot u = 0, \quad (x,t) \in \R^d \times \R_+,\\
&\pa_t u + u\cdot \nabla u = \lambda \nabla \Lambda^{-\sigma}( n^2)
\end{aligned}
\eq
with the initial data:
\[
(n(x,0), u(x,0)) = (n_0(x):= \sqrt{\rho_0(x)}, u_0(x)), \quad x \in \R^d.
\]

Note that in the pressureless and attractive Riesz interaction case, i.e.  $c_P=0$ and $\lambda=1$, we would deduce from the linear analysis that \eqref{ER_main2} is ill-posed when $\sigma<2$, see \cite[Section 2]{CJ22} for details. Therefore, in the absence of pressure, we only deal with the repulsive case. For simplicity of presentation, we take $\lambda = -1$. Then our main theorem in the pressureless and repulsive case is stated as follows.
\begin{theorem}\label{main_thm}
Let $d\ge 1$, $\sigma \in (0, \min\{d,2\})$, and $s>\max\lt\{2, \frac{d}{2} + 1\rt\}$. Suppose that the initial data $(n_0, u_0)$ satisfy
\begin{enumerate}
\item[(i)]
there exists $v_0\in E^{s+1}$ satisfying \eqref{sp_cond} and $u_0 - v_0$ is small in $H^{s+\frac \sigma2}$ and
\item[(ii)]
$n_0$ is small in $H^s$.
\end{enumerate}
Then there exists a unique global-in-time regular solution $(n,u-v)$ to \eqref{ER_main2} satisfying
\[
n \in \mc(\R_+; H^s) \quad \mbox{and} \quad u-v \in \mc(\R_+; H^{s+\frac \sigma2}).
\]
Furthermore, one also has the following temporal decay estimates: for any $t \ge 0$
\[
\|n(t)\|_{\dot{H}^\ell} \le C(1+t)^{\frac d2 - \ell-\frac\sigma2-\min\lt\{1, \frac{d-\sigma}{2}\rt\} } \quad \forall\, \ell \in [0,s]
\]
and
\[
\|(u-v)(t)\|_{\dot{H}^\ell} \le  C(1+t)^{\frac d2 - \ell -\min\lt\{1, \frac{d-\sigma}{2}\rt\}} \qquad \forall\, \ell \in \lt[0,s+\frac\sigma2\rt].
\]
Here $C > 0$ is independent of $t$.
\end{theorem}

Due to the strong singularity in the Riesz interaction potential, it is required to find some cancellations of top-order terms in the estimates. Clearly, it depends on the reformulation and regularity of solution spaces. We discover that the transformation $n = \sqrt\rho$ and the regularity of solutions $(n, u-v) \in H^s \times H^{s+\frac\sigma2}$ well capture the required cancellations of terms with highest-order derivatives. Here the sharp commutator estimates and various Sobolev inequalities are crucially employed. In addition to that, we analyze the stabilizing effect from the repulsive Riesz interaction potential to get a proper dissipation rate for $n$. Then this together with using a dispersive effect from the condition on the initial velocity $v_0 \in E^s$ provides the global regularity of solutions and the temporal decay estimates.

\begin{remark}\label{rem_decay_np}
Here, we provide some remarks on the temporal decay.
\begin{enumerate}
\item[(i)]
We can rewrite the temporal decay of $n$ in Theorem \ref{main_thm} as 
\bq\label{r1}
\|n(t)\|_{\dot{H}^\ell} \le C(1+t)^{ - \ell -\min \lt\{1-\frac{d-\sigma}{2},\, 0 \rt\} }
\eq
for $\ell \in [0, s]$. When $d>\sigma +2$, equivalently $1-\frac{d-\sigma}{2}<0$, \eqref{r1} does not give any decay estimate when $\ell \le \frac{d-\sigma}{2}-1$. However, once we use the conservation of mass:
\[ 
\|n(t)\|_{L^2}=\|n_0\|_{L^2} 
\]
(see Lemma \ref{low_est} below), and combine this with \eqref{r1} when $\ell =s$ to yield  
\bq\label{decay_imp}
\|n(t)\|_{\dot{H}^\ell} \le C(1+t)^{ - \ell -\min \lt\{\frac{\ell}{s}\lt(1-\frac{d-\sigma}{2}\rt),\, 0 \rt\} }
\eq
via a simple interpolation.
\item[(ii)]
	From Theorem \ref{main_thm} and \eqref{decay_imp}, we use the interpolation argument to obtain temporal decay estimates in $L^p$ for $2 \leq p \leq \infty$:
		\[
	\|(u-v)(t)\|_{L^p} \le  C(1+t)^{\frac dp -\min\lt\{1, \frac{d-\sigma}{2}\rt\}}
\]
 and 
\[
	\|n(t)\|_{L^p} \le C(1+t)^{ -(\frac d2 -\frac dp) - \min \lt\{\frac{1}{s} (\frac d2 -\frac dp)(1-\frac{d-\sigma}{2}),\, 0 \rt\} }
 \]
 for $t \geq 0$.
\end{enumerate}
\end{remark}

%%%%%%%%%%%%%%%%%%%%%%%%%%%%%%%%%%%%%%%%%%%%%%%%%%%%%%%%%%%%%%%%%%%%%%%%%%%%%%%%%5
%
%
%                        Section: Introduction 
%
%
%%%%%%%%%%%%%%%%%%%%%%%%%%%%%%%%%%%%%%%%%%%%%%%%%%%%%%%%%%%%%%%%%%%%%%%%%%%%%%%%%
\subsubsection{The case with pressure} We next consider the pressured case, i.e. $c_P > 0$. We set $c_P=1$ without loss of generality. In this case, we introduce the sound speed  $ n:= \kappa \rho^{\tilde{\gamma}}$ with
\[
\kappa= \frac{2 \sqrt{\gamma}}{\gamma-1} \quad \mbox{and} \quad \tilde{\gamma}=\frac{\gamma-1}{2},
\]
and reformulate system \eqref{ER_main} as the following symmetric form:
\bq\label{ER_main3}
\begin{aligned}
	&\pa_t n + u \cdot \nabla n +\tilde{\gamma} n \nabla \cdot u = 0, \quad (x,t) \in \R^d \times \R_+,\cr
	&\pa_t u + u\cdot \nabla u + \tilde{\gamma} n \nabla n  =  \lambda \nabla \Lambda^{-\sigma} n^{\frac1{\tilde{\gamma}}}
\end{aligned}
\eq
with initial data
\[
(n(x,0),u(x,0)) = ( n_0(x):=\kappa (\rho_0(x))^{\tilde{\gamma}},\, u_0(x)), \quad x\in \R^d.
\]
Before we state our main results for the system \eqref{ER_main} with $c_P = 1$, we intrdouce the following assumption $(\mathcal{A})$ on the initial data and parameters $s,\gamma$, and $\sigma$: %throughout Section \ref{sec:4}:
\begin{enumerate}
\item[($\mathcal{A}$1)]
There exists $v_0 \in E^{s+1}$ satisfying \eqref{sp_cond} and $u_0 - v_0$ is small in $H^s$.
\item[($\mathcal{A}$2)]
$n_0$ is small in $H^s$.
\item[($\mathcal{A}$3)]
If $\frac{2}{\gamma-1}$ is not an integer, assume that
\[
s<\frac{2}{\gamma-1} + \sigma-\frac12.
\]
\end{enumerate}

We first state our main theorem for system \eqref{ER_main} with the attractive or repulsive sub-Manev interactions ($1\leq \sigma <2$).
\begin{theorem}\label{main_thm2}
	Let $d\ge2$, $\sigma \in [1, 2)$, and $s>\frac d2+ 1$. Suppose that the initial data $(n_0, v_0)$ and system parameters $s,\gamma$, and $\sigma$ satisfy standing assumption $(\mathcal{A})$. Then there exists a unique global-in-time regular solution $(n,u-v)$ to \eqref{ER_main3} regardless of the sign of $\lambda$ satisfying
	\[
	n \in \mc(\R_+; H^s) \quad \mbox{and} \quad u-v \in \mc(\R_+; H^{s}).
	\]
	Furthermore, we have the following temporal decay estimates:
	\[
			\|n(t)\|_{\dot{H}^\ell} \le C(1+t)^{\frac d2 - \ell -\min\lt\{1, \frac{d(\gamma-1)}{2}\rt\}}
	\]
	and
	\[
			\|(u-v)(t)\|_{\dot{H}^\ell} \le  C(1+t)^{\frac d2 - \ell -\min\lt\{1,  \frac{d(\gamma-1)}{2} \rt\}}
	\]
	for all $\ell \in [0,s]$ and $t \geq 0$. Here $C > 0$ is independent of $t$.
\end{theorem}

\begin{remark}\label{rem2}
We provide some remarks on the result of Theorem \ref{main_thm2}.
\begin{enumerate}
\item[(i)] Since we are assuming $\sigma \ge 1$, we naturally exclude the case $d=1$ in Theorem \ref{main_thm2}.
\item[(ii)] The case $\sigma=2$ in Theorem \ref{main_thm2} can also be covered by the argument in this paper.	
\item[(iii)] For the case $\sigma\in [1, 2]$, authors in \cite{DD22} studied the global well-posedness result of \eqref{ER_main3} in the solution space $(n, u-v)\in \mathcal{C}(\mathbb{R}_{+}; \dot{H}^2 \cap \dot{H}^s \cap L^{\infty} \cap L^q ) \times \mathcal{C}(\mathbb{R}_{+};\dot{H}^2 \cap \dot{H}^s \cap L^{\infty})$ for some $q>1$ under the same conditions on $d,s$, and $\gamma$ of Theorem \ref{main_thm2}. Compared to that result, our result deals with the system \eqref{ER_main3} in a usual Sobolev space, and we are also able to provide the $\dot{H}^{\ell}$ decay estimates of $u-v$ for $0<\ell<2$.
\end{enumerate}
\end{remark}

Finally, in the super-Manev potential case ($0<\sigma <1$),  we separately present our main results in the repulsive (i.e. $\lambda=-1$) and attractive (i.e. $\lambda=1$) cases, respectively.
\begin{theorem}[repulsive case]\label{main_thm3}
	Let $d\ge1$, $\sigma \in (0, 1)$, and $s>\max\{2+\frac\sigma 2, \frac{d}{2} + 1\}$.  Suppose that the initial data $(n_0, v_0)$ and system parameters $s,\gamma$, and $\sigma$ satisfy standing assumption $(\mathcal{A})$. Moreover, we assume that $1<\gamma \leq 2$ and
	\[
\gamma < 	1+\frac{2(d-\sigma)}{d+\sigma}\quad \text{if } d=1,2. 
\]
Then there exists a unique global-in-time regular solution $(n,u-v)$ to \eqref{ER_main3} in the repulsive case $\lambda=-1$ satisfying
	\[
	n \in \mc(\R_+; H^s) \quad \mbox{and} \quad u-v \in \mc(\R_+;H^{s}).
	\]
	Furthermore, one also has the following temporal decay estimates: 
	\[
			\|n(t)\|_{\dot{H}^\ell} \le C(1+t)^{\frac d2 - \ell -\min\lt\{1,\, \frac{d(\gamma-1)}{2},\, \frac{d-\sigma}{2}\rt\}}
	\]
	and
	\[
			\|(u-v)(t)\|_{\dot{H}^\ell} \le  C(1+t)^{\frac d2 - \ell -\min\lt\{1,\, \frac{d(\gamma-1)}{2},\,\frac{d-\sigma}{2} \rt\}}
	\]
	for all $\ell \in [0,s]$ and $t \geq 0$. Here $C > 0$ is independent of $t$.
\end{theorem}

\begin{theorem}[attractive case]\label{main_thm4}
	Let $d\ge1$, $\sigma \in (0, 1)$ and $s>\max\{3-\frac{\sigma}{2}, \frac{d}{2} + 1\}$.  Suppose that the initial data $(n_0, v_0)$ and system parameters $s,\gamma$ and $\sigma$ satisfy standing assumption $(\mathcal{A})$.	We further assume that $1<\gamma \leq 2-\frac\sigma d$ and 
\[
\gamma < 1+ \frac{2}{\sigma+2} \quad \text {if } d\ge 3.
\]
	Then there exists a unique global-in-time regular solution $(n,u-v)$ to \eqref{ER_main3} in the attractive case $\lambda=1$ satisfying
	\[
	n \in \mc(\R_+; H^s) \quad \mbox{and} \quad u-v \in \mc(\R_+; H^{s}).
	\]
	Furthermore, one also has the following temporal decay estimates: 
	\[
			\|n(t)\|_{\dot{H}^\ell} \le C(1+t)^{\frac d2 - \ell -\min\lt\{1,\, \frac{d(\gamma-1)}{2}\rt\}}
	\]
	and
	\[
			\|(u-v)(t)\|_{\dot{H}^\ell} \le  C(1+t)^{\frac d2 - \ell -\min\lt\{1,\, \frac{d(\gamma-1)}{2} \rt\}}
	\]
	for all $\ell \in [0,s]$ and $t \geq 0$. Here $C > 0$ is independent of $t$.
\end{theorem}

As briefly mentioned above, in the super-Manev potential case $(0 < \sigma <1)$, we lose some differentiability of the interaction force $\nabla \Lambda^{-\sigma} \rho$. Thus, the classical energy estimates as in the sub-Manev case would not be directly applicable to that case. In the repulsive case, we utilize the arguments developed in Section \ref{ssec:px} for the pressureless case; we combine $\dot{H}^s$-estimates of $(n,u-v)$ with the weighted $\dot{H}^{s-\frac\sigma2}$-estimates of $n$ to cancel out the repulsive interaction force term arising from the $\dot{H}^s$-estimates of $u-v$.   On the other hand, the attractive case is more delicate, and it is essential to use a dissipation effect from the pressure appropriately to handle the attractive super-Manev interaction forces. Since the interaction potential is of a slightly lower order than the pressure, we employ the {\it regularity-tuning} iteration argument developed in \cite{CJLpre}, see Section \ref{ssec:attpo} for details. This enables us to handle the attractive super-Manev interaction term by the pressure term. Again, the proofs strongly rely on the energy method based on the sharp commutator estimates for the fractional Laplacian and various Sobolev embeddings.

\begin{remark}
We also provide some remarks concerning all the results in Theorems \ref{main_thm2}-\ref{main_thm4}.
\begin{enumerate}
\item[(i)] The different conditions on $s$ are used in each theorem.

\item[(ii)] The conditions on $\gamma$ in Theorems \ref{main_thm2}-\ref{main_thm4} cover the range $1<\gamma\leq  1+\frac2d$ when $d \ge 3$. 

\item[(iii)] If we impose an assumption $1<\gamma\leq 2 - \frac\sigma d$ for Theorem \ref{main_thm3} in the repulsive case, then all conditions of $\gamma$ are automatically satisfied when $d\ge1$, and the decay estimates are exactly the same as that of Theorem \ref{main_thm4} for the attractive case.

\item[(iv)] If $\frac{2}{\gamma-1}$ is not an integer, $n^{\frac{2}{\gamma-1}}$ is handled by Lemma \ref{tech_4} in the homogeneous Sobolev space, and thus the additional condition $s<\frac{2}{\gamma-1}+\sigma -\frac 12$ is required in Theorems \ref{main_thm2}-\ref{main_thm4}, otherwise it is controlled by \eqref{KP_ineq} below without any further assumptions. In this regard, no additional condition is required in Theorem \ref{main_thm}.

\item[(v)] To prove the uniqueness, we do not require any additional condition on $s$ when $\sigma \in [1,2)$. However, in the case $\sigma \in (0,1)$, we need the additional conditions $s>2+\frac\sigma2$ and $s>3-\frac\sigma2$ for repulsive ($\lambda=-1$) and attractive case ($\lambda=1$), respectively. 

\item[(vi)]
	By the classical interpolation argument, a global-in-time regular solution $(n,u-v)$ to \eqref{ER_main3} satisfies the following $L^p$ decay estimates:
	\begin{itemize}
		\item[(a)] For $1\leq \sigma <2$ and $\lambda=\pm 1$, we have
		\[
		\|n(t)\|_{L^p} \le C(1+t)^{\frac dp-\min\lt\{1, \frac{d(\gamma-1)}{2} \rt\}}\]
		and
		\[\|(u-v)(t)\|_{L^p} \le  C(1+t)^{\frac dp -\min\lt\{1, \frac{d(\gamma-1)}{2}\rt\}}
		\]
		if $\max \{d, \frac{2}{\gamma-1}\}\leq p\leq \infty$ under the same conditions of Theorem \ref{main_thm2}.
		\item[(b)] For $0<\sigma<1$ and $\lambda=-1$ (repulsive case), we obtain
		\[
		\|n(t)\|_{L^p} \le C(1+t)^{\frac dp-\min\lt\{1, \frac{d(\gamma-1)}{2},\, \frac{d-\sigma}{2}\rt\} } \]
		and 
		\[ \|(u-v)(t)\|_{L^p} \le  C(1+t)^{\frac dp -\min\lt\{1, \frac{d(\gamma-1)}{2},\, \frac{d-\sigma}{2}\rt\}}
		\]
		if $\max \{d, \frac{2}{\gamma-1},\, \frac{2d}{d-\sigma}\}\leq p\leq \infty$ under the same conditions of Theorem \ref{main_thm3}.
		\item[(c)] For $0< \sigma <1$ and $\lambda=1$ (attractive case), we duduce
		\[
		\|n(t)\|_{L^p} \le C(1+t)^{\frac dp-\min\lt\{1, \frac{d(\gamma-1)}{2} \rt\}}\]
		and
		\[
		 \|(u-v)(t)\|_{L^p} \le  C(1+t)^{\frac dp -\min\lt\{1, \frac{d(\gamma-1)}{2}\rt\}}
		\]
		if $\max \{d, \frac{2}{\gamma-1}\}\leq p\leq \infty$ under the same conditions of Theorem \ref{main_thm4}.
	\end{itemize}
Here $C > 0$ is independent of $t$. Note that similar estimates hold for any $p \in [2,\infty]$ but they might not be decay estimates depending on the choice of $p$.
\item[(vii)]

	The decay rate of $(n,u-v)$ observed in  Theorems \ref{main_thm2} and \ref{main_thm4} coincides with that of solutions to the compressible Euler equations, see \cite{GS97, G98}.
\end{enumerate}
\end{remark}

%%%%%%%%%%%%%%%%%%%%%%%%%%%%%%%%%%%%%%%%%%%%%%%%%%%%%%%%%%%%%%%%%%%%%%%%%%%%%%%%%5
%
%
%                        Section: Introduction 
%
%
%%%%%%%%%%%%%%%%%%%%%%%%%%%%%%%%%%%%%%%%%%%%%%%%%%%%%%%%%%%%%%%%%%%%%%%%%%%%%%%%%
\subsection{Organization of the paper}

The rest of this paper is organized as follows. In Section \ref{sec:2}, we list useful inequalities and provide a Gr\"onwall-type lemma. Sections \ref{sec:3} and \ref{sec:4} are devoted to establishing global-in-time well-posedness and temporal decay estimates of regular solutions to the Euler--Riesz system without and with pressure, respectively.

\vspace{0.4cm}

%%%%%%%%%%%%%%%%%%%%%%%%%%%%%%%%%%%%%%%%%%%%%%%%%%%%%%%%%%%%%%%%%%%%%%%%%%%%%%%%%
%
%
%   section2: Preliminaries
%
%
%%%%%%%%%%%%%%%%%%%%%%%%%%%%%%%%%%%%%%%%%%%%%%%%%%%%%%%%%%%%%%%%%%%%%%%%%%%%%%%%%

\section{Preliminaries}\label{sec:2}
\setcounter{equation}{0}
In this section, we present several technical lemmas that will be used throughout this paper. 
\begin{lemma}\label{tech}
%	The following inequalities hold;
	\begin{enumerate}
	\item[(i)] 
	Let $s>0$, $ r \in (1, \infty)$ and $p_1, p_2, q_1, q_2 \in (1, \infty] $. 
Then we have
\bq\label{KP_ineq}
\|\Lambda^s (fg)\|_{L^r} \lesssim  \|\Lambda^s f\|_{L^{p_1}}\|g\|_{L^{q_1}} + \|f\|_{L^{p_2}}\|\Lambda^s g\|_{L^{q_2}}
\eq
for $f\in \dot{W}^{s, p_1} \cap L^{p_2}$ and $g\in L^{q_1} \cap \dot{W}^{s, q_2}$, where 
%any $f, g \in \mathscr{S}(\R^d)$, where $\mathscr{S}(\R^d)$ denotes the Schwartz space, and 
$\frac 1r = \frac1{p_1} + \frac1{q_1} = \frac1{p_2} + \frac1{q_2}$.

		\item[(ii)] 
		 Let $s>0$. If $f\in \dot{W}^{1, \infty} \cap \dot{H}^s$ and $g\in L^{\infty} \cap \dot{H}^{s-1}$, then we obtain
		\bq\label{tech_1}
		\|[f, \Lambda^s]g\|_{L^2} \lesssim \|f \|_{\dot{H}^s}\|g\|_{L^\infty} + \|\nabla f\|_{L^\infty}\|g\|_{\dot{H}^{s-1}}.
		\eq
		
		\item[(iii)]
 Let $s>1$. If $f\in \dot{W}^{2, \infty} \cap \dot{H}^s $ and $g \in L^{\infty} \cap \dot{H}^{s-2}$, then we get
	\bq\label{tech_2}
	\| [f, \Lambda^s]g - s\nabla f \cdot \Lambda^{s-2}\nabla g\|_{L^2}\lesssim \|f\|_{\dot{H}^s}\|g\|_{L^\infty}  + \|\nabla^2 f\|_{L^\infty} \|g\|_{\dot{H}^{s-2}}.
	\eq
	
		\item[(iv)]
		 Let $0<\sigma<1$. If $\Lambda^{1-\sigma} f\in L^{\infty} \cap \dot{H}^{\frac{d}{2}}$ and $g\in \dot{W}^{\frac{\sigma}{2}, 2}$,
		 then we deduce
\bq\label{tech_5}
		 \begin{aligned}
		 	\| [\nabla \Lambda^{-\frac \sigma 2}, f ]g \|_{L^2}
		 	&\lesssim \|\Lambda^{1-\sigma} f\|_{L^\infty} \|\Lambda^{\frac\sigma2} g\|_{L^2}+ \|f\|_{\dot{H}^{\frac{d}{2}+1-\sigma}} \|\Lambda^{\frac{\sigma}{2}}g\|_{L^{2}}.
		 \end{aligned}
\eq
		\item[(v)]
Let $1\leq s<2$, then the following inequality holds 
\bq\label{moser_2term}
\begin{aligned}
&\| [\Lambda^s, f ]g \|_{L^2} 
\lesssim \| \Lambda^{1} f \|_{L^{\infty}}\| \Lambda^{s-1} g \|_{L^2}  +\|\nabla f\|_{L^{\infty}}\|\Lambda^{s-1} g\|_{L^2} +\|g\Lambda^s f\|_{L^2}\\
\end{aligned}
\eq
for $\Lambda f , \nabla f\in L^{\infty}$, and $g\Lambda^s f \in L^2$. 
Moreover, when $0<s<1$, we have  
\[\| [\Lambda^s, f ]g \|_{L^2} 
\lesssim \| \Lambda^{s_1} f \|_{L^{\infty}}\| \Lambda^{s-s_1} g \|_{L^2}  +\| g\Lambda^s  f\|_{L^2}
\]
for any $s_1<s$.
	\end{enumerate}
Here $a \ls b$ represents that there is a positive constant such that $a \leq C b$.
\end{lemma}

\begin{proof}
For the Kato--Ponce inequality (i) we refer to \cite{GO14}.
The inequalities (ii) and (iii) are found in \cite{BDDN21, DD22}. 
The commutator estimates of the form (iv) and (v) follow from the fractional Leibniz inequality obtained in \cite{Li19}.
In fact, it follows from \cite{Li19} that
\[
\|\Lambda^s \partial_{x_i} ( fg) - f\Lambda^s\partial_{x_i}g -g\Lambda^s\partial_{x_i}f \|_{L^2} 
\lesssim \| \Lambda^{s_1} f \|_{L^{p_1}} \|\Lambda^{s_2}g\|_{L^{p_2}}
\]
holds for $-1<s<0$ and $s_1, s_2 \in [0,s+1]$ and $p_1, p_2 \in [2,\infty]$ with $s+1 = s_1+s_2$ and $\frac12 = \frac1{p_1} + \frac1{p_2}$. Applying this with $s_1=1-\sigma$, $s_2=\frac\sigma2$, $p_1=\infty$, and $p_2=2$, then we have
\[
\begin{aligned}
	 \| \nabla \Lambda^{-\frac \sigma 2}\lt( f g\rt) - f\nabla \Lambda^{-\frac\sigma 2}g   \|_{L^2}
	&= \| \nabla \Lambda^{-\frac \sigma 2}\lt( f g\rt) - f\nabla \Lambda^{-\frac\sigma 2}g  - g\nabla \Lambda^{-\frac\sigma2}f  \|_{L^2}+\|  g\nabla \Lambda^{-\frac\sigma2}f  \|_{L^2}\\
	&\lesssim \|\Lambda^{1-\sigma} f\|_{L^\infty} \|\Lambda^{\frac\sigma2} g\|_{L^2}+ \|\nabla\Lambda^{-\frac\sigma2}f\|_{L^{\frac{2d}{\sigma}}} \|g\|_{L^{\frac{1}{\frac12-\frac{\sigma}{2d}}}}\\
	&\lesssim \|\Lambda^{1-\sigma} f\|_{L^\infty} \|\Lambda^{\frac\sigma2} g\|_{L^2}+ \|f\|_{\dot{H}^{\frac{d}{2}+1-\sigma}} \|\Lambda^{\frac{\sigma}{2}}g\|_{L^{2}}
\end{aligned}
\]
thanks to the Hardy--Littlewood--Sobolev inequality.
Thus we obtain \eqref{tech_5}.

To show (v), we recall that
the following inequality holds for $s>0$ and $p \in (1,\infty)$:
\[
\| \Lambda^s (fg) - \sum_{|\alpha| \le s_1} \frac{1}{\alpha!}\pa^\alpha f \Lambda^{s,\alpha} g - \sum_{|\beta| < s_2}\frac{1}{\beta!}\pa^\beta g \Lambda^{s,\beta} f\|_{L^p} \le C\|\Lambda^{s_1} f\|_{L^\infty}\|\Lambda^{s_2} g\|_{L^p}
\]
whenever $s_1, s_2\ge 0$ with $s_1 + s_2 = s$. Here, $\alpha, \beta$ denote the multi-index and $\widehat{\Lambda^{s,\alpha} h} := i^{-|\alpha|}\pa_\xi^\alpha (|\xi|^s) \hat{h}(\xi)$.
In the case $1\leq s<2$, we use the above with taking $s_1=1$ to get
\[
\begin{aligned}
\| [\Lambda^s, f ]g \|_{L^2} &\lesssim 
\| \Lambda^s ( fg) -\sum_{|\alpha| \le 1} \frac{1}{\alpha!}\pa^\alpha f \Lambda^{s,\alpha} g -g \Lambda^{s}f  \|_{L^2} +\sum_{|\alpha| = 1} \| \pa^{\alpha} f \Lambda^{s, \alpha} g \|_{L^2} + \| g\Lambda^s f\|_{L^2} \\
&\lesssim 
\| \Lambda^{1} f \|_{L^{\infty}}\| \Lambda^{s-1} g \|_{L^2}  +\|\nabla f\|_{L^{\infty}}\|\Lambda^{s-1} g\|_{L^2} +\| g\Lambda^s f\|_{L^2} \\
&\lesssim \| \Lambda^{1} f \|_{L^{\infty}}\| \Lambda^{s-1} g \|_{L^2}  +\|\nabla f\|_{L^{\infty}}\|\Lambda^{s-1} g\|_{L^2} +\| g\Lambda^s f\|_{L^2} 
\end{aligned}
\]
from the fact that $\|\Lambda^{s, \alpha} g\|_{L^2} \lesssim \|\Lambda^{s-\alpha} g\|_{L^2}$.
On the other hand, in the case $0<s<1$, it is not difficult to check that a similar process yields
\[\begin{aligned}
	\| [\Lambda^s, f ]g \|_{L^2} 
	%&\lesssim 
	%\| \Lambda^s ( fg) -  f \Lambda^{s} g -g \Lambda^{s}f  \|_{L^2} + \| g\Lambda^s f\|_{L^2} 
	\lesssim \| \Lambda^{s_1} f \|_{L^{\infty}}\| \Lambda^{s-s_1} g \|_{L^2}  +\| g\Lambda^s f\|_{L^2}  
\end{aligned}\]
for $s_1\leq s$.
\end{proof}

For the uniqueness of solutions, we need to estimate the term $\|\Lambda v\|_{L^{\infty}}$ due to the use of the inequality \eqref{moser_2term}. Here, we would like point out that $\|\Lambda v\|_{L^{\infty}}$ cannot be controlled by $\|\nabla v\|_{L^{\infty}}$. To handle this difficulty, we provide the following lemma which, in particular, gives $\Lambda v \in L^{\infty}$.
\begin{lemma}\label{lambda_est}
Let $s>0$. If $f\in \dot{H}^{\frac{d}{2}+s+\varepsilon} \cap\dot{H}^{\frac{d}{2}+s-\varepsilon} $, then we have
\[
\|\Lambda^s f\|_{L^{\infty}} \le  C_{\varepsilon} \| f\|_{\dot{H}^{\frac{d}{2}+s+\varepsilon}}^{\frac{1}{2}}\| f\|_{\dot{H}^{\frac{d}{2}+s-\varepsilon}}^{\frac{1}{2}}.
\]
\end{lemma}
\begin{proof}
From \cite[Lemma~2.1]{JKL}, it follows that
\[
\|\Lambda^s f\|_{L^{\infty}} \le  C\| |\xi|^s\hat{f} \|_{L^1} \le C \||\xi|^{\frac{d}{2}+s+\varepsilon} \hat{f}\|_{L^2}^{\frac{1}{2}} \| |\xi|^{\frac{d}{2}+s-\varepsilon} \hat{f}\|_{L^2}^{\frac{1}{2}} = C \| f\|_{\dot{H}^{\frac{d}{2}+s+\varepsilon}}^{\frac{1}{2}}\| f\|_{\dot{H}^{\frac{d}{2}+s-\varepsilon}}^{\frac{1}{2}}
\]
by Young's inequality.
\end{proof}
Due to the structure of our reformulation, we also need to employ the Sobolev inequality for the composition of functions presented in the lemma below. For the proof, we refer to  \cite{BDDN21}. 
\begin{lemma}\label{tech_4}
Let $\alpha\ge 1$ and $0\leq s <\alpha +1/2$. If $f \in L^\infty \cap \dot{H}^s$, then we have
	\[
	\| |f|^{\alpha}\|_{\dot{H}^{s}}\lesssim \|f\|_{L^{\infty}}^{\alpha-1} \|f\|_{\dot{H}^{s}}.
	\]
\end{lemma}

Finally, we provide a Gr\"onwall-type lemma below.

\begin{lemma}\label{tech_6}
Let $Y \in AC(\R_+;\R_+)$ satisfies the following differential inequality
\bq\label{diff}
\frac{d}{dt}Y(t) + \frac{a}{1+t}Y(t) \le C^{\ast} \lt(Y^2(t) + \frac{Y(t)}{(1+t)^2} + c_P\sum_{i=1}^N \frac{Y^{b_i+1}(t)}{(1+t)^{1-c_i}}\rt) \quad \text{a.e  t}\in \mathbb{R}_{+} 
\eq
for some $C^{\ast}=C^*(N)>0$, any positive integer $N$ and $c_P=0,1$. 
If $a>1$, $b_i>0$, and $c_i<ab_i$ for $i=1, \dots, N$, 
then there exists $M=M(a, b_i, c_i, C^{\ast})$ such that if $Y(0)\le M$, then we have
\[
Y(t) \le \frac{2e^{\frac{C^{\ast}t}{1+t}}}{(1+t)^a}Y(0), \quad \forall t\ge 0.
\]
\end{lemma}
\begin{proof}
The case $N=1$ is already shown in \cite[Lemma~4.1]{BDDN21}.
Following a similar argument as in \cite{BDDN21}, we extend it to any positive integer $N$.
Let us set 
$$Z(t) :=(1+t)^a e^{-\frac{C^* t}{1+t}} Y(t)$$
and then \eqref{diff} implies
\[
\begin{aligned}
\frac{d}{dt} Z(t) 
&\leq
C^*(1+t)^a e^{-\frac{C^* t}{1+t}} Y^2(t) +C^* e^{-\frac{C^* t}{1+t}}c_P \sum_{i=1}^N (1+t)^{a+ c_i-1}Y^{b_i+1}(t)\\
&= C^*(1+t)^{-a} e^{\frac{C^*t}{1+t}} Z^2(t) + C^* c_P\sum_{i=1}^N (1+t)^{c_i-1-ab_i} e^{\frac{C^* b_i t}{1+t}} Z^{b_i+1}(t).
\end{aligned}
\]
To use the bootstrap argument, for any $T>0$ we assume that
\bq\label{diff3}
Z(t) \leq 2 Z_0 \quad \forall \, t\in [0, T],
\eq
where $Z_0 := Z(0)$. 
Then it yields that 
\[
\begin{aligned}
\frac{d}{dt} Z(t) 
	\leq
4C^*(1+t)^{-a} e^{C^*} Z_0^2 + Cc_P \sum_{i=1}^N (1+t)^{c_i-1-ab_i} e^{C^* b_i } (2Z_0)^{b_i+1}.
\end{aligned}
\]
Taking integration over $[0, T]$, we obtain
\bq\label{diff2}
\begin{aligned}
	 Z(t) &\leq	Z_0 +\frac{ 4C^* e^{C^*}}{a-1} Z_0^2(1-(1+t)^{1-a} ) +  c_P\sum_{i=1}^N \frac{2^{b_i +1}C^*e^{C^* b_i } }{ab_i -c_i} Z_0^{b_i+1} (1-(1+t)^{c_i-ab_i})\\
	&\leq \Big( 1+ \frac{ 4C^* e^{C^*}}{a-1} Z_0 + c_P \sum_{i=1}^N \frac{2^{b_i +1}C^*e^{C^* b_i } }{ab_i -c_i} Z_0^{b_i}\Big)Z_0 
\end{aligned}
\eq
for $a>1$ and $ab_i >c_i$.
When $Z_0=0$, it is trivial to get \eqref{diff3}. Otherwise, we take $Z_0 > 0$ sufficiently small so that 
\[
\frac{ 4C^* e^{C^*}}{a-1} Z_0 +  c_P\sum_{i=1}^N \frac{2^{b_i +1}C^*e^{C^* b_i } }{ab_i -c_i} Z_0^{b_i} <1.
\]
if $b_i>0$. Hence, \eqref{diff2} implies \eqref{diff3} for any $T>0$. This completes the proof.
\end{proof}

%%%%%%%%%%%%%%%%%%%%%%%%%%%%%%%%%%%%%%%%%%%%%%%%%%%%%%%%%%%%%%%%%%%%%%%%%%%%%%%%%
%
%
%   section 3: A priori estimates
%
%
%%%%%%%%%%%%%%%%%%%%%%%%%%%%%%%%%%%%%%%%%%%%%%%%%%%%%%%%%%%%%%%%%%%%%%%%%%%%%%%%%

\section{Pressureless Euler--Riesz system}\label{sec:3}
\setcounter{equation}{0}
In this section, we show a priori estimates for the system \eqref{ER_main2} with repulsive case $\lambda=-1$.  
First, for simplicity, we set $w := u-v$, then the system \eqref{ER_main2} can be rewritten as follows:
\bq\label{npER}
\begin{aligned}
&\pa_t n + w \cdot \nabla n + v \cdot \nabla n + \frac12  n \nabla \cdot w + \frac12  n \nabla \cdot v = 0, \quad (x,t) \in \R^d \times \R_+,\\
&\pa_t w + w \cdot \nabla w + v \cdot \nabla w + w \cdot \nabla v = -\nabla \Lambda^{-\sigma} n^2.
\end{aligned}
\eq

\subsection{A priori estimates}
We first provide a priori estimates for smooth solutions to \eqref{npER}.

\subsubsection{$H^s \times H^{s+\frac\sigma2}$ bound estimates}
\begin{lemma}\label{low_est}
Let $(n,w)$ be a smooth solution to \eqref{npER} decaying sufficiently fast at infinity. Then we have
\[
\frac{d}{dt} \|n\|_{L^2} =0,
\]
\[ \frac{d}{dt} \|w\|_{L^2} + \frac{1-\frac{d}{2}}{1+t} \|w\|_{L^2} \lesssim  \|\nabla w\|_{L^{\infty}} \|w\|_{L^2} +\frac{\|w\|_{L^2}}{(1+t)^2} + \|\nabla \Lambda^{-\sigma}(n^2)\|_{L^2},
\]
and for $s>1$, 
\[
\begin{aligned}
	&\frac{d}{dt} X_{s, \sigma}  + \frac{s+ \min\{ 0, \frac\sigma2 +1-\frac d2\}}{1+t} X_{s, \sigma} \\
	&\quad \lesssim \|\nabla w\|_{L^\infty}X_{s, \sigma}  + \|\nabla n\|_{L^\infty}\|w\|_{\dot{H}^s}  + \frac{X_{s, \sigma} }{(1+t)^2} + \frac{\|n\|_{L^\infty}}{(1+t)^{2+s-\frac d2}} + \frac{\|w\|_{L^\infty}}{(1+t)^{2+s+\frac\sigma2-\frac d2}} \\
	&\qquad + \frac{(\|n\|_{\dot{H}^{s-1}}+\|w\|_{\dot{H}^{s-1+\frac\sigma2}})  }{(1+t)^3}  + \frac{\|\nabla n\|_{L^\infty}}{(1+t)^{s+1-\frac d2}}  + 
	\frac{\|\nabla w\|_{L^\infty}}{(1+t)^{s+1+\frac\sigma2-\frac d2}},
\end{aligned}
\]
where
$X_{s,\sigma}^2 := \|w\|_{\dot{H}^{s+\frac \sigma2}}^2 + 4\|n\|_{\dot{H}^s}^2$.
\end{lemma}
\begin{proof}
	First, we deduce from  $\eqref{npER}_1$ that 
	\[
	\frac{d}{dt} \intr |n|^2 \, dx =0.
	\]
	To estimate $\|w\|_{L^2}$, we see that
	\[\begin{aligned}
		\frac12\frac{d}{dt} \intr |w|^2 \,dx &= -\intr w (w\cdot \nabla w)\,dx -\intr w(v \cdot \nabla w)\,dx -\intr w(w \cdot \nabla v)\,dx  -\intr  w \nabla \Lambda^{-\sigma}(n^2)\,dx\\
		&= \frac12\intr (\nabla \cdot w) |w|^2\,dx + \frac12 \intr (\nabla \cdot v) |w|^2\,dx\\
		&\quad -\frac{1}{1+t}\intr |w|^2\,dx -\frac{1}{(1+t)^2} \intr w(K\cdot w)\,dx - \intr w \nabla \Lambda^{-\sigma}(n^2)\,dx\\
		&\le \|\nabla w\|_{L^\infty}\|w\|_{L^2}^2  + \frac{\frac d2-1}{1+t} \|w\|_{L^2}^2 + \frac{C}{(1+t)^2}\|w\|_{L^2}^2 + \|w\|_{L^2}\|\nabla \Lambda^{-\sigma}(n^2)\|_{L^2}
	\end{aligned}\]
	which gives the desired result.\\

Next, we show the highest-order estimates for $(n,w)$.
By applying $\Lambda^s $ to \eqref{npER}$_1$, we estimate
\[ 
\begin{aligned}
	&\frac{1}{2} \frac{d}{dt} \intr |\Lambda^s n|^2\,dx +\frac{1}{2} \intr \Lambda^{s}n (n \nabla \cdot \Lambda^{s}w)\,dx \\
	&= -\intr \Lambda^{s} n\,  (w \cdot \nabla \Lambda^{s} n)\,dx -\intr \Lambda^{s} n (v \cdot \nabla \Lambda^{s} n) \,dx
	+s \sum_{1\leq k\leq d}\intr \Lambda^{s} n \, \pa_k v \cdot  \nabla \pa_k \Lambda^{s-2} n \,dx\\
	&\quad -\frac12 \intr \Lambda^{s} n \, \Lambda^{s}(n \nabla \cdot v )\,dx 
	+\intr \Lambda^{s} n \, [w, \Lambda^{s}] \cdot \nabla n\,dx
	+\frac{1}{2}\intr \Lambda^{s} n [n, \Lambda^s]\nabla \cdot w\,dx\\
	&\quad + \intr \Lambda^{s} n \,   \Big( [v, \Lambda^s]\cdot\nabla n -s \sum_{1\leq k\leq d}\pa_k v \cdot  \nabla \pa_k \Lambda^{s-2} n  \Big)\,dx\\
	 &=:\sum_{i=1}^7 I_i.
\end{aligned}
\]
First, direct computations using \eqref{divv} give
\[
I_1 \le C\|\nabla w\|_{L^\infty}\|n\|_{\dot{H}^s}^2 \quad \mbox{and} \quad 
I_2 \le \frac{\frac d2}{1+t}\|n\|_{\dot{H}^s}^2 + \frac{C}{(1+t)^2}\|n\|_{\dot{H}^s}^2.
\]
For $I_3$,  \eqref{dv} implies  
\[
\begin{aligned}
	\sum_{1\leq k \leq d} \pa_k v \cdot \Lambda^{-2} \nabla \pa_{k} z 
	&= -\frac{1}{1+t}  \Lambda^{-2} \Big(-\sum_{1\leq k\leq d} \pa_{kk}  z\Big)  
	+\sum_{1\leq k\leq d}\frac{K_k}{(1+t)^2}\cdot \Lambda^{-2}\nabla \pa_{k}z	\\
	&= -\frac{z}{1+t}+\sum_{1\leq k\leq d}\frac{K_k}{(1+t)^2}\cdot \Lambda^{-2}\nabla \pa_{k}z,	
\end{aligned}
\]
where $K_k$ denotes the $k$th row vector of $K$ ($k=1, \dots,  d$).
Then  we combine this with $z=\Lambda^{s}n$ to get   
\[ 
\begin{aligned}
I_3&= 
s \sum_{1\leq k\leq d}\intr \Lambda^{s} n \, \pa_k v \cdot  \nabla \pa_k \Lambda^{s-2} n\,dx \\
&=	-\frac{s}{1+t} \intr |\Lambda^{s}n|^2  \,dx
	+\frac{s}{(1+t)^2} \sum_{1\leq k\leq d} \intr \Lambda^{s} n\, K_k\cdot \Lambda^{s-2} \nabla \pa_{k} n\,dx\\
	&\le -\frac{s}{1+t}\|n\|_{\dot{H}^s}^2 + \frac{C}{(1+t)^2}\|n\|_{\dot{H}^s}^2,
\end{aligned}
\]
where $C>0$ only depends on $\|K\|_{L^{\infty}}$.\\

\noindent For $I_4$, we use \eqref{divv} to get
\[\begin{aligned}
I_4
&\leq -\frac{\frac{d}{2}}{1+t}\|n\|_{\dot{H}^s} + \frac{C}{(1+t)^2} \| n\|_{\dot{H}^{s}} (\|n\|_{L^{\infty}} \| K\|_{\dot{H}^{s}}
+\| K\|_{L^{\infty}}\|n\|_{\dot{H}^{s}})
\end{aligned}
\]
due to \eqref{KP_ineq}.
For the rest, we use \eqref{tech_1} and \eqref{tech_2} to obtain
\[
\begin{aligned}
	&I_5  \leq C \| n\|_{\dot{H}^{s}} (\|\nabla n\|_{L^{\infty}} \| w\|_{\dot{H}^{s}} + \|\nabla w\|_{L^{\infty}}\|n\|_{\dot{H}^s}),\\
	&I_6
	\leq C \| n\|_{\dot{H}^{s}} (\|\nabla w\|_{L^{\infty}} \| n\|_{\dot{H}^{s}} + \|\nabla n\|_{L^{\infty}}\|w\|_{\dot{H}^{s}}), \quad \mbox{and}\\
	&I_7 
	\leq C \|n\|_{\dot{H}^{s}} (\|v\|_{\dot{H}^{s}}\|\nabla n\|_{\infty} + \|\nabla ^2 v\|_{L^{\infty}}\|n\|_{\dot{H}^{s-1}}).
	\end{aligned}
\]
Here, Proposition \ref{prop_bur} deduces that
$$
 \|\nabla^2 v\|_{L^{\infty}} \lesssim (1+t)^{-3} \quad \mbox{and} \quad
\|v\|_{\dot{H}^{\ell}} \lesssim (1+t)^{-2} \|K\|_{\dot{H}^{\ell-1}} \lesssim (1+t)^{\frac{d}{2}-\ell-1}
$$
for $\ell>1$. Thus, combining these all together with \eqref{K_est} yields 
\bq\label{high_n_est}
\begin{aligned}
\frac12&\frac{d}{dt}\|n\|_{\dot{H}^s}^2 + \frac{s}{1+t}\|n\|_{\dot{H}^s}^2 +\frac12 \intr n \Lambda^s n \nabla \cdot \Lambda^s w\,dx\\
&\lesssim \|\nabla w\|_{L^\infty}\|n\|_{\dot{H}^s}^2 + \|\nabla n\|_{L^\infty}\|w\|_{\dot{H}^s}\|n\|_{\dot{H}^s}  + \frac{\|n\|_{\dot{H}^s}^2}{(1+t)^2} \\
&\quad + \lt( \frac{\|n\|_{L^\infty}}{(1+t)^{2+s-\frac d2}} + \frac{\|n\|_{\dot{H}^{s-1}} }{(1+t)^3} + \frac{\|\nabla n\|_{L^\infty}}{(1+t)^{s+1-\frac d2}}\rt)\|n\|_{\dot{H}^s} .
\end{aligned}
\eq
Next, we apply $\Lambda^{s+\frac\sigma2} $ to \eqref{npER} and use the fact that $\Lambda^{\frac{\sigma}{2}}$ is a self-adjoint operator to obtain

\[ 
\begin{aligned}
	&\frac{1}{2} \frac{d}{dt} \intr |\Lambda^{s+\frac{\sigma}{2}} w|^2\,dx + \intr \Lambda^{s}w \cdot \nabla \Lambda^{s} (n^2)\,dx \\
	&\quad = -\intr \Lambda^{s+\frac{\sigma}{2}} w\,  (w \cdot \nabla \Lambda^{s+\frac{\sigma}{2}} w)\,dx -\intr \Lambda^{s+\frac{\sigma}{2}} w (v \cdot \nabla \Lambda^{s+\frac{\sigma}{2}} w) \,dx\\
	&\qquad
	+(s+\frac{\sigma}{2}) \sum_{1\leq k\leq d}\intr \Lambda^{s+\frac{\sigma}{2}} w \, \pa_k v \cdot  \nabla \pa_k \Lambda^{s+\frac{\sigma}{2}-2} w\,dx \\
	&\qquad - \intr \Lambda^{s+\frac{\sigma}{2}} w \, \Lambda^{s+\frac{\sigma}{2}}(w \cdot \nabla v ) \,dx
	+\intr \Lambda^{s+\frac{\sigma}{2}} w \, [w, \Lambda^{s+\frac{\sigma}{2}}] \cdot \nabla w\,dx\\
	&\qquad + \intr \Lambda^{s+\frac{\sigma}{2}} w \,   \Big( [v, \Lambda^{s+\frac{\sigma}{2}}]\cdot\nabla w -(s+\frac{\sigma}{2}) \sum_{1\leq k\leq d}\intr \pa_k v \cdot  \nabla \pa_k \Lambda^{s+\frac{\sigma}{2}-2} w \Big)\,dx.
\end{aligned}
\]
We use a similar argument to the estimates for $n$ in the above and get
\[
\begin{aligned}
\frac12&\frac{d}{dt}\|w\|_{\dot{H}^{s+\frac\sigma2}}^2 + \frac{s+\frac\sigma2 +1-\frac d2}{1+t}\|w\|_{\dot{H}^{s+\frac\sigma2}}^2 + \intr \Lambda^s w \cdot \nabla \Lambda^s (n^2)\,dx\\
&\lesssim \|\nabla w\|_{L^\infty}\|w\|_{\dot{H}^{s+\frac\sigma2}}^2  +\frac{\|w\|_{\dot{H}^{s+\frac\sigma2}}^2}{(1+t)^2}\\
&\quad +\lt( \frac{\|w\|_{L^\infty}}{(1+t)^{2+s+\frac\sigma2-\frac d2}} 
+  \frac{\|w\|_{\dot{H}^{s-1+\frac\sigma2}}}{(1+t)^3} +\frac{\|\nabla w\|_{L^\infty}}{(1+t)^{s+1+\frac\sigma2 -\frac d2}} \rt)\|w\|_{\dot{H}^{s+\frac{\sigma}{2}}}.
\end{aligned}
\]
To eliminate the last term on the left-hand side of the above, we use the integration by part to observe that
\[\begin{aligned}
 \intr \Lambda^s w \cdot \nabla \Lambda^s (n^2)\,dx& = 2\intr \Lambda^s w \cdot \Lambda^s (n\nabla n)\,dx\\
&= 2\intr n\Lambda^s w \cdot \nabla \Lambda^s n\,dx -2\intr \Lambda^s w \cdot [n, \Lambda^s]\nabla n\,dx\\
&= -2\intr \Lambda^s n \Lambda^s w \cdot \nabla n\,dx - 2\intr n \Lambda^s n \nabla \cdot \Lambda^s w\,dx -2\intr \Lambda^s w \cdot [n, \Lambda^s]\nabla n\,dx,
\end{aligned}\]
and thus we obtain
\[
\lt| \intr \Lambda^s w \cdot \nabla \Lambda^s (n^2)\,dx +2\intr n \Lambda^s n \nabla \cdot \Lambda^s w\,dx\rt|  \le C\|\nabla n\|_{L^\infty}\|w\|_{\dot{H}^s}\|n\|_{\dot{H}^s},
\]
where we used \eqref{tech_1}. Thus it follows that
\bq\label{high_w_est}
\begin{aligned}
\frac12&\frac{d}{dt}\|w\|_{\dot{H}^{s+\frac\sigma2}}^2 + \frac{s+\frac\sigma2 +1-\frac d2}{1+t}\|w\|_{\dot{H}^{s+\frac\sigma2}}^2 -2\intr n \Lambda^s n \nabla \cdot \Lambda^s w\,dx\\
&\lesssim \|\nabla w\|_{L^\infty}\|w\|_{\dot{H}^{s+\frac\sigma2}}^2 +\|\nabla n\|_{L^\infty}\|w\|_{\dot{H}^s}\|n\|_{\dot{H}^s} + \frac{\|w\|_{\dot{H}^{s+\frac\sigma2}}^2}{(1+t)^2}\\
&\quad +\lt( \frac{\|w\|_{L^\infty}}{(1+t)^{2+s+\frac\sigma2-\frac d2}} + \frac{\|w\|_{\dot{H}^{s-1+\frac\sigma2}}}{(1+t)^3} +
\frac{\|\nabla w\|_{L^\infty}}{(1+t)^{s+1+ \frac\sigma2- \frac d2}}\rt)\|w\|_{\dot{H}^{s+\frac\sigma2}}.
\end{aligned}
\eq 
Finally, we set

\[
X_{s,\sigma}^2 := \|w\|_{\dot{H}^{s+\frac \sigma2}}^2 + 4\|n\|_{\dot{H}^s}^2,
\]
and from \eqref{high_n_est} and \eqref{high_w_est}, we have
\[
\begin{aligned}
	&\frac12\frac{d}{dt}X_{s,\sigma}^2  + \frac{s+\min\{0, \frac{\sigma}{2}+1-\frac{d}{2}\}}{1+t}X_{s,\sigma}^2  \\
	&\quad \lesssim \|\nabla w\|_{L^\infty}X_{s,\sigma}^2  + \|\nabla n\|_{L^\infty}\|w\|_{\dot{H}^s}X_{s,\sigma}+ \frac{X_{s,\sigma}^2 }{(1+t)^2}\\
	&\qquad  + \lt( \frac{\|n\|_{L^\infty}}{(1+t)^{2+s-\frac d2}} + \frac{\|n\|_{\dot{H}^{s-1}} }{(1+t)^3} + \frac{\|\nabla n\|_{L^\infty}}{(1+t)^{s+1-\frac d2}}\rt)X_{s,\sigma}\\
	& \qquad  + \lt( \frac{\|w\|_{L^\infty}}{(1+t)^{2+s-\frac d2+\frac\sigma2}}+ \frac{\|w\|_{\dot{H}^{s-1+\frac\sigma2}}}{(1+t)^3} +
	\frac{\|\nabla w\|_{L^\infty}}{(1+t)^{s+1- \frac d2+ \frac\sigma2}}\rt)X_{s,\sigma},
\end{aligned}
\]
and this gives the desired estimate.
\end{proof}

\subsubsection{Temporal decay estimates}
We next investigate a priori temporal decay estimates of solutions. For this, inspired by \cite{BDDN21, DD22}, 
we introduce
\[
	n_{\ell_1, p_1}(t) := (1+t)^{\ell_1 +\frac\sigma2-\frac{d}{p_1}-1}\|n(t)\|_{\dot{W}^{\ell_1,p_1}} \quad \mbox{and} \quad w_{\ell_2, p_2}(t):= (1+t)^{\ell_2 - \frac{d}{p_2}-1}\|w(t)\|_{\dot{W}^{\ell_2, p_2}}
\]
for $\ell_1 \in [0, s]$, $\ell_2 \in [0,s+\frac\sigma2]$, and $p_1, p_2 \in [2,\infty]$. For simplicity, we write $n_{p_1} :=n_{0, p_1}$ and $w_{p_2}:=w_{0, p_2}$.
We also define 
\[
Y_{s,\sigma}(t) :=  \lt(w_{s+\frac \sigma2, 2}^2(t)  +4n_{s, 2}^2(t)\rt)^{\frac12} \quad \mbox{and} \quad Z(t) :=  n_2(t) +w_2(t) +  Y_{s,\sigma}(t).
\]
%where $q$ is given as in \eqref{q}.

Then the following lemma is obtained by using the Gagliardo--Nirenberg interpolation inequality.
\begin{lemma}\label{bdds}
If $s> \frac d2+1$, we have the following relations: 
\[
\begin{aligned}
&n_{\infty}(t) \lesssim n_{s, 2}^{\frac{\frac{d}{2}}{s}}(t)\, n_{2}^{1-\frac{\frac{d}{2}}{s}}(t), \hspace{65pt}
w_{1, \infty}(t) \lesssim  w_{s+\frac{\sigma}{2}, 2}^{\frac{\frac d2 +1}{s+\frac\sigma2}}(t)\,w_2^{1-\frac{\frac d2+1}{s+\frac\sigma2}}(t),\\
&
n_{1, \infty}(t) \lesssim  n_{s, 2}^{\frac{\frac{d}{2}+1}{s}}(t)\, n_2^{1-\frac{\frac{d}{2}+1}{s}}(t), 
\hspace{40pt}
w_{s, 2}(t) \lesssim w_{s+\frac{\sigma}{2}, 2}^{\frac{s}{s+\frac\sigma2}}(t)\, w_2^{1-\frac{s}{s+\frac\sigma2}}(t), \\
&n_{s-1, 2}(t) \lesssim n_{s, 2}^{\frac{s-1}{s}}(t)\, n_{2}^{\frac{1}{s}}(t), \hspace{60pt}
w_{s-1+\frac{\sigma}{2}, 2}(t)\lesssim  w_{s+\frac\sigma2, 2}^{\frac{s-1}{s+\frac\sigma2}}(t)\,w_2^{1-\frac{s-1}{s+\frac\sigma2}}(t),\\
&n_{\frac{d}{\sigma-1}} (t) \lesssim n_{s,2}^{\frac{\frac d2 - (\sigma-1)}{s}}(t)\, n_2^{\frac{s-\frac d2 + \sigma-1}{s}}(t) \quad \mbox{when} \quad \sigma>1.
\end{aligned}
\]
In particular, each term on the left-hand side of the inequalities above is bounded by $C Z(t)$ with a constant $C>0$ independent of $t$. 
\end{lemma}

Then our main result of this subsection can be stated as follows.
\begin{proposition}\label{decay_temp1}
	Let the assumptions of Theorem \ref{main_thm} be satisfied.
	For $T>0$, suppose that $(n,w)$ is a smooth solution to \eqref{npER} on the time interval $[0,T]$ decaying sufficiently fast at infinity. Then we have
	\[
	\frac{d}{dt} Z(t)+\frac{C_{d, \sigma}}{1+t} Z(t)\lesssim  Z(t)^2 + \frac{Z(t)}{(1+t)^2},
	\]
where $C_{d, \sigma}=1+ \min\{ \frac{d-\sigma}{2},\, 1\} > 0$.
\end{proposition}
\begin{proof}
	First, it follows from Lemma \ref{low_est} that
	\[
	\frac{d}{dt}n_2 + \frac{1+ \frac{d-\sigma}{2}}{1+t}n_2 =0
	\]
	and
	\[
	\frac{d}{dt} w_2 + \frac{2}{1+t} w_2 \lesssim  w_{1, \infty} w_2 +\frac{w_2}{(1+t)^2} + (1+t)^{-\frac{d}{2}-1}\|\nabla \Lambda^{-\sigma}(n^2)\|_{L^2}.
	\]
Here, we divide it into two cases, $0<\sigma\leq 1$ and $1<\sigma <2$. 
In particular for $1<\sigma<2$, the condition on $\sigma$ naturally implies $d \ge 2$.

\vspace{.2cm}

\noindent $\bullet$ (Case A: $0<\sigma \leq 1$)
Here, we use the Gagliardo--Nirenberg interpolation inequality and \eqref{KP_ineq} to get
\[
\begin{aligned}
	\|\nabla \Lambda^{-\sigma}(n^2)\|_{L^2} 
	&\lesssim \| n^2\|_{\dot{H}^s}^{\frac{1-\sigma}{s}} \|n^2\|_{L^2}^{1-\frac{1-\sigma}{s}}\\
	&\lesssim \lt(\|n\|_{\dot{H}^s} \|n\|_{L^\infty}\rt)^{\frac{1-\sigma}{s}} \lt(\|n\|_{L^\infty}\|n\|_{L^2}\rt)^{1-\frac{1-\sigma}{s}}\\	
	& \lesssim \lt( \|n\|_{\dot{H}^s}^{1+\frac{d}{2s}} \|n\|_{L^2}^{1-\frac{d}{2s}} \rt)^{\frac{1-\sigma}{s}} \lt( 
	\|n\|_{\dot{H}^s}^{\frac{d}{2s}} 
	\|n\|_{L^2}^{2-\frac{d}{2s}}
	\rt)^{1-\frac{1-\sigma}{s}}\\
	& =C \|n\|_{\dot{H}^s} ^{\frac{\frac d2+(1-\sigma)}{s}} \|n\|_{L^2}^{2-\frac{\frac d2+(1-\sigma)}{s}}
\end{aligned}
\]
which implies 

\[
(1+t)^{-\frac{d}{2}-1}	\|\nabla \Lambda^{-\sigma} (n^2)\|_{L^2}\lesssim  n_{s, 2}^{\frac{\frac d2+(1-\sigma)}{s}} n_2^{2-\frac{\frac d2+(1-\sigma)}{s}}.
\]
Hence, we deduce from the above estimates that

\bq\label{low_est0}
\frac{d}{dt}(n_2+ w_2) + \frac{C_{d,\sigma}}{1+t} (n_2+ w_2) 
\lesssim \frac{Z(t)}{(1+t)^2} + Z(t)^2.
\eq

\noindent $\bullet$ (Case B: $1<\sigma<2$)
In this case, we first use Hardy--Littlewood--Sobolev inequality  to get 
\[
\begin{aligned}
\|\nabla \Lambda^{-\sigma} (n^2)\|_{L^2} \lesssim \|n^2\|_{L^{ \frac{1}{\frac12 + \frac{\sigma-1}{d}}} } \lesssim \|n\|_{L^2} \|n\|_{L^{\frac{d}{\sigma-1}}},
\end{aligned}
\]
and use Lemma \ref{bdds} to yield
\[\begin{aligned}
(1+t)^{-\frac d2-1}\|\nabla \Lambda^{-\sigma}(n^2)\|_{L^2} \lesssim n_2\, n_{\frac{d}{\sigma-1}} \lesssim  n_{s, 2}^{\frac{\frac d2-(\sigma-1)}{s}} n_2^{2-\frac{\frac d2-(\sigma-1)}{s}}.
\end{aligned}\]
Thus, we also have \eqref{low_est0} in this case.\\
 
\noindent Finally, the highest-order estimate in Lemma \ref{low_est} implies
\[
\begin{aligned}
	&\frac{d}{dt} Y_{s, \sigma}  + \frac{1+\min\{1,\,  \frac{d-\sigma}{2}\}}{1+t} Y_{s, \sigma} \\
	&\quad \lesssim w_{1, \infty} Y_{s, \sigma}  + n_{1, \infty} w_{s, 2}   + \frac{Y_{s, \sigma} +n_{\infty}+w_{\infty}+n_{s-1, 2}+ w_{s-1+\frac{\sigma}{2}, 2} +n_{1, \infty}+ w_{1, \infty} }{(1+t)^2}.  
\end{aligned}
\]
By Lemma \ref{bdds}, it immediately follows that
\[%\bq\label{31_5}
\frac{d}{dt}Y_{s,\sigma}+ \frac{C_{d, \sigma}}{1+t} Y_{s,\sigma} \lesssim Z(t)^2 + \frac{Z(t)}{(1+t)^2}.
\]%\eq
Therefore, we collect all the estimates to yield the desired result.
\end{proof}

\subsection{Proof of Theorem \ref{main_thm}}
Now, we present the proof of Theorem \ref{main_thm}. 

\vspace{.2cm}

\noindent $\bullet$ (Global-in-time existence and temporal decay) First, we combine Lemma \ref{low_est} with the arguments in \cite{DD21} to deduce the local-in-time existence of regular solutions to \eqref{npER}. For the global-in-time existence, we deduce from Proposition \ref{decay_temp1} that
\bq\label{p1}
\frac{d}{dt} Z(t)+\frac{C_{d, \sigma}}{1+t} Z(t)\leq C_0\lt( Z^2(t) + \frac{Z(t)}{(1+t)^2} \rt)
\eq
for some constant $C_0>0$ independent of $t$. 
Here, $C_{d, \sigma}=1+\min \{1, \frac{d-\sigma}{2}\}$.
If
\[
Z(0) \sim \|n_0\|_{H^s} + \|w_0\|_{H^{s+\frac\sigma2}}
\]
is sufficiently small,
then by applying Lemma \ref{tech_6} to \eqref{p1} with $a = C_{d, \sigma} > 0$, we deduce
\[
Z(t) \le \frac{2e^{\frac{C_0t}{(1+t)}}}{(1+t)^{C_{d,\sigma}}}Z(0) \quad  \forall \, t>0.
\]
This shows that the desired global existence and temporal decay estimates in Theorem \ref{main_thm}.

\vspace{.2cm}

\noindent $\bullet$ (Uniqueness) Here, we show the stability in $L^2$-Sobolev spaces which also implies the uniqueness. For $i=1,2$, let $(n_i, w_i)$ be the solution to \eqref{npER} constructed in the above on the time interval $[0,T]$ corresponding to the initial data $(n_{i,0}, w_{i,0})$. Then we first have
\[
\begin{aligned}
\frac12\frac{d}{dt}\|n_1 - n_2\|_{L^2}^2
&= -\intr (n_1- n_2) (w_1 + v)\cdot \nabla(n_1 - n_2)\,dx  -\intr (n_1 - n_2) (w_1 - w_2)\cdot \nabla n_2 \,dx\\
&\quad -\frac12 \intr (n_1-n_2)^2 \nabla \cdot (w_2+ v)\,dx  -\frac12\intr (n_1 - n_2) n_1 \nabla \cdot (w_1 - w_2)\,dx.
\end{aligned}
\]
Here, integration by parts gives
\[
\begin{aligned}
	\frac12\frac{d}{dt}\|n_1 - n_2\|_{L^2}^2
	&\le C(\|\nabla (w_1 + v)\|_{L^\infty}+\|\nabla (w_2 + v)\|_{L^\infty})\|n_1 - n_2\|_{L^2}^2 \\
	&\quad + \|\nabla n_2\|_{L^\infty} \|n_1 - n_2\|_{L^2}\|w_1 - w_2\|_{L^2}  + C\|n_1\|_{L^\infty}\|n_1 - n_2\|_{L^2}\|\nabla(w_1 - w_2)\|_{L^2}\\
	&\le C( \|n_1 - n_2\|_{L^2}^2 + \|w_1 - w_2\|_{H^1}^2),
\end{aligned}
\]
where we used Young's inequality and $C=C(T)$ is a positive constant. 
We also see
\[
\begin{aligned}
\frac12\frac{d}{dt}\|w_1 - w_2\|_{L^2}^2 &= -\intr (w_1 - w_2) \cdot [(w_1 - w_2)\cdot \nabla (w_2 + v)]\,dx  - \intr (w_1-w_2)\cdot [(w_1 +v)\cdot \nabla( w_1 - w_2)]\,dx\\
&\quad -\intr (w_1 -w_2)\cdot \nabla \Lambda^{-\sigma} ((n_1)^2 - (n_2)^2)\,dx
\end{aligned}
\]
so that 
\[
\begin{aligned}
	\frac12\frac{d}{dt}\|w_1 - w_2\|_{L^2}^2 	&\le C\lt(\|\nabla(w_1 + v)\|_{L^\infty} + \|\nabla(w_2 + v)\|_{L^\infty}\rt)\|w_1-w_2\|_{L^2}^2\\
	&\quad + \|w_1 - w_2\|_{L^2}\|\nabla\Lambda^{-\sigma}((n_1)^2 - (n_2)^2)\|_{L^2}.
\end{aligned}
\]
Note that
\[\begin{aligned}
\|\nabla\Lambda^{-\sigma}((n_1)^2 - (n_2)^2)\|_{L^2} &\le \lt\{\begin{array}{lcl}\|(n_1-n_2)(n_1 + n_2)\|_{L^{\frac{1}{\frac12 + \frac{\sigma-1}{d}}}} & \mbox{if}& \sigma \in [1,2),\\[4mm]
\|(n_1 - n_2)(n_1 + n_2)\|_{\dot{H}^{1-\sigma}} &\mbox{if} & \sigma \in (0,1).
 \end{array}\rt.\\
% &\le C\|n_1 - n_2\|_{H^1}.
\end{aligned}\]
In the case $\sigma \in [1,2)$, the Sobolev embedding $ \dot{H}^{\sigma-1} \hookrightarrow L^{\frac{1}{\frac{1}{2}+\frac{\sigma-1}{d}}}$ implies
\[
\|\nabla\Lambda^{-\sigma}((n_1)^2 - (n_2)^2)\|_{L^2} \leq C\|n_1-n_2\|_{H^1}.
\]
For the case $\sigma \in (0, 1)$, we get the similar result by using \eqref{KP_ineq}. %along with the embedding $H^1 \hookrightarrow \dot{H}^{1-\sigma}$.
Thus, we obtain
\bq\label{low_diff_est1}
\frac{d}{dt}\lt(4\|n_1 -n_2\|_{L^2}^2 + \|w_1-w_2\|_{L^2}^2 \rt) \le C\|w_1-w_2\|_{L^2}^2 + C\|n_1 -n_2\|_{H^1}^2.
\eq

The above estimate \eqref{low_diff_est1} forces us to estimate higher-order derivatives of solutions. Below, we provide the estimates of  $\|n_1 - n_2\|_{\dot{H}^1}$ and $\|w_1-w_2\|_{\dot{H}^{1+\frac\sigma2}}$. Here we remark that there is a difference of $\frac\sigma2$-regularity between $n_1 - n_2$ and $w_1-w_2$ like our solution space. In order to close the estimates of uniqueness, it is required to cancel top-order terms which appear in the estimates for $\|n_1 - n_2\|_{\dot{H}^1}$ and $\|w_1 - w_2\|_{\dot{H}^{1+\frac{\sigma}{2}}}$. 

We begin with the estimate for $\|n_1 - n_2\|_{\dot{H}^1}$.
\[
\begin{aligned}
	\frac12\frac{d}{dt}\|\pa_i(n_1 - n_2)\|_{L^2}^2 
	&= -\intr \pa_i(n_1 - n_2)  \pa_i (w_1 - w_2)\cdot \nabla n_2 \,dx  -\intr \pa_i(n_1 - n_2)  (w_1 - w_2)\cdot \pa_i \nabla n_2 \,dx\\
	&\quad - \intr \pa_i (n_1 - n_2) \pa_i (w_1+v) \nabla \cdot (n_1 - n_2)\,dx  + \frac12\intr [\pa_i(n_1-n_2)]^2 \nabla \cdot (w_1 -w_2)\,dx\\
	&\quad - \frac12\intr \pa_i(n_1-n_2) \pa_i n_1 \nabla \cdot (w_1 -w_2)\,dx\\
	&\quad - \frac14\intr \pa_i(n_1-n_2) (n_1 - n_2) \nabla \cdot \pa_i (w_1 +w_2  +2v)\,dx\\
	&\quad - \frac14 \intr \pa_i (n_1 - n_2) (n_1 + n_2) \nabla \cdot \pa_i (w_1 - w_2)\,dx.
\end{aligned}
\]
Then we have
\[
\begin{aligned}
	\frac12\frac{d}{dt}\|\pa_i(n_1 - n_2)\|_{L^2}^2 
	&\le C\|n_1 - n_2\|_{H^1}\|w_1 - w_2\|_{H^1} + \|n_1 - n_2\|_{H^1} \| (w_1-w_2)\pa_i \nabla n_2\|_{L^2}\\
	&\quad + C\|n_1 - n_2\|_{H^1}^2  - \frac14 \intr \pa_i (n_1 - n_2) (n_1 + n_2) \nabla \cdot \pa_i (w_1 - w_2)\,dx.
\end{aligned}
\]
Note that 
\bq\label{diff_est_core}
\begin{aligned}
\|(w_1 - w_2)\pa_i \nabla n_2\|_{L^2} &\le \lt\{\begin{array}{lcl} \|w_1 - w_2\|_{L^\infty}\|\nabla^2 n_2\|_{L^2} & \mbox{if} & d=1,\\[2mm]
\|w_1 - w_2\|_{L^{\frac{2}{\delta}}}\|\nabla^2 n_2\|_{L^{\frac{2}{1-\delta}}}&\mbox{if}& d =2,\\[2mm]
\|w_1 - w_2\|_{L^{\frac{1}{\frac12 - \frac 1d}}}\|\nabla^2 n_2\|_{L^d} &\mbox{if} & d\ge 3. \end{array}\rt.\\
\end{aligned}
\eq
From the following Sobolev embeddings
$$H^1 \hookrightarrow L^{\infty}\quad \text{and} \quad {H}^s \hookrightarrow \dot{H}^{2}\qquad  \text{when } d=1,$$
$$ \dot{H}^{1-\delta} \hookrightarrow L^{\frac{2}{\delta}}\quad \text{and} \quad {H}^s  \hookrightarrow \dot{H}^{2+\delta}  \hookrightarrow \dot{W}^{2, \frac{2}{1-\delta}}\qquad \text{when } d=2 $$ 
for small $\delta>0$ with $s>2+\delta$, and $$\dot{H}^1 \hookrightarrow L^{\frac{1}{\frac{1}{2}-\frac{1}{d}}} \quad \text{and} \quad H^s \hookrightarrow \dot{H}^{\frac{d}{2}+1}\cap \dot{H}^2 \hookrightarrow\dot{W}^{2, d} \qquad \text{when }d\ge 3$$ 
for $s>1+\frac d2$, we arrive at 
\[
	\|(w_1 - w_2)\pa_i \nabla n_2\|_{L^2} \le C\|w_1 - w_2\|_{H^1}.
\]
Hence we use Young's inequality to obtain
\bq\label{high_diff_est_n1}
\begin{aligned}
\frac12\frac{d}{dt}\|n_1 - n_2\|_{\dot{H}^1}^2 &\le C\|n_1 - n_2\|_{H^1}^2 + C\|w_1 - w_2\|_{H^1}^2\\
&\quad - \frac14 \intr (n_1 +n_2) \nabla (n_1 - n_2) \cdot \nabla (\nabla \cdot (w_1 - w_2))\,dx.
\end{aligned}
\eq

We also estimate $\|w_1 - w_2\|_{\dot{H}^{1+\frac\sigma2}}$ as
\[\begin{aligned}
\frac12\frac{d}{dt}\|w_1 -w_2\|_{\dot{H}^{1+\frac\sigma2}}^2 &= -\intr \Lambda^{1+\frac\sigma2}(w_1-w_2) \cdot \Lambda^{1+\frac\sigma2}[ (w_1 - w_2)\cdot \nabla (w_2 + v)]\,dx\\
&\quad -\intr \Lambda^{1+\frac\sigma2}(w_1 - w_2) \cdot  \Lambda^{1+\frac\sigma2}[ (w_1 +v) \cdot \nabla (w_1 -w_2)]\,dx\\
&\quad - \intr \Lambda^{1+\frac\sigma2}(w_1 -w_2) \cdot \nabla \Lambda^{1-\frac\sigma2}((n_1)^2 - (n_2)^2)\,dx\\
&=: \sum_{i=1}^3 J_i.
\end{aligned}\]
For $J_1$, one uses \eqref{KP_ineq} and the similar arguments employed in \eqref{diff_est_core} to get
\bq\label{diff_est_core2}
\begin{aligned}
&\|\Lambda^{1+\frac\sigma2}[ (w_1 - w_2)\cdot \nabla (w_2 + v)]\|_{L^2} \\
&\quad \lesssim \|\Lambda^{1+\frac\sigma2}(w_1-w_2)\|_{L^2}\|\nabla (w_2 + v)\|_{L^\infty}   + \lt\{\begin{array}{lcl}\|w_1 -w_2\|_{L^\infty}\|\Lambda^{1+\frac\sigma2}\nabla(w_2 +v)\|_{L^2} &\mbox{if} & d=1\\
\|w_1 -w_2\|_{L^{\frac{2}{\delta}}}\|\Lambda^{1+\frac\sigma2}\nabla(w_2 + v)\|_{L^{\frac{2}{1-\delta}}} &\mbox{if} & d=2\\
\|w_1 -w_2\|_{L^{\frac{1}{\frac12 - \frac1d}}}\|\Lambda^{1+\frac\sigma2}\nabla(w_2 + v)\|_{L^d} &\mbox{if} & d\ge 3 \end{array}\rt.\\
&\quad \le C\|w_1 -w_2\|_{H^{1+\frac\sigma2}},
\end{aligned}
\eq
where we used 
\[
w_2 \in H^{s+\frac{\sigma}{2}} \quad \mbox{and} \quad v \in \dot{H}^{2+\frac{\sigma}{2}} \cap \dot{H}^{2+\delta+ \frac{\sigma}{2}}\cap  \dot{H}^{\frac{d}{2}+1+\frac{\sigma}{2}},
\]
and  $\delta>0$ is taken to be small so that $s>2+\delta$.
Thus, we have
\[
J_1 \le C\|w_1 -w_2\|_{H^{1+\frac\sigma2}}^2.
\]
Here $C=C(T)$ is a positive constant.\\

\noindent For $J_2$, we apply \eqref{moser_2term} with $f=w_1+v$, $g=\nabla(w_1-w_2)$, and $s=1+\frac\sigma2$ so that  
\bq\label{diff_est_core3}
\begin{aligned}
J_2 &= -\intr \Lambda^{1+\frac\sigma2}(w_1 - w_2) \cdot \Big( \lt[ \Lambda^{1+\frac\sigma2},\, w_1+v \rt] \cdot \nabla (w_1-w_2) \Big)\, dx \\
&\quad + \frac{1}{2} \intr  \nabla \cdot  (w_1+v)  |\Lambda^{1+\frac\sigma2}(w_1-w_2)|^2 \,dx\\
&\leq C (\|\Lambda(w_1+v)\|_{L^{\infty}} +\|\nabla(w_1+v)\|_{L^{\infty}})\|w_1-w_2\|_{H^{1+\frac{\sigma}{2}}}^2\\
&\quad +C\|w_1-w_2\|_{H^{1+\frac{\sigma}{2}}} \|\Lambda^{1+\frac\sigma2}(w_1 + v) \cdot \nabla (w_1 -w_2)\|_{L^2}\\
&\le C\|w_1 + v\|_{\dot{H}^{\frac d2 + 1 - \e}}^{1/2}\|w_1 + v\|_{\dot{H}^{\frac d2 + 1 + \e}}^{1/2}\|w_1-w_2\|_{H^{1+\frac{\sigma}{2}}}^2\\
&\quad + C\|\Lambda^{1+\frac\sigma2}(w_1 +v)\|_{L^{\frac{2d}{\sigma}}}\|w_1-w_2\|_{H^{1+\frac{\sigma}{2}}}^2\\
&\leq C \|w_1-w_2\|_{H^{1+\frac{\sigma}{2}}}^2
\end{aligned}
\eq
for sufficiently small $\e > 0$, where we used Lemma \ref{lambda_est} and $\dot{H}^{\frac d2 -\frac \sigma2} \hookrightarrow L^{\frac{2d}{\sigma}}$.

For $J_3$, one uses the argument  used in \eqref{diff_est_core} to get
\[
\begin{aligned} 
\|\nabla\cdot (\nabla(n_1 + n_2)(n_1-n_2))\|_{L^2} 
&\le 
\|\nabla(n_1-n_2) \cdot \nabla(n_1+n_2)\|_{L^2}+
\|(n_1-n_2)\nabla^2(n_1+n_2)\|_{L^2} \\
&\le C\|n_1-n_2\|_{H^1},
\end{aligned}
\]
and this implies
\[
\begin{aligned}
J_3 &= \intr \nabla\cdot (w_1 -w_2) \Lambda^2((n_1-n_2)(n_1+n_2))\,dx\\
&= \intr \nabla(\nabla\cdot (w_1-w_2)) \cdot \nabla [(n_1-n_2)(n_1+n_2)]\,dx\\
&= \intr \nabla(\nabla\cdot (w_1-w_2)) \cdot (n_1+n_2)\nabla (n_1-n_2)\,dx  - \intr \nabla\cdot(w_1-w_2)\nabla\cdot (\nabla(n_1 + n_2)(n_1-n_2))\,dx\\
&\le \intr \nabla(\nabla\cdot (w_1-w_2)) \cdot (n_1+n_2)\nabla (n_1-n_2)\,dx  + C\|w_1-w_2\|_{H^1}\|n_1-n_2\|_{H^1}.
\end{aligned}
\]
Thus, we arrive at
\[\begin{aligned}
\frac12\frac{d}{dt}\|w_1-w_2\|_{\dot{H}^{1+\frac\sigma2}}^2 &\le C\|w_1 -w_2\|_{H^{1+\frac\sigma2}}^2 + C\|n_1 -n_2\|_{H^1}^2  \\
&\quad + \intr \nabla(\nabla\cdot (w_1-w_2)) \cdot (n_1+n_2)\nabla (n_1-n_2)\,dx,
\end{aligned}\]
and we deduce from \eqref{high_diff_est_n1} that 
\bq\label{high_diff_est1}
\frac{d}{dt}\lt(4\|n_1-n_2\|_{\dot{H}^1}^2 + \|w_1-w_2\|_{\dot{H}^{1+\frac\sigma2}}^2 \rt) \le C(\|n_1 -n_2\|_{H^1}^2 + \|w_1-w_2\|_{H^{1+\frac\sigma2}}^2).
\eq
Therefore, we combine \eqref{low_diff_est1} with \eqref{high_diff_est1} to yield
\[
\frac{d}{dt}\lt(4\|n_1-n_2\|_{H^1}^2 + \|w_1-w_2\|_{H^{1+\frac\sigma2}}^2 \rt) \le C(\|n_1 -n_2\|_{H^1}^2 + \|w_1-w_2\|_{H^{1+\frac\sigma2}}^2),
\]
and subsequently Gr\"onwall's lemma gives
\[
\|n_1 -n_2\|_{H^1}^2 + \|w_1-w_2\|_{H^{1+\frac\sigma2}}^2 \le C(\|n_{1,0} -n_{2,0}\|_{H^1}^2 + \|w_{1,0}-w_{2,0}\|_{H^{1+\frac\sigma2}}^2),
\]
which implies the desired uniqueness result.

%%%%%%%%%%%%%%%%%%%%%%%%%%%%%%%%%%%%%%%%%%%%%%%%%%%%%%%%%%%%%%%%%%%%%%%%%%%%%%%%%
%
%
%   section4: Pressure case
%
%
%%%%%%%%%%%%%%%%%%%%%%%%%%%%%%%%%%%%%%%%%%%%%%%%%%%%%%%%%%%%%%%%%%%%%%%%%%%%%%%%%

\section{Euler--Riesz system}\label{sec:4}
In this section, we give a priori estimates for system \eqref{ER_main3}.  
Similarly as before, we let $w := u-v$ to rewrite the system \eqref{ER_main3} as follows:
\bq\label{pER}
\begin{aligned}
	&\pa_t n + w \cdot \nabla n + v\cdot \nabla n +\tilde{\gamma} n \nabla \cdot w +\tilde{\gamma} n \nabla \cdot v = 0, \quad (x,t) \in \R^d \times \R_+,\cr
	&\pa_t w + w\cdot \nabla w  + v\cdot \nabla w + w\cdot \nabla v+ \tilde{\gamma} n \nabla n  =  \lambda \nabla \Lambda^{-\sigma} n^{\frac1{\tilde{\gamma}}}
\end{aligned}
\eq
with $\tilde{\gamma}=(\gamma-1)/2$. Here, the lower order estimates for \eqref{pER} proceed in the same way regardless of the choice of $\sigma$ and the sign for the potential  $\lambda =\pm 1$.
 
We first start with estimates of $H^s$-norm of solutions.
\begin{lemma}\label{pres_est1}
Let $(n,w)$ be a smooth solution to \eqref{pER} with $\lambda=\pm1$ decaying sufficiently fast at infinity. Then we have
\[
	\frac{d}{dt} \|n\|_{L^2} +  \frac{d\lt( \tilde{\gamma} -\frac{1}{2}\rt)}{1+t} \|n\|_{L^2} \lesssim  \| \nabla w\|_{L^{\infty}} \|n\|_{L^2} +  \frac{\|n\|_{L^2}}{(1+t)^2},
\]
\[
	\frac{d}{dt}\|w\|_{L^2} + \frac{1-\frac d2}{1+t}\|w\|_{L^2} \lesssim \|\nabla w\|_{L^\infty} \|w\|_{L^2} +   \frac{\|w\|_{L^2}}{(1+t)^2} +\|n\|_{L^2}\|\nabla n\|_{L^{\infty}}+ \|\nabla \Lambda^{-\sigma} (n^{\frac1{\tilde{\gamma}}})\|_{L^2},
\]
	and
	\[ 
	\begin{aligned}
		\frac{1}{2} \frac{d}{dt} \w{X}_s^2 +\frac{s-\frac{d}{2}+\min\{ d\tilde{\gamma}, 1\}}{1+t} \w{X}_s^2   
	&-\lambda \intr \Lambda^s w \cdot \nabla \Lambda^{s-\sigma} (n^{\frac1{\tilde{\gamma}}})\,dx \\
	\lesssim
	(\|\nabla w\|_{L^\infty}+ \|\nabla n\|_{L^\infty}) \w{X}_s^2 
	&+ \frac{\w{X}_s^2}{(1+t)^2} \\
	 \quad + \bigg( \frac{\|n\|_{L^\infty}+\|w\|_{L^\infty}}{(1+t)^{2+s-\frac d2}}&+ \frac{\|n\|_{\dot{H}^{s-1}} +\|w\|_{\dot{H}^{s-1}}}{(1+t)^3} + \frac{\|\nabla n\|_{L^\infty}+\|\nabla w\|_{L^\infty}}{(1+t)^{s+1-\frac d2}}\bigg)\w{X}_s,
	\end{aligned}
	\]
	where $\w{X}_s^2 := \|w\|_{\dot{H}^{s}}^2 + \|n\|_{\dot{H}^s}^2$.
\end{lemma}
\begin{proof}
First, we readily observe that   
	\[
	\begin{aligned}
	\frac{1}{2}\frac{d}{dt} \intr |n|^2 \, dx 
	&=  \lt(\frac{1}{2}-\tilde{\gamma} \rt) \intr n^2 (\nabla \cdot  w) \, dx + \lt( \frac{1}{2} -\tilde{\gamma}\rt)\intr n^2  (\nabla \cdot v) 
	\, dx\\
	&\leq
	- \frac{d\lt( \tilde{\gamma} -\frac{1}{2}\rt)}{1+t} \intr n^2 \,dx 
	+ C \| \nabla w\|_{L^{\infty}}  \|n\|_{L^2}^2 
	+  \frac{C\|n\|_{L^2}^2}{(1+t)^2}. 
	\end{aligned}
	\]
For the $L^2$-estimate of $w$, the similar estimates give the desired estimate.
 
For highest-order estimates, we apply $\Lambda^s $ to \eqref{pER} to see that 
	\[\begin{aligned}
		&\pa_t \Lambda^s n + w \cdot \nabla \Lambda^s n + v \cdot \nabla \Lambda^s n -s \sum_{1\leq k\leq d} \pa_k v \cdot \Lambda^{s-2}\pa_k \nabla  n + \tilde{\gamma} n  \nabla \cdot \Lambda^s w+ \tilde{\gamma} \Lambda^{s}(n \nabla \cdot v) \\
		&\qquad  = [w , \Lambda^s]\cdot \nabla n+ \Big( [v, \Lambda^s]\cdot\nabla n -s \sum_{1\leq k\leq d} \pa_k v \cdot \Lambda^{s-2}\pa_k \nabla  n\Big) + \tilde{\gamma} [n, \Lambda^s]\nabla \cdot w 
			\end{aligned}
	\]
	and
		\[\begin{aligned}
		&\pa_t \Lambda^s w + w \cdot \nabla\Lambda^s w + v \cdot \nabla \Lambda^s w - s \sum_{1\leq k\leq d} \pa_k v \cdot \Lambda^{s-2} \pa_k \nabla w + \Lambda^{s} (w \cdot \nabla v)+\tilde{\gamma} n \nabla \Lambda^{s}  n \\
		&\qquad = \lambda\nabla \Lambda^{s-\sigma}(n^{\frac{1}{\tilde{\gamma}}}) 
		+ [w, \Lambda^s] \cdot\nabla w+\tilde{\gamma}[n, \Lambda^{s}]\nabla n 
		+ \Big([v, \Lambda^s] \cdot \nabla w- s \sum_{1\leq k\leq d} \pa_k v \cdot \Lambda^{s-2} \pa_k \nabla w  \Big).
	\end{aligned}
	\]
	By using the same argument as shown in Lemma \ref{low_est}, we deduce that   
	\[
	\begin{aligned}
		\frac{1}{2} \frac{d}{dt} &\|n\|_{\dot{H}^s}^2 +\frac{s-\frac{d}{2}+d\tilde{\gamma}}{1+t} \|n\|_{\dot{H}^s}^2  
		+\tilde{\gamma} \intr \Lambda^{s}n (n \nabla \cdot \Lambda^{s}w)\,dx \\
		&\lesssim
		\|\nabla w\|_{L^\infty}\|n\|_{\dot{H}^s}^2 
		+ \frac{\|n\|_{\dot{H}^s}^2}{(1+t)^2} 
		+ \lt( \frac{\|n\|_{L^\infty}}{(1+t)^{2+s-\frac d2}}+ \frac{\|n\|_{\dot{H}^{s-1}} }{(1+t)^3} + \frac{\|\nabla n\|_{L^\infty}}{(1+t)^{s+1-\frac d2}}\rt)\|n\|_{\dot{H}^s}\\
		&\quad +  \|\nabla n\|_{L^\infty}\|w\|_{\dot{H}^s}\|n\|_{\dot{H}^s} 
	\end{aligned}
	\]
	and
	\[ 
	\begin{aligned}
		\frac{1}{2} \frac{d}{dt} &\|w\|_{\dot{H}^s}^2 
		+ \frac{s -\frac d2+1}{1+t} \|w\|_{\dot{H}^s}^2
		+ \tilde{\gamma} \intr \Lambda^{s} w \cdot (n \nabla \Lambda^{s}n )\,dx 
		-\lambda \intr \Lambda^{s}w \cdot \nabla \Lambda^{s-\sigma} (n^{\frac1{\tilde{\gamma}}}) \,dx	\\
		&\lesssim
		\|\nabla w\|_{L^\infty}\|w\|_{\dot{H}^{s}}^2  +\frac{\|w\|_{\dot{H}^{s}}^2}{(1+t)^2} +\lt( \frac{\|w\|_{L^\infty}}{(1+t)^{2+s-\frac d2}} 
		+  \frac{\|w\|_{\dot{H}^{s-1}}}{(1+t)^3} +\frac{\|\nabla w\|_{L^\infty}}{(1+t)^{s+1 -\frac d2}} \rt)\|w\|_{\dot{H}^{s}}.
	\end{aligned}
	\]
	Moreover, using the integration by parts gives 
	\[ 
	\lt| \tilde{\gamma} \intr \Lambda^s n (n \nabla \cdot \Lambda^s w)\,dx+\tilde{\gamma} \intr \Lambda^s w \cdot (n \nabla \Lambda^s n)\,dx  \rt|
	\lesssim \|\nabla n\|_{L^{\infty}} \|n\|_{\dot{H}^s} \|w\|_{\dot{H}^s},
	\]
and thus, we arrive at
	\[
	\begin{aligned}
		\frac{1}{2} &\frac{d}{dt} \w{X}_s^2 +\frac{s-\frac{d}{2}+\min\{ d\tilde{\gamma}, 1\}}{1+t} \w{X}_s^2   
		-\lambda \intr \Lambda^s w \cdot \nabla \Lambda^{s-\sigma} (n^{\frac1{\tilde{\gamma}}})\,dx \\
		&\lesssim
		(\|\nabla w\|_{L^\infty}+ \|\nabla n\|_{L^\infty}) \w{X}_s^2 
		+ \frac{\w{X}_s^2}{(1+t)^2} \\
		& \quad + \lt( \frac{\|n\|_{L^\infty}+\|w\|_{L^\infty}}{(1+t)^{2+s-\frac d2}}+ \frac{\|n\|_{\dot{H}^{s-1}} +\|w\|_{\dot{H}^{s-1}}}{(1+t)^3} + \frac{\|\nabla n\|_{L^\infty}+\|\nabla w\|_{L^\infty}}{(1+t)^{s+1-\frac d2}}\rt)\w{X}_s
	\end{aligned}
	\]
	as desired.
\end{proof}

Now, we deal with three cases:
\begin{enumerate}
\item[(i)]
$1\leq \sigma \leq 2$ (Theorem \ref{main_thm2}),
\item[(ii)]
$0<\sigma<1$ with the replusive potential, i.e. $\lambda=-1$ (Theorem \ref{main_thm3}),
\item[(iii)]
 $0<\sigma<1$ with the attractive potential, i.e. $\lambda=1$ (Theorem \ref{main_thm4}).
\end{enumerate}

As mentioned in Introduction, the case (i), sub-Manev interaction potential case, is already taken into account in \cite{DD22} regardless of the sign of $\lambda$. However, our solution space is different from that of \cite{DD22}, see Remark \ref{rem2} (iii), and for the completeness of our work, we leave its proof in Appendix \ref{app.A}. Thus, in the rest of this section, we provide details of the estimates for cases (ii) and (iii).

%%%%%%%%%%%%%%%%%%%%%%%%%%%%%%%%%%%%%%%%%%%%%%%%%%%%%%%%%%%%%%%%%%%%%%%%%%%%%%%%%5
%
%
%                        Section: Introduction 
%
%
%%%%%%%%%%%%%%%%%%%%%%%%%%%%%%%%%%%%%%%%%%%%%%%%%%%%%%%%%%%%%%%%%%%%%%%%%%%%%%%%%

\subsection{Repulsive interaction case}

In this subsection, we present the global well-posedness of \eqref{pER} with the repulsive potential when $\sigma \in (0,1)$. Since the lowest-order estimates are given in Lemma \ref{pres_est1}, we only need to obtain the highest-order estimates.

\begin{lemma}\label{Hs_est2_case1}
	Let the assumptions of Theorem \ref{main_thm3} be satisfied.
For $T>0$, suppose that $(n,w)$ is a smooth solution to \eqref{pER} with $\lambda=-1$ on the time interval $[0,T]$ decaying fast at infinity.
	 Then we have
	\[
\begin{aligned}
	\frac{1}{2} &\frac{d}{dt} \lt(\w{X}_s^2 
	+	\frac{1}{\tilde{\gamma}^2} \intr n^{\frac{1}{\tilde{\gamma}}-2} |\Lambda^{s-\frac{\sigma}{2}} n|^2 \,dx\rt)
	+\frac{s-\frac{d}{2}+\min\{ d\tilde{\gamma}, 1, \frac{d-\sigma}{2}\}}{1+t} \lt(\w{X}_s^2  +\frac{1}{\tilde\gamma^2}\intr n^{\frac{1}{\tilde{\gamma}}-2}  |\Lambda^{s-\frac{\sigma}{2}}n|^2\,dx\rt) \\
	&\lesssim (\|\nabla w\|_{L^\infty}+ \|\nabla n\|_{L^\infty}) \w{X}_s^2 
	+ \frac{\w{X}_s^2}{(1+t)^2} \\
	& \quad + \lt( \frac{\|n\|_{L^\infty}+\|w\|_{L^\infty}}{(1+t)^{2+s-\frac d2}}+ \frac{\|n\|_{\dot{H}^{s-1}} +\|w\|_{\dot{H}^{s-1}}}{(1+t)^3} + \frac{\|\nabla n\|_{L^\infty}+\|\nabla w\|_{L^\infty}}{(1+t)^{s+1-\frac d2}}\rt)\w{X}_s\\
	& \quad +\Bigg(\|\nabla w\|_{L^\infty}\|n\|_{\dot{H}^{s-\frac\sigma2}} + \|\nabla n\|_{L^\infty}\|w\|_{\dot{H}^{s-\frac\sigma2}}+ \frac{\|n\|_{\dot{H}^{s-\frac\sigma2}}}{(1+t)^2} + \frac{\|n\|_{\dot{H}^{s-\frac\sigma2-1}}}{(1+t)^3} \\
	&\quad + \frac{\|n\|_{L^\infty} + (1+t)\|\nabla n\|_{L^\infty}}{(1+t)^{2+s-\frac\sigma2-\frac d2}}  \Bigg)\|n\|_{L^\infty}^{\frac{1}{\tilde\gamma}-2} \|n\|_{\dot{H}^{s-\frac\sigma2}} +  \|n\|_{L^\infty}^{\frac{1}{\tilde{\gamma}}-2} \|n\|_{\dot{H}^{s-\sigma}}^{\frac{\frac d2+(1-\sigma)}{s-\sigma}}\|n\|_{L^2}^{1-\frac{\frac d2+(1-\sigma)}{s-\sigma}}\|n\|_{\dot{H}^s} \|w\|_{\dot{H}^s} ,
\end{aligned}
\]
where
	$\w{X}_{s}^2 := \|w\|_{\dot{H}^{s}}^2 + \|n\|_{\dot{H}^s}^2$. 
	\end{lemma}
	\begin{proof}
We use \eqref{tech_1} and Lemma \ref{tech_4} to get 
\bq\label{riesz_est1}
\begin{aligned}
	 &\lt| \intr \Lambda^s w \cdot \nabla \Lambda^{s-\frac{\sigma}{2}} (n^{\frac{1}{\tilde{\gamma}}})\,dx - \frac{1}{\tilde{\gamma}}\intr n^{\frac{1}{\tilde\gamma}-1}\Lambda^s w \cdot \nabla \Lambda^{s-\sigma} n\,dx\rt|\\
	 &\quad  \le  \frac{1}{\tilde{\gamma}}\lt| \intr \Lambda^s w \cdot [\Lambda^{s-\sigma}, n^{\frac{1}{\tilde{\gamma}}-1}] \nabla  n\,dx \rt|\\
	 &\quad\le C\|w\|_{\dot{H}^s}\lt( \|\nabla n\|_{L^\infty}\|n^{\frac{1}{\tilde\gamma} -1}\|_{\dot{H}^{s-\sigma}} + \|\nabla (n^{\frac{1}{\tilde\gamma}-1})\|_{L^\infty} \|n\|_{\dot{H}^{s-\sigma}}\rt)\\
	 &\quad\le C\|\nabla n\|_{L^\infty} \|n\|_{L^\infty}^{\frac{1}{\tilde\gamma}-2}\|n\|_{\dot{H}^{s-\sigma}}\|w\|_{\dot{H}^s},
\end{aligned}
\eq
where we used $s<\frac{1}{\tilde\gamma} + \sigma -\frac12$.

 Now, we observe that
	\bq\label{n_lower_est}
	\begin{aligned}
		\frac{1}{2\tilde\gamma^2} &\frac{d}{dt} \intr n^{\frac{1}{\tilde{\gamma}}-2} |\Lambda^{s-\frac\sigma2} n|^2\,dx  +\frac{s-\frac{\sigma}{2}}{\tilde{\gamma}^2(1+t)}  \intr n^{\frac{1}{\tilde{\gamma}}-2}  |\Lambda^{s-\frac\sigma2} n|^2\,dx-\frac{1}{\tilde\gamma} \intr n^{\frac{1}{\tilde\gamma}-1}\nabla \Lambda^{s-\sigma}n \cdot \Lambda^s w \,dx\\
		&\lesssim  
		\biggl( \|\nabla w\|_{L^{\infty}} \|n\|_{\dot{H}^{s-\frac\sigma2}}
		+\frac{\|n\|_{\dot{H}^{s-\frac\sigma2}}}{(1+t)^2}   
		+ \frac{\|n\|_{L^\infty}}{(1+t)^{2+s-\frac\sigma2-\frac d2}}  \\
		&\qquad 	+ \|\nabla n\|_{L^\infty}\|w\|_{\dot{H}^{s-\frac\sigma2}} + \frac{\|\nabla n\|_{L^\infty} }{(1+t)^{1+s-\frac\sigma2-\frac d2}}
		+ \frac{\|n\|_{\dot{H}^{s-\frac\sigma2-1}}}{(1+t)^3} \biggl) \|n\|_{L^\infty}^{\frac{1}{\tilde{\gamma}}-2} \|n\|_{\dot{H}^{s-\frac\sigma2}} \\
		&\qquad +  \|n\|_{L^\infty}^{\frac{1}{\tilde{\gamma}}-2} \|n\|_{\dot{H}^{s-\sigma}}^{\frac{\frac d2+(1-\sigma)}{s-\sigma}}\|n\|_{L^2}^{1-\frac{\frac d2+(1-\sigma)}{s-\sigma}}\|n\|_{\dot{H}^s} \|w\|_{\dot{H}^s} 
	\end{aligned}
	\eq
for $\frac1{\tilde{\gamma}}\ge 2$. Here we postpone the proof of \eqref{n_lower_est} to Appendix \ref{app.B} for the smoothness of reading.

 Then we gather Lemma \ref{pres_est1}, \eqref{riesz_est1}, and \eqref{n_lower_est} to yield 
	\[
		\begin{aligned}
				\frac{1}{2} &\frac{d}{dt} \lt(\w{X}_s^2 
	+	\frac{1}{\tilde{\gamma}^2} \intr n^{\frac{1}{\tilde{\gamma}}-2} |\Lambda^{s-\frac{\sigma}{2}} n|^2 \,dx\rt)
	+\frac{s-\frac{d}{2}+\min\{ d\tilde{\gamma}, 1, \frac{d-\sigma}{2}\}}{1+t} \lt(\w{X}_s^2  +\frac{1}{\tilde\gamma^2}\intr n^{\frac{1}{\tilde{\gamma}}-2}  |\Lambda^{s-\frac{\sigma}{2}}n|^2\,dx\rt) \\
		&\lesssim (\|\nabla w\|_{L^\infty}+ \|\nabla n\|_{L^\infty}) \w{X}_s^2 
		+ \frac{\w{X}_s^2}{(1+t)^2} \\
		& \quad + \lt( \frac{\|n\|_{L^\infty}+\|w\|_{L^\infty}}{(1+t)^{2+s-\frac d2}}+ \frac{\|n\|_{\dot{H}^{s-1}} +\|w\|_{\dot{H}^{s-1}}}{(1+t)^3} + \frac{\|\nabla n\|_{L^\infty}+\|\nabla w\|_{L^\infty}}{(1+t)^{s+1-\frac d2}}\rt)\w{X}_s\\
		& \quad +\Bigg(\|\nabla w\|_{L^\infty}\|n\|_{\dot{H}^{s-\frac\sigma2}} + \|\nabla n\|_{L^\infty}\|w\|_{\dot{H}^{s-\frac\sigma2}}+ \frac{\|n\|_{\dot{H}^{s-\frac\sigma2}}}{(1+t)^2} + \frac{\|n\|_{\dot{H}^{s-\frac\sigma2-1}}}{(1+t)^3} \\
		&\quad + \frac{\|n\|_{L^\infty} + (1+t)\|\nabla n\|_{L^\infty}}{(1+t)^{2+s-\frac\sigma2-\frac d2}}  \Bigg)\|n\|_{L^\infty}^{\frac{1}{\tilde\gamma}-2} \|n\|_{\dot{H}^{s-\frac\sigma2}} +  \|n\|_{L^\infty}^{\frac{1}{\tilde{\gamma}}-2} \|n\|_{\dot{H}^{s-\sigma}}^{\frac{\frac d2+(1-\sigma)}{s-\sigma}}\|n\|_{L^2}^{1-\frac{\frac d2+(1-\sigma)}{s-\sigma}}\|n\|_{\dot{H}^s} \|w\|_{\dot{H}^s} ,
		\end{aligned}
		\]
		which is our desired estimate.
		\end{proof}

Now we investigate a priori temporal decay estimates. Similarly to the previous section, we introduce
\bq\label{temp_notation}
\begin{aligned}
	\w{n}_{\ell_1, p_1}(t) := (1+t)^{\ell_1-\frac{d}{p_1}-1}\|n(t)\|_{\dot{W}^{\ell_1,p_1}} \quad \mbox{and} \quad 
	\w{w}_{\ell_2, p_2}(t):= (1+t)^{\ell_2 - \frac{d}{p_2}-1}\|w(t)\|_{\dot{W}^{\ell_2, p_2}}
\end{aligned}
\eq
for $\ell_1 \in [0, s]$, $\ell_2 \in [0,s]$, and $p_1, p_2 \in [2,\infty]$. We also write $\tilde{n}_{p_1} :=\tilde{n}_{0, p_1}$ and $\tilde{w}_{p_2}:=w_{0, p_2}$ for simplicity.  We then define 
\bq\label{temp_notation_2}
\widetilde{Y}_{s}(t) :=  \lt(\widetilde{w}_{s,2}^2(t)  + \widetilde{n}_{s,2}^2(t)\rt)^{\frac12}  \quad \mbox{and} \quad 
\widetilde{Z}(t) :=  \lt( \widetilde{n}_2^2(t) + \widetilde{w}_2^2(t)  + \widetilde{Y}_{s}^2(t)\rt)^{\frac 12}.
\eq

Then similarly to Lemma \ref{bdds}, we get the following lemma using the Gagliardo--Nirenberg interpolation inequality.

\begin{lemma}\label{bdds2}
	If $s> \frac d2+1$ and $\ell \in [0,s]$, we have the following relations 
	\[
	\begin{aligned}
		&	\widetilde{n}_{1, \infty}(t) \lesssim  \widetilde{n}_{s, 2}^{\frac{\frac{d}{2}+1}{s}}(t)\,\widetilde{n}_2^{1-\frac{\frac{d}{2}+1}{s}}(t), \hspace{40pt} 
		\widetilde{w}_{1, \infty}(t) \lesssim  \widetilde{w}_{s, 2}^{\frac{\frac d2+1}{s}}(t)\,\widetilde{w}_2^{1-\frac{\frac d2 +1}{s}}(t), \\ 
		&\widetilde{n}_{s-\ell, 2}(t) \lesssim \widetilde{n}_{s, 2}^{\frac{s-\ell}{s}}(t)\, \widetilde{n}_{2}^{\frac{\ell}{s}}(t), \hspace{60pt}
		\widetilde{w}_{s-\ell, 2}(t)\lesssim  \widetilde{w}_{s, 2}^{\frac{s-\ell}{ s}}(t)\,\widetilde{w}_2^{\frac{\ell}{ s}}(t),\\
		&\widetilde{n}_{\infty}(t) \lesssim \widetilde{n}_{s, 2}^{\frac{\frac{d}{2}}{s}}(t)\, \widetilde{n}_{2}^{1-\frac{\frac{d}{2}}{s}}(t).
		\end{aligned}
	\] 
	In particular, each term on the left-hand side of the inequalities above is bounded by $C \w{Z}(t)$ with a constant $C>0$ independent of $t$. 
\end{lemma}

Now we present a priori temporal decay estimates of solutions.
\begin{proposition}\label{decay_temp2_case2}
	Let the assumptions of Theorem \ref{main_thm3} be satisfied. 
	For $T>0$, suppose that $(n,w)$ is a smooth solution to \eqref{pER} with $\lambda=-1$ on the time interval $[0,T]$ decaying sufficiently fast at infinity.
	Then we have
	\[
\begin{aligned}
	\frac{1}{2} \frac{d}{dt} (\widetilde{Z}^2(t)+W(t)) 
	+\frac{C_{d,\sigma,\gamma}}{1+t} (\widetilde{Z}^2(t)+W(t)) \lesssim \widetilde{Z}^3(t) 
	+ \frac{\widetilde{Z}^2(t) }{(1+t)^2} 
	+\frac{ \widetilde{Z}^{1+\frac{1}{\tilde{\gamma}}}(t)  }{(1+t)^{2-\sigma -\frac{1}{\tilde{\gamma}}}} 
	+  \frac{\widetilde{Z}^{\frac{1}{\tilde{\gamma}}}(t)}{(1+t)^{4-\sigma -\frac{1}{\tilde{\gamma}}}}
\end{aligned}
\]
	with $C_{d, \sigma, \gamma}=1+\min\{ 1,\, \tilde{\gamma}d,\, \frac{d-\sigma}{2}\}$, where  
	$$
	W(t) := \frac{(1+t)^{2(s-\frac{d}{2}-1)}}{\tilde{\gamma}^2}\intr n^{\frac{1}{\tilde{\gamma}}-2}  |\Lambda^{s-\frac{\sigma}{2}}n|^2\, dx .
	$$
\end{proposition}
\begin{proof}
It follows from Lemma \ref{pres_est1} that
\[
\frac{d}{dt} \w{n}_{2} +\frac{1+d\tilde{\gamma}}{1+t} \w{n}_2 
	\lesssim 
	\w{w}_{1, \infty} \w{n}_2 +\frac{\w{n}_2}{(1+t)^2}
\]
and
\[
\frac{d}{dt} \w{w}_2 
	+ \frac{2}{1+t} \w{w}_2 \lesssim \w{w}_{1,\infty}\w{w}_2+  \frac{\w{w}_2}{(1+t)^2} 
	+   \w{n}_{1, \infty}\w{n}_2
	+ \frac{ \|\nabla \Lambda^{-\sigma} (n^{\frac1{\tilde{\gamma}}})\|_{L^2}}{(1+t)^{\frac d2 +1}}.
\]
Here, we use Lemma \ref{tech_4} to estimate
\[\begin{aligned}
\|\nabla \Lambda^{-\sigma}(n^{\frac{1}{\tilde\gamma}})\|_{L^2}\lesssim \|n\|_{L^\infty}^{\frac{1}{\tilde\gamma}-1}\|n\|_{\dot{H}^{1-\sigma}} \lesssim (1+t)^{\frac d2 + 1 + \frac{1}{\tilde\gamma} -2 +\sigma} \w{n}_\infty^{\frac{1}{\tilde\gamma}-1} \w{n}_{1-\sigma,2},
\end{aligned}\]
and it gives
\bq\label{case2_lower_est}
\frac{d}{dt}\lt(\w{n}_2 + \w{w}_2 \rt) + \frac{C_{d,\sigma,\gamma}}{1+t} \lt(\w{n}_2 + \w{w}_2\rt) \lesssim \tilde{Z}^2(t) + \frac{\tilde{Z}^{\frac{1}{\tilde\gamma}}}{(1+t)^{2-\sigma-\frac{1}{\tilde\gamma}}}.
\eq
Then, Lemma \ref{Hs_est2_case1} leads us to get
	\[
	\begin{aligned}
		\frac{1}{2} &\frac{d}{dt}\lt( \widetilde{Y}_s^2 +\frac{(1+t)^{2(s-\frac{d}{2}-1)}}{\tilde{\gamma}^2} \intr n^{\frac{1}{\tilde{\gamma}}-2} |\Lambda^{s-\frac{\sigma}{2}} n|^2 \,dx\rt) \\
		&+\frac{C_{d,\sigma,\gamma}}{1+t} \lt(\widetilde{Y}_s^2
		+\frac{(1+t)^{2(s-\frac{d}{2}-1)} }{\tilde{\gamma}^2}  \intr n^{\frac{1}{\tilde{\gamma}}-2}  |\Lambda^{s-\frac{\sigma}{2}}n|^2\,dx\rt)  \\
		&\qquad \lesssim(\widetilde{w}_{1, \infty}+ \widetilde{n}_{1, \infty}) \widetilde{Y}_s^2  
		+ \frac{\widetilde{Y}_s^2 + (\widetilde{n}_{\infty}  +\widetilde{w}_{\infty} +\widetilde{n}_{1, \infty}  +\widetilde{w}_{1, \infty}  +\widetilde{n}_{s-1,2} + \widetilde{w}_{s-1,2}) \widetilde{Y}_s }{(1+t)^2} \\
		&\qquad \quad 
		+ \frac{\lt( \widetilde{w}_{1, \infty}  \w{n}_{s-\frac{\sigma}{2},2}   + \widetilde{n}_{1, \infty} \w{w}_{s-\frac{\sigma}{2},2}   \rt)  \w{n}_{\infty}^{\frac{1}{\tilde{\gamma}}-2} \w{n}_{s-\frac{\sigma}{2},2} +\w{n}_{\infty}^{\frac{1}{\tilde{\gamma}}-2}\w{n}_{s-\sigma,2}^{\frac{\frac d2 + (1-\sigma)}{s-\sigma}}\w{n}_2^{1-\frac{\frac d2 + (1-\sigma)}{s-\sigma}} \w{w}_{s,2} \w{n}_{s,2}}{(1+t)^{2-\sigma-\frac{1}{\tilde{\gamma}}}}\\
		&\qquad  \quad+ \frac{\lt(\w{n}_\infty + \w{n}_{1,\infty} + \w{n}_{s-\frac\sigma2, 2} + \w{n}_{s-\frac\sigma2 -1,2} \rt)   \widetilde{n}_{\infty}^{\frac{1}{\tilde{\gamma}}-2} \w{n}_{s-\frac{\sigma}{2},2} }{(1+t)^{4-\sigma -\frac{1}{\tilde{\gamma}}}} .
	\end{aligned}
	\]
	Moreover, thanks to Lemma \ref{bdds2}, we obtain
	\bq\label{case2_high_est}
	\begin{aligned}
		\frac{1}{2} &\frac{d}{dt}\lt( \widetilde{Y}_s^2 +\frac{(1+t)^{2(s-\frac{d}{2}-1)}}{\tilde{\gamma}^2} \intr n^{\frac{1}{\tilde{\gamma}}-2} |\Lambda^{s-\frac{\sigma}{2}} n|^2 \,dx\rt) \\
		&+\frac{C_{d,\sigma,\gamma}}{1+t} \lt(\widetilde{Y}_s^2
		+\frac{(1+t)^{2(s-\frac{d}{2}-1)} }{\tilde{\gamma}^2}  \intr n^{\frac{1}{\tilde{\gamma}}-2}  |\Lambda^{s-\frac{\sigma}{2}}n|^2\,dx\rt)  \\
		&\lesssim \widetilde{Z}^3(t) 
		+ \frac{\widetilde{Z}^2(t) }{(1+t)^2} 
		+\frac{ \widetilde{Z}^{1+\frac{1}{\tilde{\gamma}}}(t)  }{(1+t)^{2-\sigma -\frac{1}{\tilde{\gamma}}}} 
		+  \frac{\widetilde{Z}^{\frac{1}{\tilde{\gamma}}}}{(1+t)^{4-\sigma -\frac{1}{\tilde{\gamma}}}}.
	\end{aligned}
	\eq
	Hence we combine \eqref{case2_lower_est} with \eqref{case2_high_est} to yield the desired estimate.
\end{proof}

\vspace{0.3cm}

\begin{proof}[Proof of Theorem \ref{main_thm3}]
	As we previously did, we separately show the global-in-time existence and uniqueness of solutions.
	
	\vspace{.2cm}
	
	\noindent $\bullet$ (Global-in-time existence and temporal decay): For the local-in-time existence, we may proceed as we did in Section \ref{sec:3}. For the global existence, we note that the density $n$ is nonnegative and hence
	$$
	W(t) = \frac{(1+t)^{2(s-\frac{d}{2}-1)}}{\tilde{\gamma}^2}\intr n^{\frac{1}{\tilde{\gamma}}-2}  |\Lambda^{s-\frac{\sigma}{2}}n|^2\, dx \ge 0.
	$$
This together with Proposition \ref{decay_temp2_case2} implies
		\[
		\begin{aligned}
			&\frac{1}{2}\frac{d}{dt} \mathcal{Z}^2(t) 
			+\frac{C_{d, \sigma, \gamma}}{1+t} \mathcal{Z}^2(t) \lesssim \mathcal{Z}^3(t) 
			+ \frac{\mathcal{Z}^2(t)}{(1+t)^2} 
			+\frac{\mathcal{Z}^{1+\frac{1}{\tilde{\gamma}}}(t)  }{(1+t)^{2-\sigma -\frac{1}{\tilde{\gamma}}}} 
			+  \frac{\mathcal{Z}^{\frac{1}{\tilde{\gamma}}}(t)}{(1+t)^{4-\sigma -\frac{1}{\tilde{\gamma}}}},
		\end{aligned}
		\]
where $\mathcal{Z}^2(t) :=\w{Z}^2(t)+W(t)$.
Then it follows that
		\[
\begin{aligned}
	&\frac{d}{dt} \mathcal{Z}(t) 
	+\frac{C_{d, \sigma, \gamma}}{1+t} \mathcal{Z}(t) \leq C_1 \lt( \mathcal{Z}^2(t) 
	+ \frac{\mathcal{Z}(t)}{(1+t)^2} 
	+\frac{\mathcal{Z}^{\frac{1}{\tilde{\gamma}}}(t)  }{(1+t)^{2-\sigma -\frac{1}{\tilde{\gamma}}}} 
	+  \frac{\mathcal{Z}^{\frac{1}{\tilde{\gamma}}-1}(t)}{(1+t)^{4-\sigma -\frac{1}{\tilde{\gamma}}}} \rt)
\end{aligned}
\]
for a constant $C_1>0$ independent of $t$.
If $\w{Z}(0)$ is sufficient small, then 
		 \bq\label{w_0}
		 \mathcal{Z}(0) =(\w{Z}^2(0) +W(0))^{\frac 12} \leq \w{Z}(0)+W^{\frac 12}(0) \leq \w{Z}(0)+\w{Z}^{\frac {1}{2\tilde{\gamma}}}(0) \leq C \w{Z}(0) 
		 \eq
		 for $\frac1{\tilde{\gamma}}\ge 2$, since  
		 \[
		 W(0)= \frac{1}{\tilde{\gamma}}\intr n_0^{\frac{1}{\tilde{\gamma}}-2}  |\Lambda^{s-\frac{\sigma}{2}}n_0|^2\, dx  \lesssim \|n_0\|_{L^{\infty}}^{\frac{1}{\tilde{\gamma}}-2} \|n_0\|_{\dot{H}^{s-\frac{\sigma}{2}}}^2 \lesssim \widetilde{Z}^{\frac{1}{\tilde{\gamma}}}(0).
		 \]
		Then we apply Lemma \ref{tech_6} with $a=C_{d, \sigma, \gamma}$ so that 
		\bq\label{eq}
		 \mathcal{Z}(t)\le \frac{2e^{\frac{C_1t}{1+t}}}{(1+t)^a}\mathcal{Z}(0) \quad \forall \, t\ge 0,
		\eq
		whenever
		\[
		\sigma<\min \lt\{1,\, d \tilde{\gamma}, \frac{d-\sigma}{2} \rt\}\lt(\frac{1}{\tilde{\gamma}}-1\rt),
		\]
		or equivalently 
		\[
		\gamma < 
		\min\lt\{1+\frac{2}{\sigma+1},  \ 1+ \frac{2(d-\sigma)}{d+\sigma} \rt\}
		\]
		due to 
		\[
		\frac{2(d-\sigma)}{d+\sigma}  < \frac{2(d-\sigma)}{d}.
		\]
		We note that $\frac1{\tilde{\gamma}} \ge 2$ implies $\gamma \leq 2$, from which $\gamma<1+\frac{2}{\sigma+1}$ is true if $\sigma<1$. Moreover,
		\[
		2< 1+\frac{2(d-\sigma)}{d+\sigma} \ \Longleftrightarrow \ 3\sigma <d.
		\]
		Hence the condition of $\gamma$ is required as $1<\gamma\leq 2$ and
		\[
		\gamma < 
		1+\frac{2(d-\sigma)}{d+\sigma}\quad  \mbox{if } d=1,2. 
		\]	
	Therefore, \eqref{eq} together with \eqref{w_0} implies
	\[
	\widetilde{Z}(t) \le \frac{2e^{\frac{C_1t}{1+t}}}{(1+t)^a}(\widetilde{Z}(0)+\widetilde{Z}(0)^{\frac{1}{\tilde{\gamma}}-1})\leq \frac{4e^{\frac{C_1t}{1+t}}}{(1+t)^a}\widetilde{Z}(0) \quad \forall t\ge 0
	\]
due to $\w{Z}(t)= (\mathcal{Z}^2(t)+W(t))^{1/2} \leq \mathcal{Z}(t)$,
	and this concludes the desired global existence and temporal decay estimate of solutions.
	
\vspace{.2cm}
	
	\noindent $\bullet$ (Uniqueness):   As we presented in Section \ref{sec:3}, we show the stability in $L^2 \cap \dot{H}^{1 + \frac2\sigma}$ spaces. For $i=1,2$, let $(n_i, w_i)$ be the strong solution to \eqref{pER} with $\lambda =-1$ on time interval $[0,T]$ corresponding to the initial data $(n_{i,0}, w_{i,0})$.  We first see that
\[
\begin{aligned}
\frac12\frac{d}{dt}\|n_1 - n_2\|_{L^2}^2 &= -\intr (n_1 - n_2) (w_1 - w_2)\cdot \nabla n_2 \,dx -\tilde\gamma \intr (n_1-n_2)^2 \nabla \cdot (w_2 + v)\,dx\\
&\quad +\frac{1}{2} \intr (n_1-n_2)^2 \nabla \cdot (w_1 + v)\,dx  -\tilde\gamma \intr (n_1 - n_2)  n_1 \nabla \cdot (w_1 - w_2)\,dx\\
&\le C\lt( \|n_1 - n_2\|_{L^2}^2 + \|w_1 - w_2\|_{L^2}^2\rt)  -\tilde\gamma \intr (n_1 - n_2) n_1 \nabla \cdot (w_1 - w_2)\,dx.
\end{aligned}
\]
Then integration by parts gives

\[
\frac12\frac{d}{dt}\|n_1 - n_2\|_{L^2}^2 \le C\lt( \|n_1 - n_2\|_{L^2}^2 + \|w_1 - w_2\|_{L^2}^2\rt)  +\tilde\gamma \intr n_1 \nabla (n_1 - n_2)  \cdot (w_1 - w_2)\,dx.
\]
For $\|w_1-w_2\|_{L^2}$, we have

\[
\begin{aligned}
\frac12\frac{d}{dt}\|w_1 - w_2\|_{L^2}^2 &= -\intr (w_1 - w_2) \cdot [(w_1 - w_2)\cdot \nabla (w_2 + v)]\,dx - \intr (w_1-w_2)\cdot [(w_1 +v)\cdot \nabla( w_1 - w_2)]\,dx\\
&\quad -\tilde\gamma\intr (w_1-w_2) \cdot (n_1 -n_2)\cdot \nabla n_2 \,dx  -\tilde\gamma \intr n_1 (w_1-w_2) \cdot \nabla(n_1 - n_2)\,dx\\
&\quad -\intr (w_1 -w_2)\cdot \nabla \Lambda^{-\sigma} ((n_1)^{\frac{1}{\tilde\gamma}} - (n_2)^{\frac{1}{\tilde\gamma}})\,dx,
\end{aligned}
\]
and this deduces
\[
\begin{aligned}
\frac12\frac{d}{dt}\|w_1 - w_2\|_{L^2}^2 &\le C (\|w_1-w_2\|_{L^2}^2+ \|n_1-n_2\|_{L^2}^2) + \|w_1 - w_2\|_{L^2}\|\nabla\Lambda^{-\sigma}((n_1)^{\frac{1}{\tilde\gamma}} - (n_2)^{\frac{1}{\tilde\gamma}})\|_{L^2}\\
&\quad -\tilde\gamma \intr n_1 (w_1-w_2) \cdot \nabla(n_1 - n_2)\,dx.
\end{aligned}
\]
Hence, we have
\[
\begin{aligned}
	\frac{d}{dt}\lt(\|n_1 - n_2\|_{L^2}^2 + \|w_1-w_2\|_{L^2}^2 \rt) &\le C\|w_1-w_2\|_{L^2}^2 + C\|n_1 - n_2\|_{L^2}^2 \\
	&\quad +\|w_1 - w_2\|_{L^2}\|\nabla\Lambda^{-\sigma}((n_1)^{\frac{1}{\tilde\gamma}} - (n_2)^{\frac{1}{\tilde\gamma}})\|_{L^2}.
\end{aligned}
\]
Here, for any $\alpha\ge 0$, we can write
\bq\label{power_n}
\begin{aligned}
(n_1)^{\alpha} - (n_2)^{\alpha} &= \alpha \int_0^{n_1-n_2} (n_2 + \tau)^{\alpha-1}\,d\tau=\alpha(n_1-n_2)\int_0^1 (n_2 + (n_1 -n_2)\tau)^{\alpha-1}\,d\tau.
\end{aligned}
\eq
Once we set $F_{\tilde{\gamma}}(n_1, n_2) := (n_2 + (n_1 -n_2)\tau)^{\frac{1}{\tilde\gamma}-1}$, 
the similar argument employed in \eqref{diff_est_core2} gives
\bq\label{diff_est_core22}
\begin{aligned}
	\|\Lambda^{1-\sigma}[ (n_1 - n_2) F_{\tilde{\gamma}}(n_1, n_2)]\|_{L^2} &\lesssim \|\Lambda^{1-\sigma}(n_1-n_2)\|_{L^2}\|F_{\tilde{\gamma}}(n_1, n_2)\|_{L^\infty} \\
	&\quad + 
	\lt\{\begin{array}{lcl}\|n_1 -n_2\|_{L^\infty}\|\Lambda^{1-\sigma}F_{\tilde{\gamma}}(n_1, n_2)\|_{L^2} &\mbox{if} & d=1\\
	\|n_1 -n_2\|_{L^{\frac{2}{\delta}}}\|\Lambda^{1-\sigma}F_{\tilde{\gamma}}(n_1, n_2)\|_{L^{\frac{2}{1-\delta }}} &\mbox{if} & d=2\\
	\|n_1 -n_2\|_{L^{\frac{1}{\frac{1}{2}-\frac{1}{d}}}}\|\Lambda^{1-\sigma}F_{\tilde{\gamma}}(n_1, n_2)\|_{L^{d}} &\mbox{if} & d\ge 3\\
	\end{array}\rt.\\
	&\le C\|n_1 -n_2\|_{H^{1}},
\end{aligned}
\eq
where we used 
$$
\|\Lambda^{1-\sigma} F_{\tilde{\gamma}}(n_1, n_2)\|_{L^{d}} 
\lesssim \| F_{\tilde{\gamma}}(n_1, n_2)\|_{\dot{H}^{\frac{d}{2}-\sigma}}\\
\le C (\|n_1\|_{L^{\infty}} +\|n_2\|_{L^{\infty}}) (\|n_1\|_{\dot{H}^{\frac{d}{2}-\sigma}} +\|n_2\|_{\dot{H}^{\frac{d}{2}-\sigma}}) 
$$
which follows from Lemma \ref{tech_4} with the condition $\frac{d-1}{2}-\sigma<\frac{1}{\tilde{\gamma}}$ when $d\ge3$ and it holds under ($\mathcal{A}$3)\footnote{The lower dimensional cases $d=1,2$ can be easily shown in a similar way}.
Hence we get
\[\begin{aligned}
\| \nabla \Lambda^{-\sigma} ((n_1)^{\frac{1}{\tilde\gamma}} - (n_2)^{\frac{1}{\tilde\gamma}})\|_{L^2}
&\le C \int_0^1 \lt\| \Lambda^{1-\sigma} [ (n_1-n_2) ((1-\tau)n_2 + \tau n_1)^{\frac{1}{\tilde\gamma}-1}]\,\rt\|_{L^2} d\tau \le C\|n_1 - n_2\|_{H^{1}},
\end{aligned}\]
where we used the integral versions of Minkowski inequality.
Thus, we arrive at
\bq\label{low_diff_est2}
\frac{d}{dt}\lt(\|n_1 - n_2\|_{L^2}^2 + \|w_1-w_2\|_{L^2}^2 \rt) \le C\|w_1-w_2\|_{L^2}^2 + C\|n_1 - n_2\|_{H^1}^2.
\eq

For higher-order estimates, we estimate
\[\begin{aligned}
\frac12\frac{d}{dt}\|n_1-n_2\|_{\dot{H}^{1+\frac\sigma2}}^2 
&= -\intr \Lambda^{1+\frac\sigma2}(n_1-n_2) \Lambda^{1+\frac\sigma2}\lt((w_1 + v) \cdot \nabla (n_1 - n_2) \rt) dx\\
&\quad -\intr \Lambda^{1+\frac\sigma2}(n_1-n_2) \Lambda^{1+\frac\sigma2}\lt( (w_1 -w_2)\cdot \nabla n_2\rt) dx\\
&\quad -  \tilde\gamma \intr \Lambda^{1+\frac\sigma2}(n_1-n_2) \Lambda^{1+\frac\sigma2}\lt( n_1 \nabla\cdot(w_1 - w_2)\rt) dx\\
&\quad - \tilde\gamma \intr   \Lambda^{1+\frac\sigma2}(n_1-n_2) \Lambda^{1+\frac\sigma2}\lt((n_1 - n_2) \nabla \cdot (w_2  + v)\rt) dx\\
&=: \sum_{i=1}^4 K_{1i}.
\end{aligned}\]
For $K_{11}$, one uses the similar argument for \eqref{diff_est_core3} to get 
\[
K_{11} \le C\|n_1 - n_2\|_{H^{1+\frac\sigma2}}^2.
\]
For $K_{12}$ and $K_{14}$, we proceed similarly to \eqref{diff_est_core2} and obtain
\[
K_{12} \le C\|n_1-n_2\|_{H^{1+\frac\sigma2}}\|w_1 - w_2\|_{H^{1+\frac\sigma2}} \quad \mbox{and} \quad K_{14} \le C\|n_1-n_2\|_{H^{1+\frac\sigma2}}^2.
\]
For $K_{13}$, one uses  \eqref{moser_2term} and similar estimates for \eqref{diff_est_core3} to see
	\[
	\begin{aligned}
	K_{13}&= -\tilde{\gamma}\intr \Lambda^{1+\frac\sigma2}(n_1 - n_2) \Big( [ \Lambda^{1+\frac\sigma2},\,  n_1  ] \nabla  \cdot(w_1-w_2)\Big) \,dx\\
		&\quad  -\tilde{\gamma} \intr \Lambda^{1+\frac\sigma2}(n_1 - n_2)\,   n_1\, \nabla\cdot\Lambda^{1+\frac\sigma2} (w_1 -w_2)\,dx\\
		&\le C\|n_1 - n_2\|_{H^{1+\frac\sigma2}}\|w_1 - w_2\|_{H^{1+\frac\sigma2}}  -\tilde{\gamma} \intr \Lambda^{1+\frac\sigma2}(n_1 - n_2)\,   n_1 \, \nabla\cdot\Lambda^{1+\frac\sigma2} (w_1 -w_2)\,dx.
	\end{aligned}
	\]
Thus we gather the estimates for $K_{1i}$'s to attain

\bq\label{n_diff_high1}
\begin{aligned}
\frac{1}{2}\frac{d}{dt}\|n_1-n_2\|_{\dot{H}^{1+\frac\sigma2}}^2 &\le C\|n_1 - n_2\|_{H^{1+\frac\sigma2}}^2 + C\|w_1 - w_2\|_{H^{1+\frac\sigma2}}^2\\
&\quad  -\tilde{\gamma} \intr \Lambda^{1+\frac\sigma2}(n_1 - n_2)\,   n_1 \, \nabla\cdot\Lambda^{1+\frac\sigma2} (w_1 -w_2)\,dx.
\end{aligned}
\eq
Next, we estimate  $w_1 -w_2$ in $\dot{H}^{1+\frac\sigma2}$ as follows:

\[\begin{aligned}
\frac12\frac{d}{dt}\|w_1-w_2\|_{\dot{H}^{1+\frac\sigma2}}^2&= -\intr \Lambda^{1+\frac\sigma2} (w_1 -w_2) \cdot \Lambda^{1+\frac\sigma2} ((w_1 + v)\cdot \nabla (w_1- w_2) )\,dx\\
&\quad  -\intr \Lambda^{1+\frac\sigma2} (w_1 -w_2) \cdot \Lambda^{1+\frac\sigma2} ( (w_1-w_2)\cdot \nabla(w_2 + v))\,dx\\
&\quad -\tilde\gamma \intr \Lambda^{1+\frac\sigma2}(w_1-w_2) \cdot \Lambda^{1+\frac\sigma2}( n_1\, \nabla(n_1-n_2))\,dx\\
&\quad - \tilde\gamma \intr \Lambda^{1+\frac\sigma2}(w_1-w_2) \cdot \Lambda^{1+\frac\sigma2}((n_1 - n_2) \nabla n_2)\,dx\\
&\quad - \intr \Lambda^{1+\frac{\sigma}{2}}(w_1-w_2) \cdot \nabla \Lambda^{1-\frac{\sigma}{2}}((n_1)^{\frac{1}{\tilde\gamma}} - (n_2)^{\frac{1}{\tilde\gamma}} ))\,dx\\
&=: \sum_{i=1}^5 K_{2i}.
\end{aligned}\]
Here, we use the similar estimates for \eqref{diff_est_core2} and \eqref{diff_est_core3} to get
\[
\begin{aligned}
K_{21} \le C\|w_1 -w_2\|_{H^{1+\frac\sigma2}}^2, \quad K_{22} \le C\|w_1 -w_2\|_{H^{1+\frac\sigma2}}^2, \quad \mbox{and} \quad  K_{24}\le C\|w_1-w_2\|_{H^{1+\frac\sigma2}}\|n_1-n_2\|_{H^{1+\frac\sigma2}}.
\end{aligned}
\]
For $K_{23}$, we use the integration by parts, \eqref{moser_2term}, and similar arguments used in \eqref{diff_est_core3}  to obtain
\[\begin{aligned}
K_{23} &= -\tilde\gamma \intr \Lambda^{1+\frac\sigma2}(w_1-w_2) \cdot ([\Lambda^{1+\frac\sigma2},\, n_1]\, \nabla(n_1-n_2))\,dx +\tilde\gamma \intr \nabla n_1 \Lambda^{1+\frac\sigma2} (w_1 -w_2)\cdot  \Lambda^{1+\frac\sigma2}(n_1-n_2)\,dx\\
&\quad +\tilde\gamma \intr n_1  \nabla \cdot \Lambda^{1+\frac\sigma2}(w_1 -w_2)\, \Lambda^{1+\frac\sigma2}(n_1-n_2)\,dx\\
&\le  C\|w_1-w_2\|_{H^{1+\frac\sigma2}}\|n_1 - n_2\|_{H^{1+\frac\sigma2}} +\tilde\gamma \intr n_1   \nabla \cdot \Lambda^{1+\frac\sigma2} (w_1 -w_2)\, \Lambda^{1+\frac\sigma2}(n_1-n_2)\,dx.
\end{aligned}\]
For $K_{25}$, we see that
\bq\label{diff_power}
\begin{aligned}
	 \nabla ((n_1)^{\frac{1}{\tilde\gamma}} - (n_2)^{\frac{1}{\tilde\gamma}})
	&=  \frac{1}{\tilde{\gamma}}(n_1)^{\frac{1}{\tilde\gamma}-1}\nabla (n_1 - n_2) +  \frac{1}{\tilde{\gamma}}\big((n_1)^{\frac{1}{\tilde\gamma}-1} -  (n_2)^{\frac{1}{\tilde\gamma}-1}\big)\nabla n_2
\end{aligned}
\eq
so that 
\[\begin{aligned}
	K_{25}&=  \intr \nabla  \nabla \cdot (w_1-w_2) \cdot \nabla ((n_1)^{\frac{1}{\tilde\gamma}} - (n_2)^{\frac{1}{\tilde\gamma}} ))\,dx\\
	&=\frac{1}{\tilde\gamma}\intr (n_1)^{\frac{1}{\tilde\gamma}-1} \nabla \nabla \cdot(w_1 -w_2) \cdot \nabla (n_1 - n_2)\,dx\\
	&\quad - \frac{1}{\tilde\gamma}\intr  \nabla \cdot(w_1 -w_2) \nabla \cdot \lt[ ((n_1)^{\frac{1}{\tilde\gamma}-1} -  (n_2)^{\frac{1}{\tilde\gamma}-1})\nabla n_2\rt]\,dx.
\end{aligned}\]
From \eqref{power_n} with $\alpha=\frac{1}{\tilde{\gamma}}-1$, \eqref{diff_est_core22}, and $(\mathcal{A}3)$, we deduce that
\[\begin{aligned}
K_{25}\le \frac{1}{\tilde\gamma}\intr (n_1)^{\frac{1}{\tilde\gamma}-1} \nabla \nabla \cdot(w_1 -w_2) \cdot \nabla (n_1 - n_2)\,dx  + C\|w_1-w_2\|_{H^{1+\frac\sigma2}}\|n_1 -n_2\|_{H^{1+\frac\sigma2}}.
\end{aligned}\]
Then we collect all the estimates for $K_{2i}$'s to get

\bq\label{w_diff_high1}
\begin{aligned}
\frac{1}{2}\frac{d}{dt}\|w_1-w_2\|_{\dot{H}^{1+\frac\sigma2}}^2 &\le C\|w_1 - w_2\|_{H^{1+\frac\sigma2}}^2 + \|n_1 - n_2\|_{H^{1+\frac\sigma2}}^2\\
&\quad +\tilde\gamma\intr \Lambda^{1+\frac\sigma2}(n_1-n_2)\, n_1 \, \nabla\cdot \Lambda^{1+\frac\sigma2} (w_1 -w_2)\,dx\\
&\quad + \frac{1}{\tilde\gamma}\intr (n_1)^{\frac{1}{\tilde\gamma}-1} \nabla \nabla \cdot(w_1 -w_2) \cdot \nabla (n_1 - n_2)\,dx.
\end{aligned}
\eq
Now, we combine \eqref{n_diff_high1} with \eqref{w_diff_high1} to have

\bq\label{nw_comb_est1}
\begin{aligned}
\frac{d}{dt}\lt(\|n_1 -n_2\|_{\dot{H}^{1+\frac\sigma2}}^2 + \|w_1 -w_2\|_{\dot{H}^{1+\frac\sigma2}}^2 \rt) &\le C\lt(\|n_1 -n_2\|_{H^{1+\frac\sigma2}}^2 + \|w_1 -w_2\|_{H^{1+\frac\sigma2}}^2 \rt)\\
&\quad +  \frac{1}{\tilde\gamma}\intr (n_1)^{\frac{1}{\tilde\gamma}-1} \nabla \nabla \cdot(w_1 -w_2) \cdot \nabla (n_1 - n_2)\,dx.
\end{aligned}
\eq
To conclude the estimates, we additionally estimate 
\[\begin{aligned}
\frac12&\frac{d}{dt}\intr (n_1)^{\frac{1}{\tilde\gamma}-2} |\nabla(n_1 - n_2)|^2\,dx\\
&= -\frac12\lt(\frac{1}{\tilde\gamma}-2\rt)\intr (n_1)^{\frac{1}{\tilde\gamma}-3}\big((w_1 + v)\cdot \nabla n_1 + n_1 \cdot \nabla(w_1 + v) \big) |\nabla (n_1 - n_2)|^2\,dx\\
&\quad -\intr (n_1)^{\frac{1}{\tilde\gamma}-2} \nabla (n_1 - n_2)\cdot \nabla \lt((w_1 +v)\cdot \nabla (n_1 - n_2) + (w_1 - w_2)\cdot \nabla n_2\rt)\,dx\\
&\quad -\tilde{\gamma}\intr (n_1)^{\frac{1}{\tilde\gamma}-2} \nabla (n_1 - n_2)\cdot \nabla \lt(n_1 \nabla\cdot(w_1 -w_2) + (n_1-n_2)\cdot \nabla(w_2 + v)\rt)\,dx\\
&\le C\|n_1-n_2\|_{H^1}^2 + C\|w_1-w_2\|_{H^1}^2 -\tilde{\gamma} \intr (n_1)^{\frac{1}{\tilde\gamma}-1}\nabla (n_1 -n_2)\cdot \nabla \nabla\cdot(w_1-w_2)\,dx,
\end{aligned}\]
where we used the similar arguments for \eqref{n_diff_high1}.\\

\noindent We combine this with \eqref{nw_comb_est1} to arrive at
\[\begin{aligned}
\frac{d}{dt}&\lt(\|n_1 -n_2\|_{\dot{H}^{1+\frac\sigma2}}^2  + \|w_1 -w_2\|_{\dot{H}^{1+\frac\sigma2}}^2 + \frac{1}{\tilde\gamma^2} \intr (n_1)^{\frac{1}{\tilde\gamma}-2}|\nabla(n_1 - n_2)|^2\,dx \rt)\\
&\le C\|n_1 -n_2\|_{H^{1+\frac\sigma2}}^2 + C\|w_1 -w_2\|_{H^{1+\frac\sigma2}}^2,
\end{aligned}\]
and together with \eqref{low_diff_est2}, we obtain 
\[\begin{aligned}
\frac{d}{dt}&\lt(\|n_1 -n_2\|_{H^{1+\frac\sigma2}}^2  + \|w_1 -w_2\|_{H^{1+\frac\sigma2}}^2 + \frac{1}{\tilde\gamma^2} \intr (n_1)^{\frac{1}{\tilde\gamma}-2}|\nabla(n_1 - n_2)|^2\,dx \rt)\\
&\le C\|n_1 -n_2\|_{H^{1+\frac\sigma2}}^2 + C\|w_1 -w_2\|_{H^{1+\frac\sigma2}}^2,
\end{aligned}\]
and it implies the desired uniqueness result.
\end{proof}

%%%%%%%%%%%%%%%%%%%%%%%%%%%%%%%%%%%%%%%%%%%%%%%%%%%%%%%%%%%%%%%%%%%%%%%%%%%%%%%%%5
%
%
%                        Section: Introduction 
%
%
%%%%%%%%%%%%%%%%%%%%%%%%%%%%%%%%%%%%%%%%%%%%%%%%%%%%%%%%%%%%%%%%%%%%%%%%%%%%%%%%%

\subsection{Attractive interaction case}\label{ssec:attpo}
In this part, we consider the Euler system with attractive super-Manev interactions, i.e. \eqref{pER} with $\lambda = 1$ and $\sigma \in (0,1)$. For the same reason as in the repulsive case, we only provide the highest-order estimates of solutions. 

\begin{proposition}\label{Hs_est2_case2}
Let the assumptions of Theorem \ref{main_thm4} be satisfied. 
Let $(n,w)$ be a smooth solution to \eqref{pER} with $\lambda=1$ decaying sufficiently fast at infinity. Then we have
	\[
\begin{aligned}
&	\frac{1}{2} \frac{d}{dt}\lt( \w{X}_s^2
	+\sum_{k=1}^{k_0} \frac{1}{\tilde{\gamma}^{2k}} \intr n^{k(\frac{1}{\tilde{\gamma}}-2)} |\Lambda^{s-\frac{k\sigma}{2}} w|^2\, dx\rt)\\ 
	&\quad  +\frac{s-\frac{d}{2}+\min\{ d\tilde{\gamma}, 1\}}{1+t} \lt( \w{X}_s^2
	+\sum_{k=1}^{k_0} \frac{1}{\tilde{\gamma}^{2k}} \intr n^{k(\frac{1}{\tilde{\gamma}}-2)} |\Lambda^{s-\frac{k\sigma}{2}} w|^2\, dx\rt) \\
	&\qquad \lesssim
	(\|\nabla w\|_{L^\infty}+ \|\nabla n\|_{L^\infty}) \w{X}_s^2 
	+ \frac{\w{X}_s^2}{(1+t)^2} \\
	&\qquad \quad + \lt( \frac{\|n\|_{L^\infty}+\|w\|_{L^\infty}}{(1+t)^{2+s-\frac d2}}+ \frac{\|n\|_{\dot{H}^{s-1}} +\|w\|_{\dot{H}^{s-1}}}{(1+t)^3} + \frac{\|\nabla n\|_{L^\infty}+\|\nabla w\|_{L^\infty}}{(1+t)^{s+1-\frac d2}}\rt)\w{X}_s\\
	&\qquad\quad +\|\nabla n\| \sum_{k=1}^{k_0}\|n\|_{L^{\infty}}^{k(\frac{1}{\tilde{\gamma}}-2)} \lt(
	\|w\|_{\dot{H}^{s-\frac{(k-1) \sigma}{2}}} \|n\|_{\dot{H}^{s-\frac{(k+1)\sigma}{2}}} +  \|w\|_{\dot{H}^{s-\frac{k \sigma}{2}}} \|n\|_{\dot{H}^{s-\frac{k\sigma}{2}}}\rt)\\
	&\qquad\quad + \|n\|_{\dot{H}^{s-\sigma}}^{\frac{\frac d2 +1-\sigma}{s-\sigma}} \|n\|_{L^2}^{1-\frac{\frac d2+1-\sigma}{s-\sigma}} \sum_{k=1}^{k_0} \|n\|_{L^\infty}^{k\lt(\frac{1}{\tilde\gamma}-2\rt)} \|w\|_{\dot{H}^{s-\frac{(k-1)\sigma}{2}}} \|n\|_{\dot{H}^{s-\frac{(k-1)\sigma}{2}}}\\
	&\qquad\quad + \sum_{k=1}^{k_0}\|n\|_{L^{\infty}}^{k(\frac{1}{\tilde{\gamma}}-2)} \|w\|_{\dot{H}^{s-\frac{k\sigma}{2} }} \lt( \frac{\|w\|_{\dot{H}^{s-\frac{k\sigma}{2}}}}{(1+t)^2}+\frac{\|\nabla w\|_{L^\infty}}{(1+t)^{s-\frac{k\sigma}{2}+1 - \frac d2}} +\frac{\|w\|_{\dot{H}^{s-\frac{k \sigma}{2}-1}}}{(1+t)^3} \rt)\\
	&\qquad\quad + \|n\|_{L^\infty}^{(k_0+1)(\frac{1}{\tilde\gamma}-2)}\|\nabla n\|_{L^\infty}\|w\|_{\dot{H}^{s-\frac{k_0\sigma}{2}}}  \|n\|_{\dot{H}^{s-\frac{(k_0+2)\sigma}{2}}}\\
		&\qquad\quad +\|n\|_{L^\infty}^{(k_0+1)(\frac{1}{\tilde\gamma}-2)+1}\|w\|_{\dot{H}^{s-\frac{k_0\sigma}{2}}}\|n\|_{\dot{H}^{s+1-\frac{(k_0+2)\sigma}{2}}},
	\end{aligned}
\]
	where
	$\w{X}_s^2 := \|w\|_{\dot{H}^{s}}^2 + \|n\|_{\dot{H}^s}^2$ and $k_0$ is the minimum integer value of $k$ such that 
	\[
	\frac{2(1-\sigma)}{\sigma} <k \leq \frac{2(s-2)}{\sigma}.
	\]
\end{proposition}

To handle the attractive super-Manev interaction potential, we shall use the regularity-tuning iteration argument proposed in \cite{CJLpre}. We write
\[
\mathcal{R}_{\ell}(k):=  \intr n^{k(\frac{1}{\tilde{\gamma}}-2)} \Lambda^{\ell-\frac{k \sigma}{2}}w \cdot \Lambda^{\ell -\frac{k \sigma}{2}-\sigma}\nabla (n^{\frac{1}{\tilde{\gamma}}}) \, dx
\]
and
\[ 
\mathcal{P}_{\ell}(k):= \tilde{\gamma} \intr n^{k(\frac{1}{\tilde{\gamma}}-2)}\Lambda^{\ell-\frac{k\sigma}{2}} w \cdot \Lambda^{\ell-\frac{k\sigma}{2}} ( n\nabla n)\, dx
\]
for any integer $ 0\leq k < \frac{2\ell}\sigma$ and $0<\sigma<1$.

Then we need to show the following two lemmas. Since their proofs are rather technical, for the smoothness of reading, we postpone them to Appendices \ref{app.C} and \ref{app.D}, respectively.

\begin{lemma}\label{Gap}
 	Let the assumptions of Theorem \ref{main_thm4} be satisfied.
 	Then we have
	\[ 
	\begin{aligned}
\begin{aligned}
	\lt|\mathcal{R}_s (k) - \frac{1}{\tilde{\gamma}^2} \mathcal{P}_s (k+1) \rt|
	&\lesssim \|n\|_{L^{\infty}}^{(k+1)(\frac{1}{\tilde{\gamma}}-2)} 
	\|\nabla n\|_{L^{\infty}} \|w\|_{\dot{H}^{s-\frac{k \sigma}{2}}} \|n\|_{\dot{H}^{s-\frac{(k+2)\sigma}{2}}}\\
	&\quad + \|n\|_{L^{\infty}}^{(k+1)(\frac{1}{\tilde{\gamma}}-2)} 
	\|\nabla n\|_{L^{\infty}} \|w\|_{\dot{H}^{s-\frac{(k+1) \sigma}{2}}} \|n\|_{\dot{H}^{s-\frac{(k+1)\sigma}{2}}}\\
	&\quad +\|n\|_{L^\infty}^{(k+1)(\frac{1}{\tilde\gamma}-2)} \|n\|_{\dot{H}^{s-\sigma}}^{\frac{\frac d2 +1-\sigma}{s-\sigma}} \|n\|_{L^2}^{1-\frac{\frac d2+1-\sigma}{s-\sigma}}\|w\|_{\dot{H}^{s-\frac{k\sigma}{2}}} \|n\|_{\dot{H}^{s-\frac{k\sigma}{2}}}
\end{aligned}
	\end{aligned}
	\]
	for any integer $k \in [0,\, \frac{2s}\sigma-2]$.
\end{lemma}

\begin{lemma} \label{higher_it}
	Let the assumptions of Theorem \ref{main_thm4} be satisfied.
	Then we have
\[
\begin{aligned}
	\frac{1}{2} &\frac{d}{dt} \intr n^{k(\frac{1}{\tilde{\gamma}}-2)} |\Lambda^{s-\frac{k\sigma}{2}} w|^2\, dx 
	+\frac{s-\frac{d}{2}+1 }{1+t}  
	\intr n^{k(\frac{1}{\tilde{\gamma}}-2)}  |\Lambda^{s-\frac{k\sigma}{2}} w|^2\, dx \\ 
	& \le C\|n\|_{L^{\infty}}^{k(\frac{1}{\tilde{\gamma}}-2)}  \|\nabla w  \|_{L^{\infty}} \|w\|_{\dot{H}^{s-\frac{k\sigma}{2}}}^2 \\
	&\quad +C\|n\|_{L^{\infty}}^{k(\frac{1}{\tilde{\gamma}}-2)} \|w\|_{\dot{H}^{s-\frac{k \sigma}{2} }} \lt( \frac{\|w\|_{\dot{H}^{s-\frac{k\sigma}{2}}}}{(1+t)^2}+\frac{\|\nabla w\|_{L^\infty}}{(1+t)^{s-\frac{k \sigma}{2}+1 -\frac d2}} +\frac{\|w\|_{\dot{H}^{s-\frac{k \sigma}{2}-1}}}{(1+t)^3} \rt)\\
	&\quad -\mathcal{P}_s(k) + \mathcal{R}_s(k)
\end{aligned}
\]
	for any integer $k \in [1, \, \frac{2(s-1)}\sigma)$.
\end{lemma}

Now, we are ready to show the proof of Proposition \ref{Hs_est2_case2}.
\begin{proof}[Proof of Proposition \ref{Hs_est2_case2}]
From Lemma \ref{pres_est1} with $\lambda=1$, let us recall that 
	\bq\label{wp_high_est0}
	\begin{aligned}
		\frac{1}{2} &\frac{d}{dt} \w{X}_s^2 +\frac{s-\frac{d}{2}+\min\{ d\tilde{\gamma}, 1\}}{1+t} \w{X}_s^2   \\
		&\lesssim
		(\|\nabla w\|_{L^\infty}+ \|\nabla n\|_{L^\infty}) \w{X}_s^2 
		+ \frac{\w{X}_s^2}{(1+t)^2} \\
		& \quad + \lt( \frac{\|n\|_{L^\infty}+\|w\|_{L^\infty}}{(1+t)^{2+s-\frac d2}}+ \frac{\|n\|_{\dot{H}^{s-1}} +\|w\|_{\dot{H}^{s-1}}}{(1+t)^3} + \frac{\|\nabla n\|_{L^\infty}+\|\nabla w\|_{L^\infty}}{(1+t)^{s+1-\frac d2}}\rt)\w{X}_s+ \mathcal{R}_s(0).
	\end{aligned}
	\eq
 Now we shall use the iteration argument based on Lemmas \ref{Gap} and \ref{higher_it}.
 First, we apply Lemma \ref{Gap} with $k=0$ and then Lemma \ref{higher_it} with $k=1$ so that
 	\[\begin{aligned}
 	\mathcal{R}_{s}(0)
 		&\leq C\|n\|_{L^{\infty}}^{(\frac{1}{\tilde{\gamma}}-2)} 
	\|\nabla n\|_{L^{\infty}} \|w\|_{\dot{H}^{s}} \|n\|_{\dot{H}^{s-\sigma}} + C\|n\|_{L^{\infty}}^{(\frac{1}{\tilde{\gamma}}-2)} 
	\|\nabla n\|_{L^{\infty}} \|w\|_{\dot{H}^{s-\frac{ \sigma}{2}}} \|n\|_{\dot{H}^{s-\frac{\sigma}{2}}}\\
	&\quad +C\|n\|_{L^\infty}^{(\frac{1}{\tilde\gamma}-2)} \|n\|_{\dot{H}^{s-\sigma}}^{\frac{\frac d2 +1-\sigma}{s-\sigma}} \|n\|_{L^2}^{1-\frac{\frac d2+1-\sigma}{s-\sigma}}\|w\|_{\dot{H}^{s}} \|n\|_{\dot{H}^{s}} + \frac{1}{\tilde{\gamma}^2} \mathcal{P}_s (1)\\
 	&\leq C\|n\|_{L^{\infty}}^{(\frac{1}{\tilde{\gamma}}-2)} 
	\|\nabla n\|_{L^{\infty}} \|w\|_{\dot{H}^{s}} \|n\|_{\dot{H}^{s-\sigma}}  + C\|n\|_{L^{\infty}}^{(\frac{1}{\tilde{\gamma}}-2)} 
	\|\nabla n\|_{L^{\infty}} \|w\|_{\dot{H}^{s-\frac{ \sigma}{2}}} \|n\|_{\dot{H}^{s-\frac{\sigma}{2}}}\\
	&\quad +C\|n\|_{L^\infty}^{(\frac{1}{\tilde\gamma}-2)} \|n\|_{\dot{H}^{s-\sigma}}^{\frac{\frac d2 +1-\sigma}{s-\sigma}} \|n\|_{L^2}^{1-\frac{\frac d2+1-\sigma}{s-\sigma}}\|w\|_{\dot{H}^{s}} \|n\|_{\dot{H}^{s}}  + C\|n\|_{L^\infty}^{(\frac{1}{\tilde\gamma}-2)}\|\nabla w\|_{L^\infty}\|w\|_{\dot{H}^{s-\frac\sigma2}}^2\\
	 &\quad +C\|n\|_{L^{\infty}}^{\frac{1}{\tilde{\gamma}}-2} \|w\|_{\dot{H}^{s-\frac{\sigma}{2} }} \lt( \frac{\|w\|_{\dot{H}^{s-\frac{\sigma}{2}}}}{(1+t)^2}+\frac{\|\nabla w\|_{L^\infty}}{(1+t)^{s-\frac\sigma2+1 -\frac d2}} +\frac{\|w\|_{\dot{H}^{s-\frac{ \sigma}{2}-1}}}{(1+t)^3} \rt) \\
 	&\quad
 	-\frac{1}{2\tilde{\gamma}^2} \frac{d}{dt} \intr n^{\frac{1}{\tilde{\gamma}}-2} |\Lambda^{s-\frac{\sigma}{2}} w|^2\, dx  -\frac{ s -\frac{d}{2}+1 }{\tilde{\gamma}^2 (1+t)}  
 	\intr n^{\frac{1}{\tilde{\gamma}}-2}  |\Lambda^{s-\frac{\sigma}{2}} w|^2\, dx\\
 	&\quad 		+ \frac{1}{\tilde{\gamma}^{2}} \mathcal{R}_s(1).
 \end{aligned}\]
We will repeat this process until the order of derivative on $n^{\frac{1}{\tilde\gamma}}$ in $\mathcal{R}_s$, which is $s-\frac{(k+2)\sigma}{2} +1$, becomes less than $s$. Thus, we let $k_0$ be the smallest integer satisfying  
\[
\frac{2(1-\sigma)}{\sigma} \le k_0 < \frac{2(s-1)}{\sigma} \quad \mbox{for } 3-\sigma<s. \]
Once we repeat the process up to $k_0$ times, we arrive at
 	\bq\label{Rs0_est}
	\begin{aligned}
		\mathcal{R}_{s}(0)
		&\leq C\|\nabla n\| \sum_{k=1}^{k_0}\|n\|_{L^{\infty}}^{k(\frac{1}{\tilde{\gamma}}-2)} \lt(
	\|w\|_{\dot{H}^{s-\frac{(k-1) \sigma}{2}}} \|n\|_{\dot{H}^{s-\frac{(k+1)\sigma}{2}}} +  \|w\|_{\dot{H}^{s-\frac{k \sigma}{2}}} \|n\|_{\dot{H}^{s-\frac{k\sigma}{2}}}\rt)\\
	&\quad + C\|n\|_{\dot{H}^{s-\sigma}}^{\frac{\frac d2 +1-\sigma}{s-\sigma}} \|n\|_{L^2}^{1-\frac{\frac d2+1-\sigma}{s-\sigma}} \sum_{k=1}^{k_0} \|n\|_{L^\infty}^{k\lt(\frac{1}{\tilde\gamma}-2\rt)} \|w\|_{\dot{H}^{s-\frac{(k-1)\sigma}{2}}} \|n\|_{\dot{H}^{s-\frac{(k-1)\sigma}{2}}}\\
	&\quad +C \sum_{k=1}^{k_0}\|n\|_{L^{\infty}}^{k(\frac{1}{\tilde{\gamma}}-2)} \|w\|_{\dot{H}^{s-\frac{k\sigma}{2} }} \lt( \frac{\|w\|_{\dot{H}^{s-\frac{k\sigma}{2}}}}{(1+t)^2}+\frac{\|\nabla w\|_{L^\infty}}{(1+t)^{s-\frac{k\sigma}{2}+1 - \frac d2}} +\frac{\|w\|_{\dot{H}^{s-\frac{k \sigma}{2}-1}}}{(1+t)^3} \rt) \\
		&\quad 
		-\sum_{k=1}^{k_0}\Bigg( \frac{1}{2\tilde{\gamma}^{2k}} \frac{d}{dt} \intr n^{k(\frac{1}{\tilde{\gamma}}-2)} |\Lambda^{s-\frac{k\sigma}{2}} w|^2\, dx 
		+ \frac{s-\frac{d}{2} +1}{\tilde{\gamma}^{2k}(1+t)} 
		\intr  n^{k(\frac{1}{\tilde{\gamma}}-2)}  |\Lambda^{s-\frac{k\sigma}{2}} w|^2\, dx \Bigg)\\ 
		 &\quad  + \frac{1}{\tilde{\gamma}^{2k_0}} \mathcal{R}_s(k_0).
	\end{aligned}
	\eq
Since $s -\frac{(k_0+2)\sigma}{2} +1 \le s$, we now use \eqref{tech_1} and Lemma \ref{tech_4} to get
	\bq\label{Rsk_0_est}
	\begin{aligned}
		\mathcal{R}_{s}(k_0) 
		&=  \frac{1}{\tilde{\gamma}}\intr n^{k_0(\frac{1}{\tilde{\gamma}}-2)} \Lambda^{s-\frac{k_0 \sigma}{2}}w \cdot \Lambda^{s -\frac{(k_0+2) \sigma}{2}}(n^{\frac{1}{\tilde{\gamma}}-1} \nabla n) \, dx\\
		&= \frac{1}{\tilde{\gamma}} \intr n^{k_0(\frac{1}{\tilde{\gamma}}-2)} \Lambda^{s-\frac{k_0 \sigma}{2}} w \cdot [\Lambda^{s-\frac{(k_0+2) \sigma}{2}}, n^{\frac{1}{\tilde{\gamma}}-1}] \nabla n\, dx \\
		&\quad+ \frac{1}{\tilde{\gamma}} \intr n^{k_0(\frac{1}{\tilde{\gamma}}-2)+\frac{1}{\tilde{\gamma}}-1}  \Lambda^{s-\frac{k_0 \sigma}{2}} w \cdot \Lambda^{s-\frac{(k_0+2) \sigma}{2}} \nabla n\, dx\\
		&\lesssim \|n\|_{L^\infty}^{k_0(\frac{1}{\tilde\gamma}-2)}\|w\|_{\dot{H}^{s-\frac{k_0\sigma}{2}}} \lt(\|n^{\frac{1}{\tilde\gamma}}\|_{\dot{H}^{s-\frac{(k_0+2)\sigma}{2}}} \|\nabla n\|_{L^\infty} + \|\nabla (n^{\frac{1}{\tilde\gamma}})\|_{L^\infty}\|n\|_{\dot{H}^{s-\frac{(k_0+2)\sigma}{2}}} \rt)\\
		&\quad + \|n\|_{L^\infty}^{(k_0+1)(\frac{1}{\tilde\gamma}-2)+1}\|w\|_{\dot{H}^{s-\frac{k_0\sigma}{2}}}\|n\|_{\dot{H}^{s+1-\frac{(k_0+2)\sigma}{2}}}\\
		&\lesssim \|n\|_{L^\infty}^{(k_0+1)(\frac{1}{\tilde\gamma}-2)}\|\nabla n\|_{L^\infty}\|w\|_{\dot{H}^{s-\frac{k_0\sigma}{2}}}  \|n\|_{\dot{H}^{s-\frac{(k_0+2)\sigma}{2}}}\\
		&\quad +\|n\|_{L^\infty}^{(k_0+1)(\frac{1}{\tilde\gamma}-2)+1}\|w\|_{\dot{H}^{s-\frac{k_0\sigma}{2}}}\|n\|_{\dot{H}^{s+1-\frac{(k_0+2)\sigma}{2}}}.
	\end{aligned}
	\eq
Therefore, we combine \eqref{wp_high_est0}, \eqref{Rs0_est}, and \eqref{Rsk_0_est} to get the desired result.
\end{proof}

Following the notations in \eqref{temp_notation}-\eqref{temp_notation_2}, we combine Lemma \ref{pres_est1} with Proposition \ref{Hs_est2_case2} to deduce the following a priori estimates.  
\begin{proposition}\label{decay_temp2_case3}
	Let the assumptions of Theorem \ref{main_thm4} be satisfied.
	For $T>0$, suppose that $(n,w)$ is a smooth solution to \eqref{pER} with $\lambda=1$ on the time interval $[0,T]$ decaying sufficiently fast at infinity.
	Then we have
\[
\begin{aligned}
	\frac{1}{2} &\frac{d}{dt} (\w{Z}^2(t)+\sum_{k=1}^{k_0} \w{W}_k(t) ) 
	+\frac{\w{C}_{d, \sigma}}{1+t}\lt( \w{Z}^2(t)   
	+ \sum_{k=1}^{k_0}  \w{W}_k(t)\rt)\\
	&\lesssim
	\w{Z}^3 (t)
	+ \frac{\w{Z}^2(t)}{(1+t)^2}  + \sum_{k=1}^{k_0}  \frac{ \w{Z}^{2+k(\frac{1}{\tilde{\gamma}}-2)}(t) }{(1+t)^{2+k(2-\sigma-\frac{1}{\tilde{\gamma}})}}  
	+  \sum_{k=1}^{k_0+1}  \frac{\w{Z}^{3+k(\frac{1}{\tilde{\gamma}}-2)}(t) }{(1+t)^{k(2-\sigma-\frac{1}{\tilde{\gamma}})}}
\end{aligned}
\]
with $\w{C}_{d, \sigma}=1+\min\{1,\, d\tilde{\gamma}\}$,
and $k_0$ given as in Proposition \ref{Hs_est2_case2}.
Here,
	\[
	\w{W}_k(t):= \frac{(1+t)^{2(s-\frac{d}{2}-1)}}{\tilde{\gamma}^{2k}} \intr n^{k(\frac{1}{\tilde{\gamma}}-2)} |\Lambda^{s-\frac{k\sigma}{2}} w|^2\, dx>0.
	\]
\end{proposition}
\begin{proof}
	From Lemma \ref{pres_est1},  it is again clear that
	\bq\label{case3_lower_est}
\frac{d}{dt}\lt(\w{n}_2 + \w{w}_2 \rt) + \frac{C_{d,\sigma}}{1+t} \lt(\w{n}_2 + \w{w}_2\rt) \lesssim \tilde{Z}^2(t) + \frac{\tilde{Z}^{\frac{1}{\tilde\gamma}}}{(1+t)^{2-\sigma-\frac{1}{\tilde\gamma}}}.
\eq

For the highest-order estimate, we first note that Lemma \ref{bdds2} still holds. Then we can rewrite the results in Lemma \ref{Hs_est2_case2} as
\bq \label{Y_s}
\begin{aligned}
&	\frac{1}{2} \frac{d}{dt} \lt(\w{Y}_s^2 + \sum_{k=1}^{k_0} \frac{(1+t)^{2(s-\frac{d}{2}-1)}}{\tilde{\gamma}^{2k}} \intr n^{k(\frac{1}{\tilde{\gamma}}-2)} |\Lambda^{s-\frac{k\sigma}{2}} w|^2\, dx\rt)\\ 
&\quad +\frac{1+\min\{ d\tilde{\gamma}, 1\}}{1+t} \w{Y}_s^2   + \sum_{k=1}^{k_0}\frac{2(1+t)^{2(s-\frac{d}{2}-1)}  }{\tilde{\gamma}^{2k}(1+t)}  
\intr  n^{k(\frac{1}{\tilde{\gamma}}-2)}  |\Lambda^{s-\frac{k\sigma}{2}} w|^2\, dx \\
&\qquad \lesssim
(\w{w}_{1, \infty}+ \w{n}_{1, \infty} ) \w{Y}_s^2 
+ \frac{\w{Y}_s^2}{(1+t)^2} + \lt( \frac{  \w{n}_{\infty} +\w{w}_{\infty}+  \w{n}_{s-1, 2}+\w{w}_{s-1, 2} + \w{n}_{1, \infty} +\w{w}_{1, \infty}}{(1+t)^2}\rt)\w{Y}_s\\
&\qquad \quad +\w{n}_{1,\infty} \sum_{k=1}^{k_0}\frac{\w{n}_\infty^{k(\frac{1}{\tilde\gamma}-2)}\lt( \w{w}_{s-\frac{(k-1)\sigma}{2},2}\w{n}_{s-\frac{(k+1)\sigma}{2},2} + \w{w}_{s-\frac{k\sigma}{2},2}\w{n}_{s-\frac{k\sigma}{2},2}\rt)}{(1+t)^{k(2-\sigma-\frac{1}{\tilde\gamma})}}\\
&\qquad \quad  + \w{n}_{s-\sigma,2}^{\frac{\frac d2 +1-\sigma}{s-\sigma}} \w{n}_2^{1-\frac{\frac d2 +1-\sigma}{s-\sigma}} \sum_{k=1}^{k_0}\frac{\w{n}_\infty^{k(\frac{1}{\tilde\gamma}-2)}\w{w}_{s-\frac{(k-1)\sigma}{2},2}\ \w{n}_{s-\frac{(k-1)\sigma}{2},2}}{(1+t)^{k(2-\sigma-\frac{1}{\tilde\gamma})}}\\
&\qquad \quad + \sum_{k=1}^{k_0} \frac{\w{n}_\infty^{k(\frac{1}{\tilde\gamma}-2)} \w{w}_{s-\frac{k\sigma}{2},2} \lt(\w{w}_{s-\frac{k\sigma}{2},2} + \w{w}_{1,\infty} + \w{w}_{s-\frac{k\sigma}{2}-1,2 } \rt) }{(1+t)^{2+k(2-\sigma-\frac{1}{\tilde\gamma})}}\\
&\qquad \quad + \frac{\w{n}_\infty^{(k_0+1)(\frac{1}{\tilde\gamma}-2)} \w{n}_{1,\infty} \w{w}_{s-\frac{k_0\sigma}{2},2} \ \w{n}_{s-\frac{(k_0+2)\sigma}{2},2 } }{(1+t)^{(k_0+1)(2-\sigma-\frac{1}{\tilde\gamma})}}
	 +\frac{\w{n}_\infty^{(k_0+1)(\frac{1}{\tilde\gamma}-2)+1} \w{w}_{s-\frac{k_0\sigma}{2},2} \ \w{n}_{s+1-\frac{(k_0+2)\sigma}{2},2 } }{(1+t)^{(k_0+1)(2-\sigma-\frac{1}{\tilde\gamma})}}.
\end{aligned}
\eq
Then, it follows from \eqref{case3_lower_est} and \eqref{Y_s} that

\[
\begin{aligned}
	\frac{1}{2} &\frac{d}{dt} (\w{Z}^2+\sum_{k=1}^{k_0} \w{W}_k ) 
	+\frac{1+\min\{ d\tilde{\gamma}, 1\}}{1+t}\lt( \w{Z}^2   
	+ \sum_{k=1}^{k_0}  \w{W}_k\rt)\\
	&\lesssim
	\w{Z}^3 
	+ \frac{\w{Z}^2}{(1+t)^2}  + \sum_{k=1}^{k_0}  \frac{ \w{Z}^{2+k(\frac{1}{\tilde{\gamma}}-2)} }{(1+t)^{2+k(2-\sigma-\frac{1}{\tilde{\gamma}})}}  
	+  \sum_{k=1}^{k_0+1}  \frac{\w{Z}^{3+k(\frac{1}{\tilde{\gamma}}-2)} }{(1+t)^{k(2-\sigma-\frac{1}{\tilde{\gamma}})}}.
\end{aligned}
\]
This concludes the desired estimate.
\end{proof}

\begin{proof}[Proof of Theorem \ref{main_thm4}] Similarly to previous theorems, we split the proof into two parts.

\vspace{.2cm}
	
	\noindent $\bullet$ (Global-in-time existence and temporal decay): For the local-in-time existence, we also proceed as in Section \ref{sec:3}. For the global existence, since $\w{Z}(t), \w{W}_k(t)\ge0$ for any $t\in [0, T]$ and $k\in [1, k_0]$, 
by Proposition \ref{decay_temp2_case3} we deduce
			\[
	\begin{aligned}
		\frac{d}{dt} \w{\mathcal{Z}}(t)
		+\frac{\w{C}_{d, \sigma}}{1+t} \w{\mathcal{Z}}(t)    \leq C_2 \lt(
		\w{\mathcal{Z}}^2(t)
		+ \frac{\w{\mathcal{Z}}(t)}{(1+t)^2}  + \sum_{k=1}^{k_0}  \frac{ \w{\mathcal{Z}}(t)^{1+k(\frac{1}{\tilde{\gamma}}-2)} }{(1+t)^{2+k(2-\sigma-\frac{1}{\tilde{\gamma}})}}  
		+  \sum_{k=1}^{k_0+1}  \frac{\w{\mathcal{Z}}^{2+k(\frac{1}{\tilde{\gamma}}-2)} }{(1+t)^{k(2-\sigma-\frac{1}{\tilde{\gamma}})}} \rt),
	\end{aligned} 
	\]
where $\w{\mathcal{Z}}(t):=(\w{Z}^2(t)+\sum_{k=1}^{k_0} \w{W}_k(t))^{\frac12} $.
Here $C_2>0$ is independent of $t$.
	If 
	 $\w{Z}(0) \leq \varepsilon$ for sufficient small $\varepsilon>0$,
	 then 
	 \[
	 \w{\mathcal{Z}}(0) = (\w{Z}^2(0)+\sum_{k=1}^{k_0} \w{W}_k(0))^{\frac{1}{2}}  \leq
	 \w{Z}(0)+\sum_{k=1}^{k_0} \w{W}_k^{\frac{1}{2}}(0) \leq  (k_0+1) \w{Z}(0) \leq (k_0+1)\varepsilon
	 \]
	 since 
	 \[
	 \begin{aligned}
	 	\sum_{k=1}^{k_0}\w{W}_k(0) 
	 	&=\sum_{k=1}^{k_0}\frac{1}{\tilde{\gamma}^{2k}} \intr n_0^{k(\frac{1}{\tilde{\gamma}}-2)} |\Lambda^{s-\frac{k\sigma}{2}} w_0|^2\, dx \\
	 	&\lesssim \sum_{k=1}^{k_0} \|n_0\|_{L^{\infty}}^{k(\frac{1}{\tilde{\gamma}}-2)} \|w_0\|_{\dot{H}^{s-\frac{k\sigma}{2}}}^2\\
	 	&\lesssim \sum_{k=1}^{k_0} \|n_0\|_{L^{\infty}}^{k(\frac{1}{\tilde{\gamma}}-2)} \|w_0\|_{\dot{H}^{1}}^{2-\frac{2(s-1-\frac{k\sigma}{2})}{s-1}}\|w_0\|_{\dot{H}^{s}}^{\frac{2(s-1-\frac{k\sigma}{2})}{s-1}} \\
	 	&\lesssim \sum_{k=1}^{k_0} \w{Z}^{2+k(\frac{1}{\tilde{\gamma}}-2)}(0)\lesssim \w{Z}^2(0)
	 \end{aligned}
	 \]
	 for $k\leq k_0 <\frac{2(s-1)}{\sigma}$.
	Then we use Lemma \ref{tech_6} with $\widetilde{Z}(t)$ replaced by $\w{\mathcal{Z}}(t)$ and $a=\w{C}_{d, \sigma}$ to get
	\[
	\w{Z}(t) \le \w{\mathcal{Z}}(t)  \le\frac{2e^{\frac{C_2t}{1+t}}}{(1+t)^a} \w{\mathcal{Z}}(0)\leq \frac{2(k_0+1)e^{\frac{C_2t}{1+t}}}{(1+t)^a} \w{Z}(0) \quad \forall\, t\ge 0
	\]
if
\[
k\sigma < \lt( 1+k\lt(\frac{1}{\tilde{\gamma}}-2\rt)\rt) \min\{1,\,d\tilde{\gamma}\},
\]
or equivalently
\bq\label{c}
\gamma <  \min\lt\{1+\frac{2}{\sigma+2-\frac1k},\, 1+\frac{2(d-\sigma)}{d(2-\frac1k)}\rt\}
\eq
for every $k=1, \dots, k_0$. 
From the condition $\gamma\leq 2-\frac\sigma d$, we see that the second one on the right-hand side of \eqref{c} is always true, and the first one is also removed for any $k$ if $d=1,2$.
Thus \eqref{c} holds whenever
\[
\gamma < 1+ \frac{2}{\sigma+2}\qquad \text {if}\quad  d\ge 3
\]
under $\gamma \leq 2-\frac\sigma d$. Hence, we have the desired global existence and temporal decay estimates.

\vspace{.2cm}

\noindent $\bullet$ (Uniqueness):  We again show the stability in $L^2$-Sobolev spaces. For $i=1,2$, let $(n_i, w_i)$ be the regular solutions to \eqref{pER} with $\lambda =1$ on time interval $[0,T]$ corresponding to the initial data $(n_{i,0}, w_{i,0})$. Then we proceed as the repulsive case to arrive at the following $L^2$-estimate:
\bq\label{low_diff_est3}
\frac{d}{dt}\lt(\|n_1 - n_2\|_{L^2}^2 + \|w_1-w_2\|_{L^2}^2 \rt) \le C\|w_1-w_2\|_{L^2}^2 + C\|n_1 - n_2\|_{H^1}^2.
\eq
Moreover, we get
\bq\label{nw_comb_est2}
\begin{aligned}
\frac{d}{dt}\lt(\|n_1 -n_2\|_{\dot{H}^{2-\frac\sigma2}}^2 + \|w_1 -w_2\|_{\dot{H}^{2-\frac\sigma2}}^2 \rt) 
&\le C\lt(\|n_1 -n_2\|_{H^{2-\frac\sigma2}}^2 + \|w_1 -w_2\|_{H^{2-\frac\sigma2}}^2 \rt)\\
&\quad +  \intr \Lambda^{2-\frac\sigma2}(w_1 -w_2) \cdot \nabla\Lambda^{2-\frac{3\sigma}{2}} ( (n_1)^{\frac{1}{\tilde\gamma}} - (n_2)^{\frac{1}{\tilde\gamma}})\,dx
\end{aligned}
\eq
from the similar argument for \eqref{nw_comb_est1}, we use the following replacement of \eqref{diff_est_core2}:
\bq\label{diff_est_core4}
\begin{aligned}
	\|\Lambda^{2-\frac\sigma2}&[ (w_1 - w_2)\cdot \nabla (w_2 + v)]\|_{L^2} \\
	&\lesssim \|\Lambda^{2-\frac\sigma2}(w_1-w_2)\|_{L^2}\|\nabla (w_2 + v)\|_{L^\infty} \\
	&\quad + \lt\{\begin{array}{lcl}\|w_1 -w_2\|_{L^\infty}\|\Lambda^{2-\frac\sigma2}\nabla(w_2 +v)\|_{L^2} &\mbox{if} & d\le 3\\
		\|w_1 -w_2\|_{L^{\frac{1}{\frac{1}{2}-\frac{2-\frac{\sigma}{2}}{d}} }}\|\Lambda^{2-\frac\sigma2}\nabla(w_2 + v)\|_{L^{\frac{d}{2-\frac{\sigma}{2}} }} &\mbox{if} & d\ge 4\end{array}\rt.\\
	&\le C\|w_1 -w_2\|_{H^{2-\frac{\sigma}{2}}}.
\end{aligned}
\eq
Here, we used
$H^{2-\frac{\sigma}{2}} \hookrightarrow L^{\infty}$ when $d\le 3$
and $$\dot{H}^{2-\frac{\sigma}{2}} \hookrightarrow L^{\frac{1}{\frac{1}{2}-\frac{2-\frac{\sigma}{2}}{d}}},  \qquad  \dot{H}^{\frac{d}{2}+1}\cap \dot{H}^2 \hookrightarrow\dot{W}^{3-\frac{\sigma}{2}, \frac{d}{2-\frac{\sigma}{2}}} \qquad \text{when }d\ge 4.$$ 

To handle the last term in \eqref{nw_comb_est2}, we proceed similarly to that of Proposition \ref{Hs_est2_case2}. For $k \in \N\cup\{0\}$, we introduce
\[
\tilde{\mathcal{R}}(k) := \intr (n_1)^{(k+1)\lt(\frac{1}{\tilde\gamma}-2\rt)+1} \Lambda^{2-\frac{(k+1)\sigma}{2}}(w_1-w_2)\cdot \nabla \Lambda^{2-\frac{(k+3)\sigma}{2}}(n_1 - n_2)\,dx.
\]
From \eqref{diff_power}, we employ the assertion ($v$) of Lemma \ref{tech} to get
\[\begin{aligned}
	\intr& \Lambda^{2-\frac\sigma2}(w_1 -w_2) \cdot \nabla\Lambda^{2-\frac{3\sigma}{2}} ( (n_1)^{\frac{1}{\tilde\gamma}} - (n_2)^{\frac{1}{\tilde\gamma}})\,dx\\
	&= \frac{1}{\tilde\gamma}\intr \Lambda^{2-\frac\sigma2}(w_1 -w_2) \cdot ([\Lambda^{2-\frac{3\sigma}{2}}, (n_1)^{\frac{1}{\tilde\gamma}-1}] \nabla(n_1-n_2))\,dx\\
	&\quad +\frac{1}{\tilde\gamma}\intr (n_1)^{\frac{1}{\tilde\gamma}-1} \Lambda^{2-\frac\sigma2}(w_1 -w_2) \cdot   \nabla\Lambda^{2-\frac{3\sigma}{2}}  (n_1-n_2)\,dx\\
	&\quad + \frac{1}{\tilde\gamma}\intr \Lambda^{2-\frac\sigma2}(w_1 -w_2) \cdot \Lambda^{2-\frac{3\sigma}{2}} \lt[((n_1)^{\frac{1}{\tilde\gamma}-1} - (n_2)^{\frac{1}{\tilde\gamma}-1})\nabla n_2\rt] \,dx\\
	&\le  C\|w_1 -w_2\|_{H^{2-\frac\sigma2}}\lt( \|\Lambda (n_1^{\frac{1}{\tilde\gamma}-1})\|_{L^\infty}\|n_1 - n_2\|_{H^{2-\frac\sigma2}} + \|\Lambda^{2-\frac{3\sigma}{2}}(n_1^{\frac{1}{\tilde\gamma}-1})\nabla (n_1 - n_2)\|_{L^2}\rt)+\frac{1}{\tilde\gamma}\tilde{\mathcal R}(0) \\
	&\quad+ C\|w_1 -w_2\|_{H^{2-\frac\sigma2}} \lt\| ((n_1)^{\frac{1}{\tilde\gamma}-1} - (n_2)^{\frac{1}{\tilde\gamma}-1})\nabla n_2 \rt\|_{\dot{H}^{2-\frac{3\sigma}{2}}}\\
	&\le C\|w_1 -w_2\|_{H^{2-\frac\sigma2}}\|n_1 - n_2\|_{H^{2-\frac\sigma2}}+\frac{1}{\tilde\gamma}\tilde{\mathcal R}(0) \\
	&\quad+ C\|w_1 -w_2\|_{H^{2-\frac\sigma2}} \lt\| ((n_1)^{\frac{1}{\tilde\gamma}-1} - (n_2)^{\frac{1}{\tilde\gamma}-1})\nabla n_2 \rt\|_{\dot{H}^{2-\frac{3\sigma}{2}}},
\end{aligned}\]
where we used Lemma \ref{tech_4} to get
\[\begin{aligned}
 \|\Lambda^{2-\frac{3\sigma}{2}}(n_1^{\frac{1}{\tilde\gamma}-1})\nabla (n_1 - n_2)\|_{L^2}& \lesssim \lt\{\begin{array}{lcl}\|\Lambda^{2-\frac{3\sigma}{2}}(n_1^{\frac{1}{\tilde\gamma}-1})\|_{L^{\frac{2d}{\sigma}}}\|\nabla(n_1 -n_2)\|_{L^{\frac{1}{\frac12 - \frac{\sigma}{2d}}}} &\mbox{if} & d= 1\\
		\|\Lambda^{2-\frac{3\sigma}{2}}(n_1^{\frac{1}{\tilde\gamma}-1})\|_{L^{\frac{d}{1-\frac\sigma2}}}\|\nabla(n_1 -n_2)\|_{L^{\frac{1}{\frac12 - \frac{1-\frac\sigma2}{d}}}} &\mbox{if} & d\ge 2
		\end{array}\rt.\\
	&\le C\|n_1 -n_2\|_{H^{2-\frac{\sigma}{2}}},
\end{aligned}\]
To estimate the second term in the last line above, we first employ the similar argument used in \eqref{diff_est_core4} to find
\bq\label{diff_est_core5}
\begin{aligned}
	\|(n_1 -n_2)\nabla n_2\|_{\dot{H}^{2-\frac{3\sigma}{2}}} &\lesssim   \|n_1 -n_2\|_{\dot{H}^{2-\frac{3\sigma}{2}}} \| \nabla n_2\|_{L^\infty}  \\ &\quad +\lt\{\begin{array}{lcl} \|n_1 - n_2\|_{L^\infty} \| \Lambda^{2-\frac{3\sigma}{2}} \nabla n_2 \|_{L^2} &\mbox{if}& d \le 3\\
		\|n_1 - n_2\|_{L^{\frac{1}{\frac12 - \frac{2-\frac\sigma2}{d}}}} \| \Lambda^{2-\frac{3\sigma}{2}}\nabla n_2\|_{L^{\frac{d}{2-\frac\sigma2}}}&\mbox{if} & d \ge 4\end{array}\rt. \\
	& \le C \|n_1-n_2\|_{H^{2-\frac{\sigma}{2}}},
\end{aligned}
\eq
where we used $n_2 \in H^s$ with $s>\max\{3-\frac{\sigma}{2}, 1+\frac{d}{2}\}$ and $\dot{H}^{\frac{d}{2}+1-\sigma} \hookrightarrow \dot{W}^{3-\frac{3\sigma}{2}, \frac{d}{2-\frac{\sigma}{2}}}$ when $d\ge 4$.

In a similar way, we use \eqref{power_n} and the Minkowski inequality to get
\[\begin{aligned}
&\| ((n_1)^{\frac{1}{\tilde\gamma}-1} - (n_2)^{\frac{1}{\tilde\gamma}-1})\nabla n_2 \|_{\dot{H}^{2-\frac{3\sigma}{2}}}\\
&\qquad \le C\int_0^1  \| (n_1 - n_2) \nabla n_2 \tilde{F}_{\tilde{\gamma}}(n_1, n_2) \|_{\dot{H}^{2-\frac{3\sigma}{2}}}\,d\tau\\
&\qquad \le C \int_0^1 \|(n_1 -n_2)\nabla n_2\|_{\dot{H}^{2-\frac{3\sigma}{2}}} \| \tilde{F}_{\tilde{\gamma}}(n_1, n_2)\,\|_{L^\infty} d\tau  \\
&\qquad \quad 
+\lt\{\begin{array}{lcl}
	\int_0^1 \|n_1 - n_2\|_{L^\infty} \|\nabla n_2\|_{L^{\infty}} \|  \tilde{F}_{\tilde{\gamma}}(n_1, n_2) \|_{\dot{H}^{2-\frac{3\sigma}{2}}}\,d\tau  &\mbox{if}& d \le 3\\
 \int_0^1\|n_1 - n_2\|_{L^{\frac{1}{\frac12 - \frac{2-\frac\sigma2}{d}}}}\|\nabla n_2\|_{L^{\infty}}  \| \Lambda^{2-\frac{3\sigma}{2}} \tilde{F}_{\tilde{\gamma}}(n_1, n_2) \|_{L^{\frac{d}{2-\frac\sigma2}}}\,d\tau &\mbox{if} & d \ge 4,
 \end{array}\rt. \\
%&\qquad \le C\|n_1 - n_2\|_{H^{2-\frac\sigma2}},
\end{aligned}\]
where $\tilde{F}_{\tilde{\gamma}}(n_1, n_2) :=(n_2 + (n_1 - n_2)\tau)^{\frac{1}{\tilde\gamma}-2} $.\\
 
\noindent It is easy to check that $\tilde{F}_{\tilde{\gamma}}(n_1,n_2) \in L^{\infty} \cap \dot{H}^{2-\frac{3\sigma}{2}} \cap \dot{H}^{\frac{d}{2}-\sigma}$ by Lemma \ref{tech_4} under ($\mathcal{A}3$),
and thus it follows from \eqref{diff_est_core5} that
\[
\| ((n_1)^{\frac{1}{\tilde\gamma}-1} - (n_2)^{\frac{1}{\tilde\gamma}-1})\nabla n_2 \|_{\dot{H}^{2-\frac{3\sigma}{2}}} \lesssim \|n_1 -n_2\|_{H^{2-\frac{\sigma}{2}}}.
\]
Thus \eqref{nw_comb_est2} leads to
\bq\label{nw_comb_est3}
\begin{aligned}
\frac{d}{dt}\lt(\|n_1 -n_2\|_{\dot{H}^{2-\frac\sigma2}}^2 + \|w_1 -w_2\|_{\dot{H}^{2-\frac\sigma2}}^2 \rt)  \le C\lt(\|n_1 -n_2\|_{H^{2-\frac\sigma2}}^2 + \|w_1 -w_2\|_{H^{2-\frac\sigma2}}^2 \rt) + \frac{1}{\tilde\gamma}\tilde{\mathcal{R}}(0).\\
\end{aligned}\eq

Now, we proceed to the iteration argument; for $k \ge 1$, we estimate
\[\begin{aligned}
\frac12&\frac{d}{dt}\intr (n_1)^{k\lt(\frac{1}{\tilde\gamma}-2\rt)} |\Lambda^{2-\frac{(k+1)\sigma}{2}}(w_1 -w_2)|^2\,dx \\
& = -\frac 12\intr ((w_1 +v)\cdot \nabla (n_1)^{k\lt(\frac{1}{\tilde\gamma}-2\rt)}|\Lambda^{2-\frac{(k+1)\sigma}{2}}(w_1 -w_2)|^2\,dx\\
&\quad -\frac {k(1 -2\tilde{\gamma} )}{2}\intr  (n_1)^{k\lt(\frac{1}{\tilde\gamma}-2\rt)} \cdot \nabla(w_2 + v)) |\Lambda^{2-\frac{(k+1)\sigma}{2}}(w_1 -w_2)|^2\,dx\\
&\quad - \intr (n_1)^{k\lt(\frac{1}{\tilde\gamma}-2\rt)} \Lambda^{2-\frac{(k+1)\sigma}{2}}(w_1 -w_2) \cdot \Lambda^{2-\frac{(k+1)\sigma}{2}}\lt((w_1 + v) \cdot \nabla (w_1 - w_2) \rt)\,dx\\
&\quad - \intr (n_1)^{k\lt(\frac{1}{\tilde\gamma}-2\rt)} \Lambda^{2-\frac{(k+1)\sigma}{2}}(w_1 -w_2) \cdot \Lambda^{2-\frac{(k+1)\sigma}{2}}\lt( (w_1 -w_2) \cdot \nabla (w_2+ v) \rt)\,dx\\
&\quad - \tilde\gamma \intr (n_1)^{k\lt(\frac{1}{\tilde\gamma}-2\rt)} \Lambda^{2-\frac{(k+1)\sigma}{2}}(w_1 -w_2) \cdot \Lambda^{2-\frac{(k+1)\sigma}{2}}\lt(n_1 \nabla (n_1-n_2) \rt)\,dx\\
&\quad - \tilde\gamma \intr (n_1)^{k\lt(\frac{1}{\tilde\gamma}-2\rt)} \Lambda^{2-\frac{(k+1)\sigma}{2}}(w_1 -w_2) \cdot \Lambda^{2-\frac{(k+1)\sigma}{2}}\lt((n_1-n_2) \nabla n_2) \rt)\,dx\\
&\quad + \frac{1}{\tilde\gamma} \intr (n_1)^{k\lt(\frac{1}{\tilde\gamma}-2\rt)} \Lambda^{2-\frac{(k+1)\sigma}{2}}(w_1 -w_2) \cdot  \Lambda^{2-\frac{(k+3)\sigma}{2}}\lt((n_1)^{\frac{1}{\tilde\gamma}-1} \nabla(n_1 - n_2) \rt)\,dx\\
&\quad + \frac{1}{\tilde\gamma} \intr (n_1)^{k\lt(\frac{1}{\tilde\gamma}-2\rt)} \Lambda^{2-\frac{(k+1)\sigma}{2}}(w_1 -w_2) \cdot  \Lambda^{2-\frac{(k+3)\sigma}{2}}\lt[((n_1)^{\frac{1}{\tilde\gamma}-1} - (n_2)^{\frac{1}{\tilde\gamma}-1}) \nabla n_2) \rt]\,dx\\
&=: \sum_{i=1}^8 \mathcal{I}_i.
\end{aligned}\]
For $\mathcal{I}_2$, one directly gets
\[
\mathcal{I}_2 \le C\|w_1 -w_2\|_{H^{2-\frac\sigma2}}^2.
\]
For $\mathcal{I}_i$ with $i=3,4,6,7$, we use ($v$) of Lemma \ref{tech}, Lemma \ref{lambda_est} and the similar arguments employed in \eqref{diff_est_core4} to yield
\[
\mathcal{I}_3 \le -\mathcal{I}_1 + C\|w_1 -w_2\|_{H^{2-\frac\sigma2}}^2, \quad \mathcal{I}_4 \le C\|w_1 -w_2\|_{H^{2-\frac\sigma2}}^2, \mathcal{I}_6 \le C\|w_1-w_2\|_{H^{2-\frac\sigma2}}\|n_1 - n_2\|_{H^{2-\frac\sigma2}}, 
\]
and
\[
\mathcal{I}_7 \le \frac{1}{\tilde\gamma}\tilde{\mathcal{R}}(k) + C\|w_1-w_2\|_{H^{2-\frac\sigma2}}\|n_1 - n_2\|_{H^{2-\frac\sigma2}}.
\]
For $\mathcal{I}_5$, we may write as 
\[
\begin{aligned}
\mathcal{I}_5 &= - \tilde\gamma \intr (n_1)^{k\lt(\frac{1}{\tilde\gamma}-2\rt)} \Lambda^{2-\frac{(k+1)\sigma}{2}}(w_1 -w_2) \cdot [\Lambda^{2-\frac{(k+1)\sigma}{2}}, n_1] \nabla (n_1-n_2) \,dx \\
&\quad +\tilde\gamma \intr\nabla (n_1)^{k\lt(\frac{1}{\tilde\gamma}-2\rt)+1} \cdot \Lambda^{2-\frac{(k+1)\sigma}{2}}(w_1 -w_2)  \Lambda^{2-\frac{(k+1)\sigma}{2}} (n_1-n_2) \,dx\\
&\quad -\tilde\gamma \intr\Lambda^{2-\frac{k\sigma}{2}}(w_1 -w_2)  \cdot  \lt[\nabla \Lambda^{-\frac\sigma2},(n_1)^{k\lt(\frac{1}{\tilde\gamma}-2\rt)+1} \rt] \Lambda^{2-\frac{(k+1)\sigma}{2}} (n_1-n_2) \,dx\\
&\quad -\tilde\gamma  \tilde{\mathcal{R}}(k-1).
\end{aligned}
\]
We use Lemma \ref{tech} (ii) and (v),  Lemma \ref{lambda_est}, and  \eqref{tech_5} to obtain
\[
\begin{aligned}
	\mathcal{I}_5 \le C\|w_1-w_2\|_{H^{2-\frac\sigma2}}\|n_1 - n_2\|_{H^{2-\frac\sigma2}}  -\tilde\gamma  \tilde{\mathcal{R}}(k-1).
\end{aligned}
\]
For $\mathcal{I}_8$, we deduce
\[\begin{aligned}
\mathcal{I}_8 &\le C\|w_1-w_2\|_{H^{2-\frac\sigma2}} \lt\|(n_1- n_2) \nabla n_2\int_0^1 (n_2 + (n_1-n_2)\tau)^{\frac{1}{\tilde\gamma}-2}\,d\tau \rt\|_{\dot{H}^{2-\frac{(k+3)\sigma}{2}}}\\
&\le C\|w_1 -w_2\|_{H^{2-\frac\sigma2}}\|n_1 -n_2\|_{H^{2-\frac\sigma2}}.
\end{aligned}\]
Then we collect all the estimates for $\mathcal{I}_i$'s to attain

\bq\label{iter_uniq}
\begin{aligned}
\frac{d}{dt}\intr (n_1)^{k\lt(\frac{1}{\tilde\gamma}-2\rt)} |\Lambda^{2-\frac{(k+1)\sigma}{2}} (w_1 -w_2)|^2\,dx  & \le C\|w_1 -w_2\|_{H^{2-\frac\sigma2}}^2 + C\|n_1 -n_2\|_{H^{2-\frac\sigma2}}^2\\
&\quad -2\tilde\gamma \tilde{\mathcal R}(k-1) + \frac{2}{\tilde\gamma}\tilde{\mathcal R}(k),
\end{aligned}
\eq
and choose the minimum $k_0$ satisfying $\frac{(k_0 +2)\sigma}{2} \ge 1$. Then, we combine \eqref{nw_comb_est3} with \eqref{iter_uniq} to yield
\[%\bq\label{nw_comb_est4}
\begin{aligned}
\frac{d}{dt}&\lt( \|n_1 -n_2\|_{\dot{H}^{2-\frac\sigma2}}^2 + \|w_1 -w_2\|_{\dot{H}^{2-\frac\sigma2}}^2 + \sum_{k=1}^{k_0}\frac{1}{\tilde\gamma^{2k}} \intr (n_1)^{k\lt(\frac{1}{\tilde\gamma}-2\rt)} |\Lambda^{2-\frac{(k+1)\sigma}{2}}(w_1 -w_2)|^2\,dx\rt)\\
&\le C\|w_1 -w_2\|_{H^{2-\frac\sigma2}}^2 + C\|n_1 -n_2\|_{H^{2-\frac\sigma2}}^2 + \frac{2}{\tilde\gamma^{2k_0+1}}\tilde{\mathcal R}(k_0)\\
&\le C\|w_1 -w_2\|_{H^{2-\frac\sigma2}}^2 + C\|n_1 -n_2\|_{H^{2-\frac\sigma2}}^2,
\end{aligned}\]
%\eq
and we further combine this with \eqref{low_diff_est3} to have
\[\begin{aligned}
\frac{d}{dt}&\lt( \|n_1 -n_2\|_{H^{2-\frac\sigma2}}^2 + \|w_1 -w_2\|_{H^{2-\frac\sigma2}}^2 + \sum_{k=1}^{k_0}\frac{1}{\tilde\gamma^{2k}} \intr (n_1)^{k\lt(\frac{1}{\tilde\gamma}-2\rt)} |\Lambda^{2-\frac{(k+1)\sigma}{2}}(w_1 -w_2)|^2\,dx\rt)\\
&\leq C\|w_1 -w_2\|_{H^{2-\frac\sigma2}}^2 + C\|n_1 -n_2\|_{H^{2-\frac\sigma2}}^2.
\end{aligned}\]
Finally, we apply Gr\"onwall's lemma to get the desired uniqueness.
\end{proof}

%%%%%%%%%%%%%%%%%%%%%%%%%%%%%%%%%%%%%%%%%%%%%%%%%%%%%%%%%%%%%%%%%%%%%%%%%%%%%%%%%
%
%
%   
%
%
%%%%%%%%%%%%%%%%%%%%%%%%%%%%%%%%%%%%%%%%%%%%%%%%%%%%%%%%%%%%%%%%%%%%%%%%%%%%%%%%%

\section*{Acknowledgments}
The work of Y.-P. Choi and Y. Lee are supported by NRF grant no. 2022R1A2C1002820. The work of Y. Lee is supported by NRF grant (No. NRF-2022R1I1A1A01068481).

%%%%%%%%%%%%%%%%%%%%%%%%%%%%%%%%%%%%%%%%%%%%%%%%%%%%%%%%%%%%%%%%%%%%%%%%%%%%%%%%%
%
%
%   
%
%
%%%%%%%%%%%%%%%%%%%%%%%%%%%%%%%%%%%%%%%%%%%%%%%%%%%%%%%%%%%%%%%%%%%%%%%%%%%%%%%%%

\appendix

\section{Euler system with sub-Manev interactions:  $1\leq\sigma<2$ case}\label{app.A}

In this appendix, we present the  $H^s$ estimates of solutions in the case $1\leq  \sigma < 2$, and from which we give a proof of Theorem \ref{main_thm2} by obtaining a priori temporal decay. This proof is also available for $\sigma=2$.

\begin{lemma}\label{high_est_app}
	Let the assumptions of Theorem \ref{main_thm2} be satisfied.
For $T>0$, suppose that $(n,w)$ is a smooth solution to \eqref{pER} with $\lambda=\pm1$ on the time interval $[0,T]$ decaying fast at infinity. 
	Then we have
	\[
	\begin{aligned}
		&\frac{d}{dt} \w{X}_s +\frac{s-\frac{d}{2}+\min\{1, d\tilde{\gamma}\}}{1+t} \w{X}_s   \\
		&\quad \lesssim
		(\|\nabla w\|_{L^\infty}+ \|\nabla n\|_{L^\infty}) \w{X}_s
		+ \frac{\w{X}_s}{(1+t)^2} 
		+  \frac{\|n\|_{L^\infty}+\|w\|_{L^\infty}}{(1+t)^{2+s-\frac d2}}+ \frac{\|n\|_{\dot{H}^{s-1}} +\|w\|_{\dot{H}^{s-1}}}{(1+t)^3} \\ 
		&\qquad + \frac{\|\nabla n\|_{L^\infty}+\|\nabla w\|_{L^\infty}}{(1+t)^{s+1-\frac d2}}  +\|n \|_{L^{\infty}}^{\frac{1}{\tilde{\gamma}}-1} \|n\|_{\dot{H}^{s-\sigma+1}},
	\end{aligned}
	\]
	where
	$\w{X}_{s} := (\|w\|_{\dot{H}^{s}}^2 + \|n\|_{\dot{H}^s}^2)^{\frac12}$.
\end{lemma}
\begin{proof}
	From Lemma \ref{pres_est1}, let us recall that 
	\[
	\begin{aligned}
		\frac{1}{2} &\frac{d}{dt} \w{X}_s^2 +\frac{s-\frac{d}{2}+\min\{ d\tilde{\gamma}, 1\}}{1+t} \w{X}_s^2   
		-\lambda \intr \Lambda^{s}w \cdot \nabla \Lambda^{s-\sigma} (n^{\frac1{\tilde{\gamma}}}) \\
		&\lesssim
		(\|\nabla w\|_{L^\infty}+ \|\nabla n\|_{L^\infty}) \w{X}_s^2 
		+ \frac{\w{X}_s^2}{(1+t)^2} \\
		& \quad + \lt( \frac{\|n\|_{L^\infty}+\|w\|_{L^\infty}}{(1+t)^{2+s-\frac d2}}+ \frac{\|n\|_{\dot{H}^{s-1}} +\|w\|_{\dot{H}^{s-1}}}{(1+t)^3} + \frac{\|\nabla n\|_{L^\infty}+\|\nabla w\|_{L^\infty}}{(1+t)^{s+1-\frac d2}}\rt)\w{X}_s.
	\end{aligned}
	\]
	Then we use Lemma \ref{tech_4} to get
	\[
	\begin{aligned}
		\intr \Lambda^{s} w \cdot \nabla \Lambda^{s-\sigma} (n^{\frac{1}{\tilde{\gamma}}})\, dx\lesssim \|w\|_{\dot{H}^s}\|n^{\frac{1}{\tilde\gamma}}\|_{\dot{H}^{s-\sigma+1}}
		\lesssim \|w\|_{\dot{H}^s}\|n\|_{L^\infty}^{\frac{1}{\tilde{\gamma}}-1}\|n\|_{\dot{H}^{s-\sigma+1}}
	\end{aligned}
	\]
if $s \ge 1+\sigma $ and $s<\frac1{\tilde{\gamma}}+\sigma-1/2$.
Hence, we obtain
	\[
	\begin{aligned}
		\frac{1}{2} &\frac{d}{dt} \w{X}_s^2 +\frac{s-\frac{d}{2}+\min\{ d\tilde{\gamma}, 1\}}{1+t} \w{X}_s^2   \\
		&\lesssim
		(\|\nabla w\|_{L^\infty}+ \|\nabla n\|_{L^\infty}) \w{X}_s^2 
		+ \frac{\w{X}_s^2}{(1+t)^2} \\
		& \quad + \lt( \frac{\|n\|_{L^\infty}+\|w\|_{L^\infty}}{(1+t)^{2+s-\frac d2}}+ \frac{\|n\|_{\dot{H}^{s-1}} +\|w\|_{\dot{H}^{s-1}}}{(1+t)^3} + \frac{\|\nabla n\|_{L^\infty}+\|\nabla w\|_{L^\infty}}{(1+t)^{s+1-\frac d2}}\rt)\w{X}_s\\
		&\quad +\|n \|_{L^{\infty}}^{\frac{1}{\tilde{\gamma}}-1} \|n\|_{\dot{H}^{s-\sigma+1}} \w{X}_s
	\end{aligned}
	\]
	which implies the desired result.
\end{proof}

Now we provide a priori temporal decay in the case $1\leq \sigma < 2$.

\begin{lemma}\label{temp_est_app}
	Let the assumptions of Theorem \ref{main_thm2} be satisfied.
For $T>0$, suppose that $(n,w)$ is a smooth solution to \eqref{pER} with $\lambda=\pm1$ on the time interval $[0,T]$ decaying fast at infinity. 
	Then one has
	\[
	\frac{d}{dt} \w{Z} (t)
	+\frac{C_{d, \gamma}}{1+t} \w{Z} (t) \lesssim \w{Z}^2(t)  + \frac{\w{Z}(t) }{(1+t)^2} 
	+\frac{\w{Z}^{\frac{1}{\tilde{\gamma}}}(t)}{ (1+t)^{2-\sigma-\frac{1}{\tilde{\gamma}}}}
	\]
	with $C_{d, \gamma}=1+ \min\{1, d\tilde{\gamma}\}$ and $\w{Z}$ is defined in \eqref{temp_notation_2}.
	
\end{lemma}
\begin{proof}
	From Lemma \ref{pres_est1},  we follow the notation \eqref{temp_notation} to obtain
\bq\label{case1_lower_est}
\frac{d}{dt} (\w{n}_2+\w{w}_2) +  \frac{1+\min\{1, d\tilde{\gamma}\} }{1+t} (\w{n}_2 +\w{w}_2)\lesssim
\w{Z}^2(t) +  \frac{\w{Z}(t)}{(1+t)^2} 
+ \frac{ \w{Z}^{\frac{1}{\tilde{\gamma}}}(t)}{(1+t)^{2-\sigma-\frac{1}{\tilde{\gamma}}}}.
\eq
Next, we use the notation \eqref{temp_notation_2} and deduce from Lemma \ref{high_est_app} that
	\[
	\begin{aligned}
		\frac{d}{dt}& \w{Y}_s +\frac{1+\min\{1, d\tilde{\gamma}\}}{1+t} \w{Y}_s \\
		&\lesssim (\w{w}_{1, \infty} + \w{n}_{1, \infty})\w{Y}_s +\frac{\w{Y}_s}{(1+t)^2} +\frac{\w{n}_{\infty}+\w{w}_{\infty} + \w{n}_{s-1, 2}+\w{w}_{s-1, 2}+\w{w}_{1, \infty}+\w{n}_{1, \infty}}{(1+t)^2}\\
		&\quad  +\frac{\w{n}_{\infty}^{\frac{1}{\tilde{\gamma}}-2} \w{n}_{1, \infty} \w{n}_{s-\sigma, 2} + \w{n}_{\infty}^{\frac{1}{\tilde{\gamma}}-1} \w{n}_{s-\sigma+1, 2}}{(1+t)^{2-\sigma-\frac{1}{\tilde{\gamma}}}},
	\end{aligned}
	\]
	and we use Lemma \ref{bdds2} along with the relations
	\[
	\w{n}_{s-\sigma+1, 2}(t) \lesssim \w{n}_{s, 2}^{\frac{s-\sigma+1}{s}}(t)\, \w{n}_{2}^{1-\frac{s-\sigma+1}{s}}(t),\qquad \quad 
	\w{n}_{s-\sigma, 2}(t) \lesssim \w{n}_{s, 2}^{\frac{\sigma}{s}}(t)\, \w{n}_2^{1-\frac{\sigma}{s}}(t)\quad \text{for}\ s>\frac{d}{2}+1
	\] 
	to yield
	\[
	\frac{d}{dt} \w{Y}_s (t)
	+\frac{1+\min\{1,  d\tilde{\gamma}\}}{1+t} \w{Y}_s (t) \lesssim \w{Z}^2(t)  + \frac{\w{Z}(t) }{(1+t)^2} 
	+\frac{\w{Z}^{\frac{1}{\tilde{\gamma}}}(t)}{ (1+t)^{2-\sigma-\frac{1}{\tilde{\gamma}}}}.
	\]
	Hence we combine this with \eqref{case1_lower_est} to conclude the desired result.
\end{proof}

\begin{proof}[Proof of Theorem \ref{main_thm2}]
	We again deduce the local existence from the arguments in Section \ref{sec:3}. For the global existence,  Lemma \ref{temp_est_app} shows
	\[
	\frac{d}{dt} \w{Z} (t)
	+\frac{C_{d, \gamma}}{1+t} \w{Z} (t) \leq C_3 \lt( \w{Z}^2(t)  + \frac{\w{Z}(t) }{(1+t)^2} 
	+\frac{\w{Z}^{\frac{1}{\tilde{\gamma}}}(t)}{ (1+t)^{2-\sigma-\frac{1}{\tilde{\gamma}}}}\rt)
	\]
	with $C_{d, \gamma}=1+ \min\{1, d\tilde{\gamma}\}$ for a constant $C_3>0$ independent of $t$.
	By applying Lemma \ref{tech_6} to the above with $a=C_{d, \gamma}$, we have
	\[
	\widetilde{Z}(t) \le \frac{2e^{\frac{C_3t}{1+t}}}{(1+t)^a}\widetilde{Z}(0) \quad \forall\, t\ge 0,
	\]
	whenever 
	\[
	\sigma <\min \{ 1, d \tilde{\gamma}\} \lt(\frac{1}{\tilde{\gamma}}-1\rt),
	\]
	which is equivalent to
	\[
	\gamma <\min\lt\{ 1+\frac{2}{\sigma+1}, \ 1+\frac{2(d-\sigma)}{d}\rt\}.
	\]
	It implies
	\[
	\gamma < \lt\{ \begin{array}{lcl} 1+ \frac{2}{\sigma+1} & \mbox{if} & d\ge 3,\\[2mm]
	3-\sigma & \mbox{if} &d=2,\end{array}\rt.
	\]
due to $\frac{1}{\sigma+1} \le \frac{d-\sigma}{d}$ when $d\ge 3$ and $\sigma \in [1,2]$.

	For the uniqueness, the estimates in the proof of Theorem \ref{main_thm3} again gives
	\[\begin{aligned}
	\frac{d}{dt}\lt(\|n_1 - n_2\|_{L^2}^2 + \|w_1 -w_2\|_{L^2}^2 \rt) &\le C\|n_1-n_2\|_{L^2}^2   +C\|w_1-w_2\|_{L^2}^2\\
	&\quad + C\|w_1-w_2\|_{L^2}\|\nabla\Lambda^{-\sigma}((n_1)^{\frac{1}{\tilde\gamma}} - (n_2)^{\frac{1}{\tilde\gamma}})\|_{L^2}\\
	&\le C\|n_1-n_2\|_{L^2}^2   +C\|w_1-w_2\|_{L^2}^2,
	\end{aligned}\]
	where we used Hardy--Littlewood--Sobolev theorem to get
	\[\begin{aligned}
	\|\nabla\Lambda^{-\sigma}((n_1)^{\frac{1}{\tilde\gamma}} - (n_2)^{\frac{1}{\tilde\gamma}})\|_{L^2}&\le C \|(n_1)^{\frac{1}{\tilde\gamma}} - (n_2)^{\frac{1}{\tilde\gamma}}\|_{L^{\frac{1}{\frac12 + \frac{\sigma-1}{d}}}}\\
	&\le C\|n_1 - n_2\|_{L^2} \lt\|\int_0^1 (n_2 + (n_1- n_2)\tau)^{\frac{1}{\tilde\gamma}-1}\,d\tau \rt\|_{L^{\frac{d}{\sigma-1}}}\\
	&\le C\|n_1-n_2\|_{L^2}.
	\end{aligned}\]
	Thus, applying Gr\"onwall's lemma gives the desired result.
\end{proof}
%%%%%%%%%%%%%%%%%%%%%%%%%%%%%%%%%%%%%%%%%%%%%%%%%%%%%%%%%%%%%%%%%%%%%%%%%%%%%%%%%5
%
%
%                         
%
%
%%%%%%%%%%%%%%%%%%%%%%%%%%%%%%%%%%%%%%%%%%%%%%%%%%%%%%%%%%%%%%%%%%%%%%%%%%%%%%%%%
\section{Highest-order estimates in the repulsive case }\label{app.B}
In this part, we provide the details of proof of \eqref{n_lower_est}. For this, we actually show
	\[
	\begin{aligned}
		\frac{1}{2} &\frac{d}{dt} \intr n^{\frac{1}{\tilde{\gamma}}-2} |\Lambda^\ell n|^2\,dx  +\frac{\ell }{1+t} \intr n^{\frac{1}{\tilde{\gamma}}-2}  |\Lambda^{\ell}n|^2\,dx- \tilde\gamma \intr n^{\frac{1}{\tilde\gamma}-1}\nabla \Lambda^{\ell-\frac\sigma2}n \cdot \Lambda^{\ell+\frac\sigma2} w\,dx\\
		&\leq C\|n\|_{L^\infty}^{\frac{1}{\tilde{\gamma}}-2} \|\nabla w\|_{L^{\infty}} \|n\|_{\dot{H}^{\ell}}^2+\frac{C}{(1+t)^2} \|n\|_{L^\infty}^{\frac{1}{\tilde{\gamma}}-2} \|n\|_{\dot{H}^{\ell}}^2 + \frac{C}{(1+t)^{2+\ell-\frac d2}}\|n\|_{L^\infty}^{\frac{1}{\tilde\gamma}-1}\|n\|_{\dot{H}^\ell}\\
		&\quad + C\|n\|_{L^\infty}^{\frac{1}{\tilde\gamma}-2}\|\nabla n\|_{L^\infty}\|w\|_{\dot{H}^\ell}\|n\|_{\dot{H}^\ell} + \frac{C}{(1+t)^{1+\ell-\frac d2}}\|n\|_{L^\infty}^{\frac{1}{\tilde\gamma}-2}\|\nabla n\|_{L^\infty}\|n\|_{\dot{H}^\ell} \\
		&\quad + \frac{C}{(1+t)^3}\|n\|_{L^\infty}^{\frac{1}{\tilde\gamma}-2}\|n\|_{\dot{H}^{\ell-1}}\|n\|_{\dot{H}^\ell} + C\|n\|_{L^\infty}^{\frac{1}{\tilde\gamma}-2} \|n\|_{\dot{H}^{s-\sigma}}^{\frac{\frac d2+(1-\sigma)}{s-\sigma}}\|n\|_{L^2}^{1-\frac{\frac d2+(1-\sigma)}{s-\sigma}}\|n\|_{\dot{H}^{\ell+\frac\sigma2}} \|w\|_{\dot{H}^{\ell+\frac\sigma2}}
	\end{aligned}
	\]
	for any $1<\ell \le s-\frac\sigma2$.
	
	Now, we observe that
	\[ 
	\begin{aligned}
		\frac{1}{2} \frac{d}{dt} \intr n^{\frac{1}{\tilde{\gamma}}-2} |\Lambda^\ell n|^2 \,dx &= \frac{1}{2}\lt(\frac{1}{\tilde{\gamma}}-2\rt) \intr n^{\frac{1}{\tilde{\gamma}}-3} \pa_t n\, |\Lambda^{\ell}n|^2 \,dx
		+ \intr n^{\frac{1}{\tilde{\gamma}}-2} \Lambda^\ell n \Lambda^\ell \pa_t n\,dx =: \sfI_1 + \sfI_2.
	\end{aligned}
	\]
For $\sfI_1$, we obtain
	\[
	\begin{aligned}
		\sfI_1&= 
		-\frac{1}{2} \intr   w \cdot \nabla (n^{\frac{1}{\tilde{\gamma}}-2}) \, |\Lambda^{\ell}n|^2\,dx 
		-\frac{1}{2} \intr  v\cdot \nabla (n^{\frac{1}{\tilde{\gamma}}-2}) \, |\Lambda^{\ell}n|^2 \,dx \\
		&\quad
		-\frac 12(1-2\tilde\gamma) \intr n^{\frac{1}{\tilde{\gamma}}-2}  \nabla \cdot w \, |\Lambda^{\ell}n|^2 \,dx
		-\frac12(1-2\tilde\gamma) \intr n^{\frac{1}{\tilde{\gamma}}-2}   \nabla \cdot v\, |\Lambda^{\ell}n|^2 \,dx\\
		&=: \sum_{i=1}^4 \sfI_{1i}.
	\end{aligned}
	\]
	Here, we easily estimate
	\[
	\sfI_{13} \lesssim \|n\|_{L^\infty}^{\frac{1}{\tilde{\gamma}}-2} \|\nabla w\|_{L^{\infty}} \|n\|_{\dot{H}^{\ell}}^2
	\]
	and
	\[
		\sfI_{14} \leq -\frac{\frac d2-d\tilde\gamma}{(1+t)}\intr n^{\frac{1}{\tilde{\gamma}}-2}  |\Lambda^{\ell}n|^2\,dx +\frac{C}{(1+t)^2} \|n\|_{L^\infty}^{\frac{1}{\tilde{\gamma}}-2} \|n\|_{\dot{H}^{\ell}}^2.
	\]
	Meanwhile, for $\sfI_2$, we get
	\[
	\begin{aligned}
		\sfI_2&= -\intr n^{\frac{1}{\tilde{\gamma}}-2} \Lambda^{\ell} n\,  (w \cdot \nabla \Lambda^{\ell} n)\,dx -\intr n^{\frac{1}{\tilde{\gamma}}-2}\Lambda^{\ell} n (v \cdot \nabla \Lambda^{\ell} n)\,dx\\
		&\quad +\ell \sum_{1\leq k\leq d}\intr n^{\frac{1}{\tilde{\gamma}}-2} \Lambda^{\ell} n \, \pa_k v \cdot  \nabla \pa_k \Lambda^{\ell-2} n\,dx \\
		&\quad -\tilde{\gamma} \intr n^{\frac{1}{\tilde{\gamma}}-2} \Lambda^{\ell} n \, \Lambda^{\ell}(n \nabla \cdot v ) \,dx
		+\intr n^{\frac{1}{\tilde{\gamma}}-2} \Lambda^{\ell} n \, [w, \Lambda^{\ell}] \cdot \nabla n\,dx	\\
		&\quad - \tilde{\gamma}\intr n^{\frac{1}{\tilde\gamma}-1} \Lambda^\ell n \nabla \cdot\Lambda^\ell w\,dx +  \tilde{\gamma}\intr n^{\frac{1}{\tilde{\gamma}}-2} \Lambda^{\ell} n [n, \Lambda^\ell]\nabla \cdot w\,dx\\
		&\quad + \intr n^{\frac{1}{\tilde{\gamma}}-2} \Lambda^{\ell} n \,   \Big( [v, \Lambda^\ell]\cdot\nabla n -\ell \sum_{1\leq k\leq d}\pa_k v \cdot  \nabla \pa_k \Lambda^{\ell-2} n  \Big)\,dx\\
		&=: \sum_{i=1}^8 \sfI_{2i},
	\end{aligned}
	\]
	and integration by parts implies
	\[\begin{aligned}
	\sfI_{21} = \frac12\intr \nabla\cdot\lt(n^{\frac{1}{\tilde\gamma}-2} w\rt)|\Lambda^\ell n|^2\,dx  \le -\sfI_{11} + C\|n\|_{L^\infty}^{\frac{1}{\tilde\gamma}-2}\|\nabla w\|_{L^\infty}\|n\|_{\dot{H}^\ell}^2,
	\end{aligned}\]
	and
	\[\begin{aligned}
	\sfI_{22}&= \frac12\intr \nabla \cdot \lt(n^{\frac{1}{\tilde\gamma}-2} v\rt)|\Lambda^\ell n|^2\,dx  \le -\sfI_{12} + \frac{\frac d2}{1+t}\intr n^{\frac{1}{\tilde\gamma}-2}|\Lambda^\ell n|^2 \,dx + \frac{C}{(1+t)^2}\|n\|_{L^\infty}^{\frac{1}{\tilde\gamma}-2}\|n\|_{\dot{H}^\ell}^2.
	\end{aligned}\]	
We also directly obtain
\[
\sfI_{23} \le -\frac{\ell}{1+t}\intr n^{\frac{1}{\tilde\gamma}-2} |\Lambda^\ell n|^2\,dx + \frac{C}{(1+t)^2}\|n\|_{L^\infty}^{\frac{1}{\tilde\gamma} -2}\|n\|_{\dot{H}^\ell}^2
\]
and
\[\begin{aligned}
\sfI_{24}&\le -\frac{d\tilde\gamma}{1+t}\intr n^{\frac{1}{\tilde\gamma}-2} |\Lambda^\ell n|^2\,dx + \frac{C}{(1+t)^2}\lt(\|n\|_{L^\infty}^{\frac{1}{\tilde\gamma}-2} \|n\|_{\dot{H}^\ell}^2 + \|n\|_{L^\infty}^{\frac{1}{\tilde\gamma}-1}\|K\|_{\dot{H}^\ell}\|n\|_{\dot{H}^\ell} \rt)\\
&\le -\frac{d\tilde\gamma}{1+t}\intr n^{\frac{1}{\tilde\gamma}-2} |\Lambda^\ell n|^2\,dx + \frac{C}{(1+t)^2}\|n\|_{L^\infty}^{\frac{1}{\tilde\gamma}-2} \|n\|_{\dot{H}^\ell}^2 + \frac{C}{(1+t)^{2+\ell-\frac d2}}\|n\|_{L^\infty}^{\frac{1}{\tilde\gamma}-1}\|n\|_{\dot{H}^\ell}.
\end{aligned}\]
Next, by using \eqref{tech_1}, \eqref{tech_2}, and Proposition \ref{prop_bur}, we deduce
\[
\sfI_{25}+\sfI_{27}\lesssim \|n\|_{L^\infty}^{\frac{1}{\tilde\gamma}-2}\|n\|_{\dot{H}^\ell}\lt(\|\nabla n\|_{L^\infty} \|w\|_{\dot{H}^\ell} + \|\nabla w\|_{L^\infty}\|n\|_{\dot{H}^\ell} \rt)
\]
and
\[\begin{aligned}
\sfI_{28}&\lesssim \|n\|_{L^\infty}^{\frac{1}{\tilde\gamma}-2}\|n\|_{\dot{H}^\ell}\lt(\|\nabla n\|_{L^\infty}\|v\|_{\dot{H}^\ell} + \|\nabla^2 v\|_{L^\infty}\|n\|_{\dot{H}^{\ell-1}}\rt)\\
&\lesssim \|n\|_{L^\infty}^{\frac{1}{\tilde\gamma}-2}\|n\|_{\dot{H}^\ell}\lt( \frac{\|\nabla n\|_{L^\infty}}{(1+t)^{\ell+1-\frac d2}} + \frac{\|n\|_{\dot{H}^{\ell-1}}}{(1+t)^3} \rt).
\end{aligned}\]
Finally, we use \eqref{tech_5}, the Gagliardo--Nirenberg--Sobolev inequality, and Lemma \ref{tech_4} to get
\[
\begin{aligned}
\sfI_{26}&= \tilde\gamma \intr \nabla \Lambda^{-\frac \sigma 2}\lt( n^{\frac{1}{\tilde\gamma}-1} \Lambda^\ell n\rt) \cdot \Lambda^{\ell+\frac\sigma2} w\,dx\\
&= \tilde\gamma\intr \lt(\lt[\nabla \Lambda^{-\frac \sigma 2}, n^{\frac{1}{\tilde\gamma}-1}\rt] \Lambda^\ell n + n^{\frac{1}{\tilde\gamma}-1}\nabla \Lambda^{\ell-\frac\sigma2}n \rt)\cdot \Lambda^{\ell+\frac\sigma2} w\,dx\\
&\le C\|\Lambda^{1-\sigma} (n^{\frac{1}{\tilde\gamma}-1})\|_{L^\infty} \|\Lambda^{\ell+\frac\sigma2} n\|_{L^2}\|\Lambda^{\ell+\frac\sigma 2}w\|_{L^2} + C\|n^{\frac{1}{\tilde\gamma}-1}\|_{\dot{H}^{\frac{d}{2}+1-\sigma}} \|\Lambda^{\ell+\frac d2} n\|_{L^2}\|\Lambda^{\ell+\frac\sigma2} w\|_{L^2} \\
&\quad + \tilde\gamma \intr n^{\frac{1}{\tilde\gamma}-1}\nabla \Lambda^{\ell-\frac\sigma2}n \cdot \Lambda^{\ell+\frac\sigma2} w\,dx\\
&\le  C\|n^{\frac{1}{\tilde\gamma}-1}\|_{\dot{H}^{s-\sigma}}^{\frac{\frac d2 + (1-\sigma)}{s-\sigma}}\|n^{\frac{1}{\tilde\gamma}-1}\|_{L^2}^{1-\frac{\frac d2 + (1-\sigma)}{s-\sigma}}\|\Lambda^{\ell+\frac\sigma2} n\|_{L^2}\|\Lambda^{\ell+\frac\sigma 2}w\|_{L^2}   +  \tilde\gamma \intr n^{\frac{1}{\tilde\gamma}-1}\nabla \Lambda^{\ell-\frac\sigma2}n \cdot \Lambda^{\ell+\frac\sigma2} w\,dx\\
&\le C\|n\|_{L^\infty}^{\frac{1}{\tilde\gamma}-2}\|n\|_{\dot{H}^{s-\sigma}}^{\frac{\frac d2 +(1-\sigma)}{s-\sigma}}\|n\|_{L^2}^{1-\frac{\frac d2 +(1-\sigma)}{s-\sigma}}\|\Lambda^{\ell+\frac\sigma2} n\|_{L^2}\|\Lambda^{\ell+\frac\sigma 2}w\|_{L^2}  +  \tilde\gamma \intr n^{\frac{1}{\tilde\gamma}-1}\nabla \Lambda^{\ell-\frac\sigma2}n \cdot \Lambda^{\ell+\frac\sigma2} w\,dx
\end{aligned}
\]
under the additional condition $s<\frac{1}{\tilde{\gamma} }+\sigma -\frac 12$ when $\frac{1}{\tilde{\gamma}}$ is not an integer. 
Hence we collect all the estimates for $\sfI_{1i}$'s and $\sfI_{2i}$'s to yield the desired result.

%%%%%%%%%%%%%%%%%%%%%%%%%%%%%%%%%%%%%%%%%%%%%%%%%%%%%%%%%%%%%%%%%%%%%%%%%%%%%%%%%5
%
%
%                        
%
%
%%%%%%%%%%%%%%%%%%%%%%%%%%%%%%%%%%%%%%%%%%%%%%%%%%%%%%%%%%%%%%%%%%%%%%%%%%%%%%%%%

\section{Highest-order estimates in the attractive case}
%%%%%%%%%%%%%%%%%%%%%%%%%%%%%%%%%%%%%%%%%%%%%%%%%%%%%%%%%%%%%%%%%%%%%%%%%%%%%%%%%5
%
%
%                         
%
%
%%%%%%%%%%%%%%%%%%%%%%%%%%%%%%%%%%%%%%%%%%%%%%%%%%%%%%%%%%%%%%%%%%%%%%%%%%%%%%%%%
In this appendix, we present the proofs of Lemmas \ref{Gap} and  \ref{higher_it}.
\subsection{Proof of Lemma \ref{Gap}}\label{app.C}
For $k\ge 0$, we first see that 
\bq\label{P_k}
\begin{aligned}
	\mathcal{R}_s (k) 
	&= \frac{1}{\tilde{\gamma}} \intr n^{k(\frac{1}{\tilde{\gamma}}-2)} \Lambda^{s-\frac{k \sigma}{2}} w \cdot  \Lambda^{s-\frac{(k+2)\sigma}{2}} (n^{\frac{1}{\tilde{\gamma}}-1} \nabla n)\,dx\\
	& =\frac{1}{\tilde{\gamma}} \intr n^{k(\frac{1}{\tilde{\gamma}}-2)} \Lambda^{s-\frac{k \sigma}{2}} w \cdot [\Lambda^{s-\frac{(k+2)\sigma}{2}}, n^{\frac{1}{\tilde{\gamma}}-1}] \nabla n\, dx \\
	&\quad +\frac{1}{\tilde{\gamma}} \intr n^{k(\frac{1}{\tilde{\gamma}}-2)+\frac{1}{\tilde{\gamma}}-1} \Lambda^{s-\frac{k \sigma}{2}} w \cdot \Lambda^{s-\frac{(k+2)\sigma}{2}} \nabla n\,dx\\
	&\le C\|n\|_{L^\infty}^{k\lt(\frac{1}{\tilde\gamma}-2\rt)}\|w\|_{\dot{H}^{s-\frac{k\sigma}{2}}} \lt(\|\nabla (n^{\frac{1}{\tilde\gamma}-1})\|_{L^\infty}\|n\|_{\dot{H}^{s-\frac{(k+2)\sigma}{2}}} + \|\nabla n\|_{L^\infty}\|n^{\frac{1}{\tilde\gamma}-1}\|_{\dot{H}^{s-\frac{(k+2)\sigma}{2}}} \rt)\\
	 &\quad +\frac{1}{\tilde{\gamma}} \intr n^{(k+1)(\frac{1}{\tilde{\gamma}}-2)+1} \Lambda^{s-\frac{k \sigma}{2}} w \cdot \Lambda^{s-\frac{(k+2)\sigma}{2}} \nabla n\,dx\\
	 &\le C \|n\|_{L^\infty}^{(k+2)\lt(\frac{1}{\tilde\gamma}-2\rt)}\|\nabla n\|_{L^\infty}\|w\|_{\dot{H}^{s-\frac{k\sigma}{2}}} \|n\|_{\dot{H}^{s-\frac{(k+2)\sigma}{2}}}\\
	 &\quad +\frac{1}{\tilde{\gamma}} \intr n^{(k+1)(\frac{1}{\tilde{\gamma}}-2)+1} \Lambda^{s-\frac{k \sigma}{2}} w \cdot \Lambda^{s-\frac{(k+2)\sigma}{2}} \nabla n\,dx,
\end{aligned}
\eq
where we use \eqref{tech_1} and Lemma \ref{tech_4} with the assumption $\frac{(k+2)\sigma}{2} \leq s < \frac{1}{\tilde{\gamma}}-\frac{1}{2} +\frac{(k+2)\sigma}{2}$ if $\frac{1}{\tilde{\gamma}}$ is not an integer. 

On the other hand, we write as 
\begin{align*}
\frac{1}{\tilde\gamma}\mathcal{P}_s(k+1)
&= \intr n^{(k+1)\lt(\frac{1}{\tilde\gamma}-2\rt)+1} \Lambda^{s-\frac{(k+1)\sigma}{2}} w \cdot \nabla \Lambda^{s-\frac{(k+1)\sigma}{2}}n\,dx\\
&\quad + \intr n^{(k+1)\lt(\frac{1}{\tilde\gamma}-2\rt)}  \Lambda^{s-\frac{(k+1)\sigma}{2}} w \cdot \lt[\Lambda^{s-\frac{(k+1)\sigma}{2}}, n \rt]\nabla n\,dx\\
&= -\intr  n^{(k+1)\lt(\frac{1}{\tilde\gamma}-2\rt)+1}\lt(\nabla \cdot \Lambda^{s-\frac{(k+1)\sigma}{2}} w\rt)  \Lambda^{s-\frac{(k+1)\sigma}{2}}n\,dx\\
&\quad -\intr \nabla (n^{(k+1)\lt(\frac{1}{\tilde\gamma}-2\rt)+1})\cdot \Lambda^{s-\frac{(k+1)\sigma}{2}} w \, \Lambda^{s-\frac{(k+1)\sigma}{2}}n\,dx\\
&\quad + \intr n^{(k+1)\lt(\frac{1}{\tilde\gamma}-2\rt)}  \Lambda^{s-\frac{(k+1)\sigma}{2}} w \cdot \lt[\Lambda^{s-\frac{(k+1)\sigma}{2}}, n \rt]\nabla n\,dx\\
&= \intr\Big[\Lambda^{-\frac \sigma2}\nabla , n^{(k+1)\lt(\frac{1}{\tilde\gamma}-2\rt)+1}\Big]  \Lambda^{s-\frac{(k+1)\sigma}{2}}n  \cdot \Lambda^{s-\frac{k\sigma}{2}}w\,dx\\
&\quad +  \intr n^{(k+1)\lt(\frac{1}{\tilde\gamma}-2\rt)+1}\Lambda^{s-\frac{k\sigma}{2}} w \cdot \nabla \Lambda^{s-\frac{(k+2)\sigma}{2}}n\,dx\\
&\quad -\intr \nabla (n^{(k+1)\lt(\frac{1}{\tilde\gamma}-2\rt)+1})\cdot \Lambda^{s-\frac{(k+1)\sigma}{2}} w \, \Lambda^{s-\frac{(k+1)\sigma}{2}}n\,dx\\
&\quad + \intr n^{(k+1)\lt(\frac{1}{\tilde\gamma}-2\rt)}  \Lambda^{s-\frac{(k+1)\sigma}{2}} w \cdot \lt[\Lambda^{s-\frac{(k+1)\sigma}{2}}, n \rt]\nabla n\,dx\\
&=: \sum_{i=1}^4 \sfJ_i
\end{align*}
by integration by parts.
By using \eqref{tech_5} and the Gagliardo--Nirenberg interpolation together with Lemma \ref{tech_4}, we estimate $\sfJ_1$ as 
\[\begin{aligned}
	\sfJ_1 &\lesssim \lt\| \Big[\Lambda^{-\frac \sigma2}\nabla , n^{(k+1)\lt(\frac{1}{\tilde\gamma}-2\rt)+1}\Big]  \Lambda^{s-\frac{(k+1)\sigma}{2}}n \rt\|_{L^2}\|w\|_{\dot{H}^{s-\frac{k\sigma}{2}}}\\
	&\lesssim \lt(\|\Lambda^{1-\sigma} (n^{(k+1)\lt(\frac{1}{\tilde\gamma}-2\rt)+1})\|_{L^{\infty}} + \|n^{(k+1)\lt(\frac{1}{\tilde\gamma}-2\rt)+1}\|_{\dot{H}^{\frac{d}{2}+1-\sigma}} \rt)\|n\|_{\dot{H}^{s-\frac{k\sigma}{2}}}\|w\|_{\dot{H}^{s-\frac{k\sigma}{2}}}\\
	&\lesssim \|n^{(k+1)\lt(\frac{1}{\tilde\gamma}-2\rt)+1}  \|_{\dot{H}^{s-\sigma}}^{\frac{\frac d2 +1-\sigma}{s-\sigma}}\| n^{(k+1)\lt(\frac{1}{\tilde\gamma}-2\rt)+1} \|_{L^2}^{1-\frac{\frac d2 +1-\sigma}{s-\sigma}}\|w\|_{\dot{H}^{s-\frac{k\sigma}{2}}} \|n\|_{\dot{H}^{s-\frac{k\sigma}{2}}}\\
	&\lesssim \|n\|_{L^\infty}^{(k+1)\lt(\frac{1}{\tilde\gamma}-2\rt)} \|n\|_{\dot{H}^{s-\sigma}}^{\frac{\frac d2 +1-\sigma}{s-\sigma}} \|n\|_{L^2}^{1-\frac{\frac d2+1-\sigma}{s-\sigma}}\|w\|_{\dot{H}^{s-\frac{k\sigma}{2}}} \|n\|_{\dot{H}^{s-\frac{k\sigma}{2}}}.
\end{aligned}\]
We also get
 \[
\sfJ_3  + \sfJ_4 \lesssim \|n\|_{L^\infty}^{(k+1)\lt(\frac{1}{\tilde\gamma}-2\rt)} \|w\|_{\dot{H}^{s-\frac{(k+1)\sigma}{2}}}\|\nabla n\|_{L^\infty}\|n\|_{\dot{H}^{s-\frac{(k+1)\sigma}{2}}}
 \] 
thanks to \eqref{tech_1}. 
Finally, we put $\frac{\sfJ_2}{\tilde{\gamma}}$ into \eqref{P_k} and gather all the above estimates to yield the desired relation.
%%%%%%%%%%%%%%%%%%%%%%%%%%%%%%%%%%%%%%%%%%%%%%%%%%%%%%%%%%%%%%%%%%%%%%%%%%%%%%%%%5
%
%
%                       
%
%
%%%%%%%%%%%%%%%%%%%%%%%%%%%%%%%%%%%%%%%%%%%%%%%%%%%%%%%%%%%%%%%%%%%%%%%%%%%%%%%%%
\subsection{Proof of Lemma \ref{higher_it}}\label{app.D}
We see that 
\[
\begin{aligned}
	\frac{1}{2} &\frac{d}{dt} \intr n^{k(\frac{1}{\tilde{\gamma}}-2)} |\Lambda^{s-\frac{k\sigma}{2}} w|^2\, dx \\
	&=\frac{k}{2} \lt(\frac{1}{\tilde{\gamma}}-2\rt) \intr n^{k(\frac{1}{\tilde{\gamma}}-2)-1} \pa_t n |\Lambda^{s-\frac{k\sigma}{2}} w|^2\, dx
	+\intr n^{k(\frac{1}{\tilde{\gamma}}-2)} \Lambda^{s-\frac{k\sigma}{2}} w \Lambda^{s-\frac{k\sigma}{2}}\pa_t w\, dx \\
	& =: \sfK_1+\sfK_2.
\end{aligned}
\]
We first rewrite $\sfK_1$ as 
\[
\begin{aligned}
	\sfK_1
	&= 	-\frac{1}{2} \intr  w \cdot \nabla (n^{k(\frac{1}{\tilde{\gamma}}-2)}) |\Lambda^{s-\frac{k\sigma}{2}} w|^2\, dx 
	-	\frac{1}{2} \intr  v \cdot \nabla  (n^{k(\frac{1}{\tilde{\gamma}}-2)}) |\Lambda^{s-\frac{k\sigma}{2}} w|^2\, dx\\ 
	&\quad -	\frac{k\tilde{\gamma}}{2} \lt(\frac{1}{\tilde{\gamma}}-2\rt) \intr n^{k(\frac{1}{\tilde{\gamma}}-2)}  (\nabla \cdot  w) |\Lambda^{s-\frac{k\sigma}{2}} w|^2\, dx 
	\\
	&\quad-\frac{k\tilde{\gamma}}{2} \lt(\frac{1}{\tilde{\gamma}}-2\rt) \intr n^{k(\frac{1}{\tilde{\gamma}}-2)}  ( \nabla \cdot v) |\Lambda^{s-\frac{k\sigma}{2}} w|^2\, dx \\
	& =: \sum_{i=1}^4 \sfK_{1i}.
\end{aligned}
\]
Here we easily get
\[
\sfK_{13} \lesssim \|n\|_{L^{\infty}}^{k(\frac{1}{\tilde{\gamma}}-2)}  \|\nabla w  \|_{L^{\infty}} \|w\|_{\dot{H}^{s-\frac{k\sigma}{2}}}^2 
\]
and
\[
\sfK_{14} \leq -\frac{k\tilde{\gamma} d}{2(1+t)} \lt(\frac{1}{\tilde{\gamma}}-2\rt) 
\intr n^{k(\frac{1}{\tilde{\gamma}}-2)}  |\Lambda^{s-\frac{k\sigma}{2}} w|^2\, dx +\frac{C}{(1+t)^2} \|n\|_{L^{\infty}}^{k(\frac{1}{\tilde{\gamma}}-2)} \|w\|_{\dot{H}^{s-\frac{k\sigma}{2}}}^2.
\]
Meanwhile, since $s>1+\frac{k \sigma}{2}$, we find
\[
\begin{aligned}
	\sfK_2&=
	-\frac{1}{2}\intr n^{k(\frac{1}{\tilde{\gamma}}-2)} w \cdot \nabla (| \Lambda^{s-\frac{k\sigma}{2}} w|^2) \, dx 
	-\frac{1}{2} \intr n^{k(\frac{1}{\tilde{\gamma}}-2)}   v \cdot\nabla (| \Lambda^{s-\frac{k\sigma}{2}} w|^2)  \, dx \\
	&\quad +\lt(s-\frac{k\sigma}{2}\rt)\intr n^{k(\frac{1}{\tilde{\gamma}}-2)} \Lambda^{s-\frac{k\sigma}{2}} w \sum_{1\leq k\leq d} \pa_k v \cdot \Lambda^{s-\frac{k\sigma}{2}-2} \pa_k \nabla w\, dx \\
	&\quad -\intr n^{k(\frac{1}{\tilde{\gamma}}-2)} \Lambda^{s-\frac{k\sigma}{2}} w \Lambda^{s-\frac{k\sigma}{2}} (w \cdot \nabla v)\, dx 
	+\intr n^{k(\frac{1}{\tilde{\gamma}}-2)} \Lambda^{s-\frac{k\sigma}{2}} w\, ( [w, \Lambda^{s-\frac{k\sigma}{2}}] \cdot\nabla w ) \, dx\\ 
	&\quad +\intr n^{k(\frac{1}{\tilde{\gamma}}-2)} \Lambda^{s-\frac{k\sigma}{2}} w \,\Big([v, \Lambda^{s-\frac{k\sigma}{2}}] \cdot \nabla w- (s-\frac{k\sigma}{2}) \sum_{1\leq k\leq d} \pa_k v \cdot \Lambda^{s-\frac{k\sigma}{2}-2} \pa_k \nabla w  \Big)\, dx\\
	&\quad  - \mathcal{P}_s(k) 
	+\mathcal{R}_s(k)\\
	&=: \sum_{i=1}^6 \sfK_{2i}- \mathcal{P}_s(k) 
	+\mathcal{R}_s(k).
\end{aligned}
\]
For $\sfK_{21}$ and $\sfK_{22}$, it is easy to see
\[
\sfK_{21}
	= -\frac{1}{2}\intr n^{k(\frac{1}{\tilde{\gamma}}-2)} w \cdot \nabla (| \Lambda^{s-\frac{k\sigma}{2}} w|^2) \, dx  
	=\frac{1}{2}\intr n^{k(\frac{1}{\tilde{\gamma}}-2)} (\nabla \cdot w)  | \Lambda^{s-\frac{k\sigma}{2}} w|^2 \, dx -\sfK_{11}
\]
and
\[
\sfK_{22}
	=-\frac{1}{2} \intr n^{k(\frac{1}{\tilde{\gamma}}-2)}   v \cdot\nabla (| \Lambda^{s-\frac{k\sigma}{2}} w|^2)  \, dx 
	= \frac{1}{2} \intr n^{k(\frac{1}{\tilde{\gamma}}-2)}  (\nabla \cdot v) | \Lambda^{s-\frac{k\sigma}{2}} w|^2 \, dx -\sfK_{12}.
\]
Thus, we get
\[
\sfK_{11} + \sfK_{21} \lesssim  \|n\|_{L^{\infty}}^{k(\frac{1}{\tilde{\gamma}}-2)} \|\nabla w\|_{L^{\infty}}  \| w\|_{\dot{H}^{s-\frac{k\sigma}{2}}}^2
\]
and
\[
\sfK_{12} +\sfK_{22} \leq \frac{d}{2(1+t)}\intr n^{k(\frac{1}{\tilde{\gamma}}-2)}   | \Lambda^{s-\frac{k\sigma}{2}} w|^2 \, dx +\frac{C}{(1+t)^2}  \|n\|_{L^{\infty}}^{k(\frac{1}{\tilde{\gamma}}-2)}   \|  w\|_{\dot{H}^{s-\frac{k \sigma}{2}}}^2 .
\]
For $\sfK_{23}$ and $\sfK_{24}$, we estimate
\[
\begin{aligned}
	\sfK_{23} +\sfK_{24}
	&=
	(s-\frac{k\sigma}{2})\intr n^{k(\frac{1}{\tilde{\gamma}}-2)} \Lambda^{s-\frac{k\sigma}{2}} w \sum_{1\leq k\leq d} \pa_k v \cdot \Lambda^{s-\frac{k\sigma}{2}-2} \pa_k \nabla w\, dx \\
	&\quad
	-\intr n^{k(\frac{1}{\tilde{\gamma}}-2)} \Lambda^{s-\frac{k\sigma}{2}} w \Lambda^{s-\frac{k\sigma}{2}} (w \cdot \nabla v)\, dx \\
	&\leq -\frac{s-\frac{k\sigma}{2}+1}{1+t}\intr n^{k(\frac{1}{\tilde{\gamma}}-2)} |\Lambda^{s-\frac{k\sigma}{2}} w |^2\, dx +\frac{C}{(1+t)^2} \|n\|_{L^{\infty}}^{k(\frac{1}{\tilde{\gamma}}-2)} \|w\|_{\dot{H}^{s-\frac{k\sigma}{2}}}^2.
\end{aligned}
\]
For $\sfK_{25}$ and $\sfK_{26}$, we use \eqref{tech_1} and \eqref{tech_2} to obtain
\[
\sfK_{25} \lesssim \|n\|_{L^{\infty}}^{k(\frac{1}{\tilde{\gamma}}-2)} \| w\|_{\dot{H}^{s-\frac{k \sigma}{2}}} \|\nabla w\|_{L^{\infty}} \| w\|_{\dot{H}^{s-\frac{k \sigma}{2}}}
\]
and
\[
\begin{aligned}
	&\sfK_{26}
	\lesssim \|n\|_{L^{\infty}}^{k(\frac{1}{\tilde{\gamma}}-2)} \|w\|_{\dot{H}^{s-\frac{k \sigma}{2} }} (\|v\|_{\dot{H}^{s-\frac{k \sigma}{2} }}\|\nabla w\|_{L^\infty} + \|\nabla ^2 v\|_{L^{\infty}}\|w\|_{\dot{H}^{s-\frac{k \sigma}{2}-1}})\\
	&\quad 	\lesssim \|n\|_{L^{\infty}}^{k(\frac{1}{\tilde{\gamma}}-2)} \|w\|_{\dot{H}^{s-\frac{k \sigma}{2} }} \lt( \frac{\|\nabla w\|_{L^\infty}}{(1+t)^{s-\frac{k\sigma}2+1 -\frac d2}} +\frac{\|w\|_{\dot{H}^{s-\frac{k \sigma}{2}-1}}}{(1+t)^3} \rt).
\end{aligned}
\]
Therefore, we deduce that
\[
\begin{aligned}
	\frac{1}{2} &\frac{d}{dt} \intr n^{k(\frac{1}{\tilde{\gamma}}-2)} |\Lambda^{s-\frac{k\sigma}{2}} w|^2\, dx 
	+\frac{s-\frac{d}{2}+1 +k( \frac{d-\sigma}{2}-\tilde{\gamma}d)}{1+t}  
	\intr n^{k(\frac{1}{\tilde{\gamma}}-2)}  |\Lambda^{s-\frac{k\sigma}{2}} w|^2\, dx \\ 
	& \le C\|n\|_{L^{\infty}}^{k(\frac{1}{\tilde{\gamma}}-2)}  \|\nabla w  \|_{L^{\infty}} \|w\|_{\dot{H}^{s-\frac{k\sigma}{2}}}^2 \\
	&\quad +C\|n\|_{L^{\infty}}^{k(\frac{1}{\tilde{\gamma}}-2)} \|w\|_{\dot{H}^{s-\frac{k \sigma}{2} }} \lt( \frac{\|w\|_{\dot{H}^{s-\frac{k\sigma}{2}}}}{(1+t)^2}+\frac{\|\nabla w\|_{L^\infty}}{(1+t)^{s-\frac{k\sigma}2+1 -\frac d2}} +\frac{\|w\|_{\dot{H}^{s-\frac{k \sigma}{2}-1}}}{(1+t)^3} \rt)\\
	&\quad -\mathcal{P}_s(k) + \mathcal{R}_s(k),
\end{aligned}
\]
which implies the desired result if $\frac{d-\sigma}{2}-\tilde{\gamma} d \ge 0$ (equivalently $\gamma \leq 2-\frac{\sigma}{d}$).

%%%%%%%%%%%%%%%%%%%%%%%%%%%%%%%%%%%%%%%%%%%%%%%%%%%%%%%%%%%%%%%%%%%%%%%%%%%%%%%%%
%
%
%                        thebibliography
%
%
%%%%%%%%%%%%%%%%%%%%%%%%%%%%%%%%%%%%%%%%%%%%%%%%%%%%%%%%%%%%%%%%%%%%%%%%%%%%%%%%%

\end{document}